\documentclass[11pt,a4paper,reqno]{amsart}
\usepackage{amscd,amssymb,epic,eepic,mathptmx,hyperref}

\providecommand{\BBb}[1]{{\mathbb{#1}}}

\providecommand{\cal}[1]{{\mathcal{#1}}}   

\newcommand{\ang}[1]{\langle#1\rangle}
\newcommand{\bignt}[1]{\lfloor#1\rfloor}
\newcommand{\Bcirc}{\overset{\lower 1.5pt%
              \hbox{$@,@,@,@,@,\scriptscriptstyle\circ$}}B{}}
\newcommand{\Binfty}{\overset{\lower 1.5pt%
              \hbox{$@,@,@,@,@,\scriptscriptstyle\infty$}}B{}}
\newcommand{\bigdot}{\mathbin{\raise.65\jot\hbox{$\scriptscriptstyle\bullet$}}}

\newcommand{\C}{{\BBb C}}

\newcommand{\dbar}{d\hspace{1pt}\llap{$\overline{\phantom o}$}}
\newcommand{\Dm}{\BBb D}

\newcommand{\dual}[2]{\langle\,#1,\,#2\,\rangle}
\newcommand{\Dual}[2]{\bigl\langle\,#1,\,#2\,\bigr\rangle}

\newcommand{\erd}{\overset{\lower 1pt\hbox{\large.}}{e}
                  \overset{\lower 1pt\hbox{\large.}}{r}}

\newcommand{\F}{\BBb F}
\newcommand{\Fcirc}{\overset{\lower 1.5pt%
               \hbox{$@,@,@,@,@,\scriptscriptstyle\circ$}}F{}}
\newcommand{\fracc}[2]{{
                \textstyle\frac{#1}{\raise 1pt\hbox{$\scriptstyle #2$}}}}
\newcommand{\fracp}{\fracc1p}

\newcommand{\fracci}[2]{{\frac{#1}{\raise 1pt\hbox{$\scriptscriptstyle #2$}}}}
\newcommand{\fracpi}{\fracci1p}

\newcommand{\im}{\operatorname{i}}

\newcommand{\kd}[1]{\boldsymbol{[\![}{#1}\boldsymbol{]\!]}}

\newcommand{\lap}{\operatorname{\Delta}}

\newcommand{\mlap}{-\!\operatorname{\Delta}}

\newcommand{\Mi}{{\scriptscriptstyle M}}

\newcommand{\norm}[2]{\mathinner{\|}#1\,|#2\|}
\newcommand{\Norm}[2]{\mathinner{\bigl\|\,#1\,\big|#2\bigr\|}}

\newcommand{\op}[1]{\operatorname{#1}}

\newcommand{\N}{\BBb N}

\newcommand{\R}{{\BBb R}}
\newcommand{\Rn}{{\BBb R}^{n}}
\newcommand{\Rp}{\overline{{\BBb R}}_+}
\newcommand{\Rpp}{\overline{{\BBb R}}^2_{++}}
\providecommand{\rom}[1]{\upn{#1}}

{
\end{minipage}%
\par\medskip\noindent}%
\newcounter{enmcount}\renewcommand{\theenmcount}{{\rm\arabic{enmcount}}}

\newcounter{rmcount}\renewcommand{\thermcount}{{\rm\roman{rmcount}}}

\newcounter{Rmcount}\renewcommand{\theRmcount}{{\rm\Roman{Rmcount}}}

\newcommand{\sgo}{s.g.o.\ }
\newcommand{\sgos}{s.g.o.s\ }
\newcommand{\smlnt}[1]{\lceil#1\rceil}

\newcommand{\supp}{\operatorname{supp}}
\newcommand{\Z}{\BBb Z}

\renewcommand{\acute}[1]{\overset{{\scriptscriptstyle /}}{#1}}
\renewcommand{\check}[1]{\overset{{\scriptscriptstyle \vee}}{#1}}
\renewcommand{\hat}[1]{\overset{{\scriptscriptstyle \wedge}}{#1}}

\numberwithin{equation}{section}

\newtheorem{thm}{Theorem}
\numberwithin{thm}{section}
\newtheorem{prop}[thm]{Proposition}
\newtheorem{lem}[thm]{Lemma}
\newtheorem{cor}[thm]{Corollary}

\theoremstyle{definition}
\newtheorem{defn}[thm]{Definition}

\newtheorem{exmp}[thm]{Example}
 \numberwithin{exercise}{section}
 
\theoremstyle{remark}
\newtheorem{rem}[thm]{Remark}

\title[Elliptic boundary problems and the Boutet de Monvel\dots]{Elliptic boundary problems\\ 
 and the Boutet de Monvel calculus\\
 in Besov and Triebel--Lizorkin spaces}
\author{Jon Johnsen}
\address{Mathematical Institute, University of Copenhagen, DK-2100
Copenhagen~{\O}, Denmark}
\email{jjohnsen@math.ku.dk}
\thanks{Partially supported by the Danish Natural Sciences Research
Council, no.~11--1221--1.
\\[9\jot]{\tt Appeared in Mathematica Scandinavica, {\bf 79} (1996), 25--85.}}

\begin{document}
\maketitle

\begin{abstract}{The Boutet de Monvel calculus of pseudo-differential
boundary operators is generalised to the scales of Besov and
Triebel--Lizorkin spaces, $B^{s}_{p,q}$ and $F^{s}_{p,q}$, with
$s\in\R$ and $p$ and $q\in\,]0,\infty]$ (though with $p<\infty$ for
the $F^{s}_{p,q}$ spaces).
 
The continuity and Fredholm properties proved here 
extend those in \cite{F2} and \cite{G3}, and the results on range
complements of surjectively elliptic Green operators improve the
earlier known, even for the classical spaces with $1<p<\infty$.

The symbol classes treated are the $x$-uniformly estimated ones. On
$\Rp^n$ a trace operator $T$ and a singular Green operator $G$, both of class
0, are defined in general to be 
 \begin{equation}
 T=K^*e^+,\qquad G=r^+G_1^*e^+
 \end{equation}
where the Poisson operator $K$ is $\op{OPK}(e^{iD_{x'}\cdot
D_{\xi'}}\overline{\tilde t}(x',x_n,\xi'))$ and the singular Green operator
$G_1$ equals $\op{OPG}(e^{iD_{x'}\cdot D_{\xi'}}\overline{\tilde g}
(x',y_n,x_n,\xi'))$, respectively. 
}
\end{abstract}

\section{Summary}  \label{summ-sect}
\thispagestyle{empty}

As a main example of the considerations in this article one may take
an elliptic differential operator $A=\sum_{|\alpha|\le
d}a_\alpha(x)D^\alpha$ on an open bounded set $\Omega\subset\Rn$ with
$C^\infty$ boundary $\Gamma:=\partial\Omega$ and a trace operator $T$
for which
 \begin{equation}
 \begin{aligned}
 Au(x)&=f(x)\quad\text{ in}\quad\Omega, \\
 Tu(x)&=\varphi(x)\quad\text{ on}\quad \Gamma,
 \end{aligned}
 \label{i1}
 \end{equation}
is a boundary value problem that is elliptic in the sense of Agmon,
Douglis and Nirenberg \cite{ADN1}.

The topics to be discussed in this connection are:
 \begin{itemize} \addtolength{\itemsep}{\jot}
  \item[(I)] solvability and regularity results in the
{\em Besov\/} and {\em Triebel--Lizorkin\/} spaces $B^{s}_{p,q}$ and
$F^{s}_{p,q}$ with $s\in\R$ and $p$ and $q\in\,]0,\infty]$ (and $p<\infty$
in the $F$ case),
  \item[(II)] a generalisation of the pseudo-differential calculus 
of Boutet de~Monvel \cite{BM71} (for problems like that in 
\eqref{i1}, e.g.) to the setting in (I).
 \end{itemize}
In fact (I) and (II) are treated simultaneously in order to give a
{\em unified\/} treatment of these two systematic points of view.

It is also the purpose to present modifications of some basic facts in
the calculus with full proofs. The arguments should be of interest,
partly because they are fairly elementary, and partly 
since they bring the $\Rn_+$-part of the calculus closer to the
pseudo-differential theory of H\"ormander \cite[Sect.~18.1]{H}.

The results are presented in the rest of this section, and
Section~\ref{prel-sect} settles the notation and the prerequisites on
the $B^{s}_{p,q}$ and $F^{s}_{p,q}$ spaces. Section~\ref{Rnp-sect}
describes the operator classes of the calculus relatively to the
half-space $\Rn_+$, whereas the continuity properties in the
$\Rn_+$-case are proved in
Section~\ref{cont-sect}. Fredholm properties are treated in
Section~\ref{grn-sect}, even for multi-order systems acting in vector
bundles, and applications are indicated in Section~\ref{appl-sect}.

\subsection*{On the spaces} The project in (I) covers many different spaces
at the same time, since:
 \begin{itemize}\addtolength{\itemsep}{2\jot}
                \addtolength{\topsep}{2\jot}  
  \item $C^s=B^{s}_{\infty,\infty}$ for $s>0$ (H\"older--Zygmund spaces),
  \item $W^s_p=B^{s}_{p,p}$ for $s\in\R_+\!\setminus\N_0$ and
$1<p<\infty$ (Slobodetski{\u\i} spaces),
  \item $W^k_p=F^{k}_{p,2}$ for $k\in\N_0$ and $1<p<\infty$ (Sobolev
spaces),
  \item $H^s_p=F^{s}_{p,2}$ for $s\in\R$ and $1<p<\infty$ (Bessel
potential spaces).
 \end{itemize}
In particular $F^{0}_{p,2}=L_p$ for $1<p<\infty$ and
$B^{s}_{2,2}=F^{s}_{2,2}=H^s$ for $s\in\R$ (Lebesgue and Sobolev
spaces).

For these relations the reader is referred to the books of H.~Triebel
\cite{T2,T3}. As a further motivation, note that the local Hardy
space $h_p$ equals $F^{0}_{p,2}$ for $0<p<\infty$, and that the relation to
Morrey--Campanato spaces is explained in \cite{T3}, together with the fact 
that also $F^{s}_{p,\infty}$ has been considered earlier on in \cite{Chr,DeS}.
(The local BMO space $bmo$ equals $F^{0}_{\infty,2}$, which is not treated 
here. However, a recent work of J.~Marschall \cite{Mar95} may provide a point of
departure for an 
extension of the present results to the $F^{s}_{\infty,q}$ spaces.)

For a presentation of the results the reader is referred to
Theorems~\ref{i1-thm} and \ref{i2-thm} and to Corollaries~\ref{i1-cor} and
\ref{i2-cor} below. 

The calculus of Boutet de~Monvel was originally worked out for the
$H^s$ Sobolev spaces \cite{BM71}, and an extension to $B^{s}_{p,q}$ and
$F^{s}_{p,2}$ with $1<p<\infty$ and $1\le q\le\infty$  was given by G.~Grubb
\cite{G3}. Among the earlier attempts at an $L_p$ 
theory for the calculus the shortcomings of \cite{RS} are accounted
for in \cite[Rem.~3.2]{G3}, so the reader may refer to the details there.

An extension of the calculus to the $B^{s}_{p,q}$ and $F^{s}_{p,q}$ 
scales, with the restriction $p<\infty$ in the $F$~case, has been worked 
out already by J.~Franke. However, at the central point 
the arguments are not contained in his thesis \cite{F2}, and 
the only published material on this work is the review article
\cite{F1}, which does not contain proofs. In addition to this, the 
concept of negative class for operators of the form $P_\Omega+G$\,---\,and
hence for general Green operators $\cal A$\,---\,was 
first introduced in \cite{G3}, and already
for classical problems like \eqref{i1} this notion is indispensable 
for an optimal description of parametrices, cf.~\cite[Thm.~5.4~ff.]{G3}.

On these grounds the author has included in his thesis
\cite{JJ93} an extension to the $B^{s}_{p,q}$ and $F^{s}_{p,q}$ spaces.
It is presented here with some improvements.  

It should also be mentioned that 
the $L_p$ theory of {\em differential\/} boundary problems, studied in
\cite{ADN1,ADN2}, \cite{So2} \dots, was considered first in 
the $B^{s}_{p,q}$ and $F^{s}_{p,q}$ scales by H.~Triebel
in \cite{Tri78}, cf.~also \cite{T2}, although with some 
restrictions for $p$ and $q<1$. An extension to the case
with $p$ and $q$ also in the whole of $\,]0,1]$ has been worked out by
Franke and T.~Runst \cite{FR95}.

\subsection*{On the calculus} Boundary problems like \eqref{i1} 
are represented here by
Green operators in the pseudo-differential boundary operator calculus
of L.~Boutet de~Monvel, cf.~(II). A typical example is obtained from a
matrix operator       
 \begin{equation}
 \cal A=\begin{pmatrix}P_\Omega+G&K\\ T& S\end{pmatrix}\colon
 \smash[b]{
  \begin{array}{ccc}
  C^\infty(\overline{\Omega})^N\\ \oplus\\ C^\infty(\Gamma)^M
  \end{array} 
 \to
  \begin{array}{ccc}
  C^\infty(\overline{\Omega})^{N'}\\ \oplus\\ C^\infty(\Gamma)^{M'}
  \end{array}
 }  
 \label{i2} 
 \end{equation}
where 
 \begin{itemize} \addtolength{\itemsep}{\jot}
  \item $P_\Omega=r_\Omega P e_\Omega$ is the truncation to $\Omega$
of a pseudo-differential operator on $\Rn$ with the symbol lying in 
$S^{d}_{1,0}$ and satisfying a transmission condition at
$\Gamma\subset\Rn$;
  \item $G$ is a singular Green operator,
  \item $K$ is a Poisson (or potential) operator,
  \item $T$ is a trace operator,
  \item $S$ is a pseudo-differential operator on $\Gamma$.
 \end{itemize}
(See Section~\ref{Rnp-sect} below for the expressions of these operators in
local coordinates.) Then $\cal A$ is said to have {\em order\/} $d$ and
{\em class\/} $r$, for numbers $d\in\R$ and $r\in\Z$, if all entries
have this order and both $T$ and $P_\Omega+G$ have class $r$.

As an example, if in \eqref{i1} $P_\Omega=A=\lap^2$ 
(the biharmonic operator) and $Tu=(\gamma_0 u,\gamma_1 u)$ with
$\gamma_0 u=u|_{\Gamma}$ and $\gamma_1 u=-i\frac{\partial u}{\partial
\vec n}|_{\Gamma}$, $\vec n$ being the outward unit normal vector field
at $\Gamma$, then $N=1=N'$, $M=0$ and $M'=2$ together with $G=0$, $K=0$
and $S=0$ allows one to read \eqref{i1} as an equation for the operator
$\cal A$. (This is actually a slightly more general situation with
multi-order and multi-class.)

The five types of operators listed above are also defined on
more general distributions than $C^\infty$ functions. In this
respect the following theorem is proved here:

\begin{thm} \label{i1-thm}
Let $s\in\R$ and $p$ and $q\in\,]0,\infty]$, and suppose that $\cal A$ has
order $d\in\R$ and class $r\in\Z$. 

If $s>r+\max(\fracp-1,\fracc np-n)$, then $\cal A$ has the continuity
properties 
 \begin{align}
 \cal A&\colon
  \begin{array}{ccc}
  B^{s}_{p,q}(\overline{\Omega})^N\\ \oplus\\ B^{s-\fracpi}_{p,q}(\Gamma)^M
  \end{array} 
 \to
  \begin{array}{ccc}
  B^{s-d}_{p,q}(\overline{\Omega})^{N'}\\ \oplus
         \\ B^{s-d-\fracpi}_{p,q}(\Gamma)^{M'}
  \end{array}
 \label{i5}
 \\[3\jot] 
 \cal A&\colon
  \begin{array}{ccc}
  F^{s}_{p,q}(\overline{\Omega})^N\\ \oplus\\ B^{s-\fracpi}_{p,p}(\Gamma)^M
  \end{array} 
 \to
  \begin{array}{ccc}
  F^{s-d}_{p,q}(\overline{\Omega})^{N'}\\ \oplus
         \\ B^{s-d-\fracpi}_{p,p}(\Gamma)^{M'}
  \end{array}
 \label{i6}
 \end{align}
provided $p<\infty$ in \eqref{i6}.

Furthermore, if $\cal A$ for some $s_1<r+\max(\fracc1{p_1}-1,\fracc
n{p_1}-n)$ is continuous from either 
$B^{s_1}_{p_1,q}(\overline{\Omega})^N\oplus 
B^{s_1-\fracci1{p_1}}_{p_1,q_1}(\Gamma)^M$
or $F^{s_1}_{p_1,q_1}(\overline{\Omega})^N\oplus 
B^{s_1-\fracpi}_{p_1,p_1}(\Gamma)^M$ 
to the space $\cal D'(\Omega)^{N'}\times \cal
D'(\Gamma)^{M'}$ , then both $T$ and $P_\Omega+G$ have {\em class\/} $\le r-1$.
\end{thm}

The theorem also gives statements for each kind of the operators
$P_\Omega$, $G$, $K$, $T$ and $S$ by consideration of examples of $\cal A$ with
suitable zero-entries. Thus the
`working definition' of the class concept\,---\,namely that an
operator is of class $r$ if and only if it is continuous from
$H^r(\overline{\Omega})$ to $\cal D'$\,---\,is generalised to
the $B^{s}_{p,q}$ and $F^{s}_{p,q}$ setting. 

When $\cal A$ is elliptic (and in particular has polyhomogeneous
symbols of order $d\in\Z$), the theorem applies equally well to its parametrix
$\widetilde{\cal A}$, which may be taken of order $-d$ and class $r-d$ 
(cf.~\cite{G3}).
Then $\widetilde{\cal A}$ is continuous from the right to the left
in \eqref{i5} and \eqref{i6} for every $s>r+\max(\fracp-1,\fracc np-n)$.

Consequently $\cal A$ has the expected {\em inverse regularity\/}
properties:

\begin{cor} \label{i1-cor}
Let $\cal A$ be elliptic of order $d$ and class $r$. Let $(u,\psi)$ 
belong to $B^{s}_{p,q}(\overline{\Omega})^N\oplus
B^{s-\fracpi}_{p,q}(\Gamma)^M$ for some $s>r+\max(\fracp-1,\fracc np-n)$
and assume that $(u,\psi)$\,---\,for a parameter with
$s_1>r+\max(\fracc1{p_1}-1,\fracc n{p_1}-n)$\,---\,satisfies
 \begin{equation}
 \begin{pmatrix}P_\Omega+G&K\\ T& S\end{pmatrix}
 \begin{pmatrix}u\\ \psi\end{pmatrix}= \begin{pmatrix}f\\ \phi\end{pmatrix}
 \in
 \begin{array}{ccc}
  B^{s_1-d}_{p_1,q_1}(\overline{\Omega})^{N'}\\ \oplus\\
       B^{s_1-d-1/{p_1}}_{p_1,q_1}(\Gamma)^{M'} 
  \end{array} .
 \label{i8} 
 \end{equation}
Then $(u,\varphi)$ is also an element of $B^{s_1}_{p_1,q_1}
(\overline{\Omega})^N\oplus B^{s_1-\fracci1{p_1}}_{p_1,q_1}(\Gamma)^M$. 

In the $F^{s}_{p,q}$ spaces $\cal A$ has analogous inverse 
regularity properties (if $q=p$ in the spaces over $\Gamma$), and the
statements likewise carry over to the mixed cases with $(u,\psi)$
given in Besov spaces and $(f,\varphi)$ prescribed in
Triebel--Lizorkin spaces, or vice versa.
\end{cor}
 
In the elliptic case, the Fredholm properties of $\cal A$ are improved
and extended to the following Theorem~\ref{i2-thm}. Thus an elliptic 
Green operator is a Fredholm operator with $(s,p,q)$-independent kernel and 
index ($1^\circ$ and $2^\circ$), for which each choice of a smooth 
range-complement (by $3^\circ$) can be used as such for
every admissible $(s,p,q)$.

\begin{thm} \label{i2-thm}
Let $\cal A$ be elliptic of order $d\in\Z$ and class $r\in\Z$.
Consider for each $p$ and $q\in\,]0,\infty]$ and 
$s>r+\max(\fracp-1,\fracc np-n)$ the two operators in formulae
\eqref{i5} and \eqref{i6} above.

$1^\circ$ For each such $(s,p,q)$ the operators in \eqref{i5} and
\eqref{i6} have the same kernel, $\ker\cal A$. Here $\ker\cal A$ is a
finite dimensional subspace of $C^\infty(\overline{\Omega})^N\oplus
C^\infty(\Gamma)^M$, which is independent of $(s,p,q)$.

$2^\circ$ For each $(s,p,q)$ the operators have closed ranges.
Moreover, there exists a finite dimensional subspace $\cal N\subset
C^\infty(\overline{\Omega})^{N'}\oplus C^\infty(\Gamma)^{M'}$ which
for each $(s,p,q)$ is a range-complement of both operators. That is, 
 \begin{gather}
 \cal N\oplus \cal A\big(B^{s}_{p,q}(\overline{\Omega})^N\oplus 
        B^{s-\fracpi}_{p,q}(\Gamma)^M\big)
 =B^{s-d}_{p,q}(\overline{\Omega})^{N'}\oplus 
        B^{s-d-\fracpi}_{p,q}(\Gamma)^{M'}
 \label{i11} \\
 \cal N\oplus \cal A\big(F^{s}_{p,q}(\overline{\Omega})^N\oplus 
        B^{s-\fracpi}_{p,p}(\Gamma)^M\big)
 =F^{s-d}_{p,q}(\overline{\Omega})^{N'}\oplus 
        B^{s-d-\fracpi}_{p,p}(\Gamma)^{M'}
 \label{i12}
 \end{gather}
whenever $s>r+\max(\fracp-1,\fracc np-n)$.

$3^\circ$ If $\cal N\subset C^\infty(\overline{\Omega})^{N'}\oplus 
C^\infty(\Gamma)^{M'}$ is any subspace such that either \eqref{i11} or
\eqref{i12} holds for some parameter $(s_1,p_1,q_1)$ with
$s_1>r+\max(\fracc n{p_1}-1,\fracc n{p_1}-n)$, then $\cal N$ has
finite dimension and both \eqref{i11} and \eqref{i12} hold for every
$(s,p,q)$ that satisfies $s>r+\max(\fracp-1,\fracc np-n)$.
\end{thm}

In the determination of specific examples of $\cal N$ the following result
concerning annihilation should be of importance.

For a given subspace $\cal N\subset
C^\infty(\overline{\Omega})^{N'}\oplus C^\infty(\Gamma)^{M'}$ it is
convenient to let $\cal N^\perp$ denote the distributions $f$ and
$\varphi$ for which $\dual{f}{g}_\Omega+\dual{\varphi}{\eta}_\Gamma$
makes sense and equals $0$ for all $(g,\eta)$ in $\cal N$. These
questions are meaningful for $\cal A$'s codomains provided each
element $(g,\eta)$ has sufficiently many vanishing traces $\gamma_j
g$:

\begin{cor} \label{i2-cor}
Let $\cal A$ be as in Theorem~\ref{i2-thm}, and let
$\cal N\subset C^\infty(\overline{\Omega})^{N'}\oplus
C^\infty(\Gamma)^{M'}$ be a subspace for which each element $(g,\eta)$
satisfies $\gamma_j g=0$ for $j<d-r$ (void if $d\le r$).
Moreover, let one of the identities
 \begin{gather}
 \cal N^\perp\cap\big(B^{s-d}_{p,q}(\overline{\Omega})^{N'}\oplus
                  B^{s-d-\fracpi}_{p,q}(\Gamma)^{M'}\big)
 =\cal A(B^{s}_{p,q}(\overline{\Omega})^{N}\oplus
        B^{s-\fracpi}_{p,q}(\Gamma)^{M})
 \label{i16} \\
 \cal N^\perp\cap\big(F^{s-d}_{p,q}(\overline{\Omega})^{N'}\oplus
                  B^{s-d-\fracpi}_{p,p}(\Gamma)^{M'}\big)
 =\cal A(F^{s}_{p,q}(\overline{\Omega})^{N}\oplus
        B^{s-\fracpi}_{p,p}(\Gamma)^{M}) 
 \label{i17}
 \end{gather}
hold for a parameter $(s_1,p_1,q_1)$ with $s_1>r+\max(\fracc1{p_1}-1,
\fracc n{p_1}-n)$.

Then $\cal N$ is a range complement and both
\eqref{i16} and \eqref{i17} hold for all $(s,p,q)$ with
$s>r+\max(\fracp-1,\fracc np-n)$.
\end{cor}

\begin{exmp} \label{i-exmp}
$\cal A=\left(\begin{smallmatrix}
\mlap\\ \gamma_1 \end{smallmatrix}\right)$ represents the Neumann
problem for the Laplacian, and for $(s,p,q)=(2,2,2)$ the data 
$(f,\varphi)$ belong to the range of $\cal A$
precisely when $\int_\Omega f+i\int_\Gamma\varphi\,=0$. By
Corollary~\ref{i2-cor}, this annihilation of $(1_\Omega,i1_\Gamma)$
characterises the range of $\cal A$ for every $(s,p,q)$ with
$s>2+\max(\fracp-1,\fracc np-n)$.
\end{exmp}

In the preceding exposition the scope has been
restricted somewhat for simplicity's sake. In
fact $\cal A$ could equally well have been a multi-order and
multi-class system in the Douglis--Nirenberg sense, even with
each entry matrixformed and with each of the matrix entries acting between 
sections of vector bundles over $\Omega$ and $\Gamma$\,---\,i.e.~$P_\Omega=
(P_{i,j,\Omega})$ where $P_{i,j,\Omega}\colon C^\infty(E_i)\to
C^\infty(E'_j)$ etc.\,---\,as in \cite[Cor.~5.5]{G3}. 
Theorems~\ref{i1-thm} and \ref{i2-thm} are both proved in this generality 
below, the latter even in a version for one-sided elliptic operators.

\subsection*{The methods}
To carry out the analysis in the $B^{s}_{p,q}$ and $F^{s}_{p,q}$
spaces the definitions and results based on Fourier analysis, as
presented in \cite{T2}, are adopted. Together with M.~Yamazaki's theorems
on convergence of series of distributions satisfying spectral
conditions, \cite[Thm.s~3.6 and 3.7]{Y1}, these are the main tools
used here concerning the function spaces. 

For the treatment of the five types of operators in the $B^{s}_{p,q}$
and $F^{s}_{p,q}$ spaces it is used that a pseudo-differential
operator $P$ on $\Rn$ is known to be bounded 
 \begin{equation}
 P\colon B^{s}_{p,q}(\Rn)\to B^{s-d}_{p,q}(\Rn),\qquad
 P\colon F^{s}_{p,q}(\Rn)\to F^{s-d}_{p,q}(\Rn)
 \label{i14}
 \end{equation}
for $s\in\R$ and $p$ and $q\in\,]0,\infty]$ (with $p$ finite in the
$F$ case) whenever the symbol belongs to the H\"ormander class
$S^{d}_{1,0}(\Rn\times\Rn)$; that is, whenever the $x$-uniform
estimate\,---\,with $\ang{\xi}:=(1+|\xi|^2)^{\frac12}$\,---
 \begin{equation}
 \sup\bigl\{\,\ang{\xi}^{-(d-|\alpha|)}
                |D^\beta_x D^\alpha_\xi p(x,\xi)|\bigm| x,\xi\in\Rn\,\bigr\} 
 =:C_{\alpha\beta}<\infty
 \label{i15}
 \end{equation}
is valid for $p(x,\xi)\in\cal E(\R^{2n})$ for all multi-indices 
$\alpha$ and $\beta$.

The central result in \eqref{i14} was obtained for $p<1$ (even for more
general symbols) by Bui~Huy~Qui and L.~P\"aiv\"arinta
\cite{Bui,Pai}, and it has been reproved (with further
generalisations) in M.~Yamazaki's paper \cite{Y1}, e.g., where also
the history of this $L_p$ theory is outlined.

In order to carry over the continuity in \eqref{i14} to a version of
Theorem~\ref{i1-thm} for the halfspace $\Rn_+$, the $x$-uniformly estimated
symbols and symbol-kernels of \cite{G3} and  \cite{GK} are treated
here. Cf.~\eqref{-24} ff.\ below for these classes.

For the $P_\Omega$'s in particular, the so-called {\em uniform
two-sided\/} transmission condition at $\Gamma$ is required to hold
for $P$. In local coordinates this amounts to the fulfilment of
\eqref{-2} below.

By and large, the $\Rn_+$-version of Theorem~\ref{i1-thm} is deduced
from \eqref{i14} by a method of
attack that is rather close to the one adopted in \cite{G3}, and hence
in a way that is quite standard within the calculus. However, to
include the interval $\,]0,1[\,$ for the integral-exponents $p$, a fresh
approach is needed. For this point J.~Franke \cite{F2} has given an
argument based on estimates of para-multiplication operators (like
those in \cite{Y1}, for example) and on
complex interpolation of the Besov and Triebel--Lizorkin spaces.

In addition to this there is a main technical difficulty in the fact
that denseness of the Schwartz space $\cal S(\Rn)$ in either of 
$B^{s}_{p,q}(\Rn)$ and $F^{s}_{p,q}(\Rn)$ holds precisely when
both $p<\infty$ and $q<\infty$ do so. The approach taken here is to
define trace operators $T=\op{OPT}(\tilde t)$ and singular Green
operators $G=\op{OPG}(\tilde g)$ of class $0$ on $\Rn_+$ by the
formulae
 \begin{equation}
 Tu=K^*e^+u,\qquad Gu=r^+G_1^*e^+u
 \label{i20}
 \end{equation} 
when $e^+$ makes sense on $u\in\cal S'(\Rp^n)$. Hereby
$K=\op{OPK}(e^{iD_{x'}\cdot D_{\xi'}}\overline{\tilde t}(x',x_n,\xi'))$
and $G_1=\op{OPG}(e^{iD_{x'}\cdot D_{\xi'}}\overline{\tilde g}
(x',y_n,x_n,\xi'))$, that have their adjoints $K^*$ and $G_1^*$
defined on $\cal S'_0(\Rp^n)$.

\subsection*{Comparison with other works} The continuity properties shown 
here extend those in \cite[Thm.~3.11 ff.]{G3}, mainly to the case with
$p\in\,]0,\infty]$. (The special 
results on $B^{s}_{p,p}\cap H^s_p$ and $B^{s}_{p,p}\cup H^s_p$ there
are recovered by use of the full statements below, see
Theorem~\ref{sgo-thm} ff.)
The results of \cite{F2} are extended to multi-order systems, that can
have class $r<0$, in which case the ranges of $s$ are larger than his.
Using techniques from \cite{G3}, the present restrictions on $s$ are
proved to be essentially sharp. For the subscales $B^{s}_{p,p}$ and 
$F^{s}_{p,2}$ the borderline cases $s=r+\fracp-1$ were first analysed
in \cite{G3}, but here in the more complicated situation with
$q\in\,]0,\infty]$ this question is only given a rudimentary
treatment; cf.~Remark~\ref{trac-rem} below. In
Section~\ref{cont-sect}, continuity from $\cal S'(\Rp^n)$ is shown to
hold precisely for operators of class~$-\infty$.

The result on $\ker\cal A$ extends the one in \cite{G3} to the full
scales $B^{s}_{p,q}$ and $F^{s}_{p,q}$, whereas $1^\circ$ and
$2^\circ$ in Theorem~\ref{i2-thm} generalise \cite{F2} to the $\cal
A$'s considered here. The exact range characterisations
of surjectively elliptic operators in \cite[Thm.~5.4]{G3}  amount
to  annihilation of a specific $(s,p,q)$-independent 
finite dimensional $C^\infty$ space. Extending this,
Corollary~\ref{i2-cor} shows that a smooth 
space $\cal N$ need only be annihilated by the range for a {\em single\/}
$(s,p,q)$, for then it is so for all admissible parameters;
cf.~also Example~\ref{i-exmp} above. For the existence of such an
$\cal N$ having vanishing traces in case $r\le d$, see
\cite[Rem.~5.3]{G3}. Moreover, a {\em smooth\/} space $\cal N$
complements the range either for all $(s,p,q)$ or for none by
Theorem~\ref{i2-thm}. Even for differential problems with $1<p<\infty$
this conclusion has seemingly not been formulated before.

After submission of the first version of this paper, 
I became aware of a work of D.--C.~Chang, S.~G.~Krantz and
E.~M.~Stein \cite{ChaKraSte93}, which also deals with boundary problems in
spaces with $p<1$. They consider the solution operators 
$R_D$ and $R_N$ for the boundary homogeneous Dirichl\'et and Neumann 
problems, respectively, for $\mlap$, and they show that 
$\partial^2_{jk}R_D$ and $\partial^2_{jk}R_N$ are bounded from the 
`minimal' local Hardy space $r_\Omega F^{0}_{p,2;0}(\overline{\Omega})$ 
for every $p>0$, whilst for $F^{0}_{p,2}(\overline{\Omega})$
this holds for $\partial^2_{jk}R_D$ if and only if $p>\frac{n}{n+1}$
and for $\partial^2_{jk}R_N$ if and only if $p>1$.

On one hand, by application of Theorem~\ref{i1} or \ref{grn-thm}
to $\partial^2_{jk}R_D$ and $\partial^2_{jk}R_N$ as special cases,
the present general theory also yields the boundedness on
$F^{0}_{p,2}(\overline{\Omega})$ for $p>\frac{n}{n+1}$ and $p>1$ 
as well as the unboundedness for $p<\frac{n}{n+1}$ and $p<1$ 
(since the operators are known to have class $-1$ and $0$, respectively). 
On the other hand, however, scales like 
$r_\Omega F^{s}_{p,q;0}(\overline{\Omega})$ are not considered here.

\bigskip

Perhaps the spaces with $p$ and $q\in\,]0,1[\,$ deserve some extra
attention in view of the fact that $B^{s}_{p,q}$ and $F^{s}_{p,q}$
with such exponents are merely quasi-Banach spaces. For continuity
questions it is well known that \eqref{1.27'} below
can be applied with succes instead of the quasi-triangle inequality;
cf. also Remark~\ref{tvs-rem} below. So for Theorem~\ref{i1-thm} the
essential difficulties lie in the case with $p=\infty$, which is
handled by means of \eqref{i20}.

To prove Theorem~\ref{i2-thm} it might seem necessary to generalise
the notion of Fredholm operators to quasi-Banach spaces (as in
\cite{F2} and \cite{FR95}). However, this approach is neither necessary
nor particularly useful here. In fact
the restriction to $s>r+\max(\fracp-1,\fracc np-n)$ when $T$ and
$P_\Omega+G$ are of class $r$ allows embeddings into spaces with $p$
and $q\in\,]1,\infty[\,$ on which the operators are defined. This
gives a way to deduce the various properties from the Banach space
cases. Cf.~also Remark~\ref{clas-rem}. 

In reality the `extra' spaces over $\Omega$ with $p$ and
$q\in\,]0,1[\,\cup\{\infty\}$ treated here do not in comparison provide
any `new' functions to which a given operator can be applied,
cf.~Remark~\ref{clas-rem}.

From this point of view the achievement in the present article consists
rather of continuity with respect to new topologies and of more
detailed Fredholm properties.

\subsubsection*{Acknowledgement} During the work I have received much
encouragement and support from my advisor G.~Grubb, and I
shall always be grateful for this.

\section{Preliminaries} \label{prel-sect} 

In this section an overview of the Besov spaces $B^{s}_{p,q}$ and
Triebel--Lizorkin spaces $F^{s}_{p,q}$ is given.
Subsections~\ref{Ythm-ssect}, \ref{tensor-ssect}, \ref{ext-ssect} and
\ref{intp-ssect} are vital for the treatment of spaces with integral-exponent
$p<1$, in particular because they provide a substitute for the duality
arguments in \cite{G3}, that only work for $p>1$. 

\subsection{Notation} \label{notation-ssect}
For a normed or quasi-normed space $X$, $\norm xX$ denotes the 
quasi-norm of the vector $x$. Recall that 
$X$ is quasi-normed when the triangle inequality is weakened to
$\norm{x+y}X\le c(\norm xX+\norm yX)$ for some 
$c\ge1$ independent of $x$ and $y$. (The prefix `quasi-' is omitted
when confusion is unlikely to occur.)

As simple examples there are $L_p(\Rn)$ and $\ell_p:=\ell_p(\N_0)$ for 
$p\in\,]0,\infty]$, where $c=2^{\fracpi-1}$ is possible for $p<1$.
However, it is a stronger fact that 
 \begin{equation}
 \norm{f+g}{L_p}\le(\norm{f}{L_p}^p+\norm {g}{L_p}^p)^{\fracpi},
 \quad\text{ for}\quad0<p\le1,
 \label{1.27}
 \end{equation}
and this inequality has an exact analogue for the $\ell_p$ spaces.

The vector space of bounded linear operators from $X$ to $Y$ is
denoted $\Bbb L(X,Y)$; the operator quasi-norm $\norm{\cdot}{\Bbb
L(X,Y)}$ satisfies the quasi-triangle inequality with the same
constant as $\norm{\cdot}{Y}$.

The space of compactly supported smooth functions is written 
$C^\infty_0(\Omega)$ or $\cal D(\Omega)$ when $\Omega\subset\Rn$ is open,
and $\cal D'(\Omega)$ is the dual space of distributions on $\Omega$.
The duality between $u\in\cal D'(\Omega)$
and $\varphi\in C^\infty_0(\Omega)$ is denoted $\ang{u,\varphi}$.

The Schwartz space of rapidly decreasing functions is denoted by 
$\cal S=\cal S(\Rn)$, and the dual space of tempered distributions
by $\cal S'=\cal S'(\Rn)$. The seminorms on $\cal S(\Rn)$ are taken to be
$\norm{\psi}{\cal S,\alpha,\beta}=
\sup\bigl\{\,|x^\alpha D^\beta\psi| \bigm| x\in\Rn\,\bigr\}$ for
$\alpha,\beta\in\N_0^n$ or equivalently $\norm{\psi}{\cal
S,N}=\max\bigl\{\,\norm{\psi}{\cal S,\alpha,\beta}\bigm|
|\alpha|,|\beta|\le N\,\bigr\}$ for $N\in\N_0$.

Throughout $D^\alpha=(-i)^{|\alpha|}\frac{\partial^{\alpha_1}}{\partial 
x_1^{\alpha_1}}\dots\frac{\partial^{\alpha_n}}{\partial 
x_1^{\alpha_n}}$, where $|\alpha|=\alpha_1+\dots+\alpha_n$ for
$\alpha\in\N_0^n$. 

With the norm $\norm{f}{C(\Rn)}=\sup|f|$, it is convenient to let
 \begin{equation}
 C(\Rn)= \bigl\{\,f\in L_\infty(\Rn)\bigm| 
                  \text{$f$ is uniformly continuous}\,\bigr\}. 
 \end{equation}
Moreover,
$C^k_b(\Rn)=\{\,f\mid D^\alpha f\in C(\Rn),\ |\alpha|\le k\,\}$ and
$C^\infty(\Rn)=\cap C^k_b(\Rn)$ is the space of smooth functions with
bounded derivatives of any order; it is equipped with the semi-norms
$\sup\bigl\{\,|D^\alpha f(x)|\bigm| x\in\Rn,\ |\alpha|\le k\,\bigr\}$. 
(This distinguishes the space $C^k_b$ from the H\"older--Zygmund space
$B^{s}_{\infty,\infty}=C^s$, $s>0$. Then $C^k\subset
C^{k-1}_b$ and $C^\infty=\cap B^{s}_{\infty,\infty}$.)

For the Besov spaces $B^{s}_{p,q}(\Rn)$, where $s\in\R$ and
$0<p,q\le\infty$, and the Triebel--Lizorkin spaces $F^{s}_{p,q}$
(considered for $0<p<\infty$ only) the notation of H.~Triebel in
\cite{T2} is adopted.

When $\Omega\subset\Rn$ is open, 
$C^\infty(\overline{\Omega})$, $B^{s}_{p,q}(\overline{\Omega})$ and
$F^{s}_{p,q}(\overline{\Omega})$ etc.\ are defined by
restriction to $\Omega$. E.g., $C(\overline{\Omega})=r_\Omega C(\Rn)$
where $r_\Omega\colon \cal D'(\Rn)\to\cal D'(\Omega)$ is the
transpose of the extension by 0 outside of $\Omega$, denoted
$e_\Omega\colon C^\infty_0(\Omega)\to C^\infty_0(\Rn)$. When
$\Omega=\R^n_\pm$ the abbreviations $r^\pm=r_{\R^n_\pm}$ and
$e^\pm=e_{\R^n_\pm}$ are used. Here $\Rn_\pm$ denotes the halfspace
where $x_n\gtrless 0$ and 
$\overline{\R}_{\pm}^n:=\{\,x\in\Rn\mid x_n\gtreqless 0\,\}$ its closure.

Moreover, $B^{s}_{p,q;0}(\overline{\Omega})$, $\cal S_0'(\overline{\Omega})$
etc.\  denote subspaces supported by $\overline{\Omega}$, e.g.,
 \begin{equation}
 \cal S'_0(\Rp^n)=\{\,u\in\cal S'(\Rn)\mid \supp u\subset\Rp^n\,\}.
 \label{-49}
 \end{equation}

The Fourier transform is denoted by $\cal Fu(\xi)=\hat u(\xi)=\int_{\Rn}
e^{-ix\cdot\xi}u(x)\,dx$, and  the notation $\cal F^{-1}v(x)=\check
v(x)$ is used for its inverse; the co-Fourier transform is written 
$\overline{\cal F}u(\xi)=
\int_{\Rn}e^{ix\cdot\xi}u(x)\,dx$ and its inverse is denoted
$\overline{\cal F}^{-1}v(x)$. For functions $u(x',x_n)\in\cal
S(\Rn)$, where $x'=(x_1,\dots,x_{n-1})$, a partial transformation in
$x'$ is indicated by $\cal F_{x'\to\xi'}u(x',x_n)=
\acute u(\xi',x_n)=\int_{\R^{n-1}} 
e^{-ix'\cdot\xi'}u(x',x_n)\,dx'$. 
Indexations like this are also used for the other transformations and for
functions of, say, $n-1$ variables.
However, in any case `$\hat{\hphantom v}$'
indicates a Fourier transformation
with respect to all variables; when the meaning is clear this replaces
$\cal F_{x'\to\xi'}v(x')$ etc.

For $u\in C^\infty(\Rp^n)$ we let $\gamma_0u(x')=u(x',0)$ and
$\gamma_ju=\gamma_0D^j_{x_n}u$. As usual
$\ang{x}=(1+|x|^2)^{\scriptscriptstyle\frac{1}{2}}$ and 
$\ang{x'}=\ang{(x',0)}$, where $|\cdot|$ is the euclidean norm on
$\Rn$. The measure $(2\pi)^{-n}\,dx$ is abbreviated $\dbar x$, and
$\dbar x':=(2\pi)^{1-n}\,dx'$ on $\R^{n-1}$. Usually it is clear from
the context whether $p$ denotes an integral-exponent in $\,]0,\infty]$
or a symbol $p=p(x,\xi)$ (in $S^{d}_{1,0}$).

The convention that $t_\pm=\max(0,\pm t)$ is used for $t\in\R$, and
$\bignt t$ and $\smlnt t$ denote the largest integer $\le t$
and the smallest integer $\ge t$, respectively.
For each given assertion  we shall follow D. E. Knuth's
suggestion in \cite{K} and let $\kd{assertion}$ denote 1 and 0
when the assertion is true respectively false.

\subsection{The spaces} \label{spaces-ssect}
For the definition of $B^{s}_{p,q}$ and $F^{s}_{p,q}$ the conventions in 
\cite{Y1} (that are equivalent to the ones in \cite{T2,T3}) are employed.

First a partition of unity, $1=\sum_{j=0}^\infty\Phi_j$, is 
constructed: From $\Psi\in C^\infty(\R)$, such
that $\Psi(t)=1$ for $0\le t\le\tfrac{11}{10}$ and $\Psi(t)=0$
for $\tfrac{13}{10}\le t$, the functions
 \begin{equation}
 \Psi_j(\xi)=\kd{j\in\N_0} \Psi(2^{-j}|\xi|)
 \label{1.14} 
 \end{equation}
are introduced and used to define
 \begin{equation}
  \Phi_j(\xi)=\Psi_j(\xi)-\Psi_{j-1}(\xi),
  \quad\text{ for}\quad j\in\Bbb Z\,.
  \label{1.15} 
 \end{equation}
Secondly there is then a decomposition, with (weak) convergence 
in $\cal S'$,
 \begin{equation}
 u=\sum_{j=0}^\infty\,u_j
 =\sum_{j=0}^\infty\,\cal F^{-1}\Phi_j\cal Fu\,,
 \quad\text{ for every}\quad u\in\cal S'\,.
 \label{1.16} 
 \end{equation}
Here the convention $u_j:=\cal F^{-1}\Phi_j\cal Fu=\cal
F^{-1}(\Phi_j\hat u)$ is used, as it is throughout.

\smallskip

Now the Besov space $B^{s}_{p,q}(\Rn)$ with {\em smoothness
index $s\in\R$, integral-exponent $p\in\left]0,\infty\right]$ 
{\rm and} sum-exponent} $q\in\left]0,\infty\right]$, is defined as
 \begin{equation}
 B^{s}_{p,q}(\Rn)=\bigl\{\,u\in\cal S'(\Rn)\bigm|
 \Norm{ \{2^{sj} \norm{\cal F^{-1}\Phi_j\cal Fu}{L_p} \}_{j=0}^\infty
       }{\ell_q} <\infty\,\bigr\},
 \label{1.17} 
 \end{equation}
and the Triebel--Lizorkin space $F^{s}_{p,q}(\Rn)$ with {\em smoothness
index $s\in\R$, integral-exponent $p\in\left]0,\infty\right[\,$ 
{\rm and} sum-exponent} $q\in\left]0,\infty\right]$ is defined as
 \begin{equation}
 F^{s}_{p,q}(\Rn)=\bigl\{\,u\in\cal S'(\Rn)\bigm|
 \Norm{ \norm{ \{2^{sj}\cal F^{-1}\Phi_j\cal
 Fu\}_{j=0}^\infty}{\ell_q} (\cdot)}{L_p} <\infty\,\bigr\}\,.
 \label{1.18}
 \end{equation} 
For the history of these spaces we refer to Triebel's books \cite{T2,T3}. 

The spaces $B^{s}_{p,q}$ and $F^{s}_{p,q}$ are quasi-Banach spaces
with the quasi-norms given by the finite expressions in
\eqref{1.17} and \eqref{1.18}. Concerning an analogue of \eqref{1.27} one has
 \begin{equation}
 \norm{f+g}{B^{s}_{p,q}}\le(\norm f{B^{s}_{p,q}}^\lambda+
 \norm {g}{B^{s}_{p,q}}^\lambda)^{\frac{1}{\lambda}},
 \quad\text{ for $\lambda=\min(1,p,q)$},
 \label{1.27'}
 \end{equation}
with a similar inequality for the Triebel--Lizorkin spaces.

\begin{exmp} \label{delta-exmp} 
The delta distribution $\delta_0(x)$ belongs to
$B^{\fracci np-n}_{p,\infty}(\Rn)$ for each 
$p\in\,]0,\infty]$, since by definition \eqref{1.17},
 \begin{equation}
 \norm{\delta_0}{B^{\fracci np-n}_{p,\infty}}
  =\max_{j=0,1}(2^{j(\fracci np-n)}\norm{\check\Phi_j}{L_p}) <\infty.
 \label{1.32}
 \end{equation}
\end{exmp}

\begin{rem} \label{tvs-rem}
For the reader's sake a piece of folklore is recalled, namely 
that \eqref{1.27'} leads
to the fact that, say, $d(u,v)=\norm{u-v}{F^{s}_{p,q}}^\lambda$ for
$\lambda=\min(1,p,q)$ is a metric on $F^{s}_{p,q}(\Rn)$. For this
reason  both $B^{s}_{p,q}(\Rn)$ and $F^{s}_{p,q}(\Rn)$ are
topological vector spaces with the topology induced by a translation
invariant metric\,---\,even when $p$ or $q$ is $<1$. The same conclusion
applies to, say, $\Bbb L(B^{s}_{p,q},B^{t}_{r,o})$ (where the operator
quasi-norm inherits the constants $c$ and $\lambda$ from $B^{t}_{r,o}$).

Concerning functional analysis, this shows 
that these spaces {\em in any case\/} are
examples of the F-spaces in W.~Rudin's monograph \cite{R}, 
and hence one may refer to the exposition there. 
In particular the closed graph theorem is applicable.
\end{rem}

\subsection{Properties}  \label{prprts-ssect}
In the rest of this subsection the explicit mention of the restriction 
$p<\infty$ concerning the~Triebel--Lizorkin spaces is omitted. E.g., 
\eqref{1.19} below should be read with $p\in\,]0,\infty]$ in 
the $B^{s}_{p,q}$ part and with $p\in\,]0,\infty[\,$ in 
the $F^{s}_{p,q}$ part. 
Furthermore, to avoid repetition the underlying set is suppressed when it
is $\Rn$.

\smallskip

Identifications with other spaces are found in Section~\ref{summ-sect}.

The spaces $B^{s}_{p,q}$ and $F^{s}_{p,q}$ are
complete, for $p$ and $q\ge1$ they are Banach spaces, and in any case
$\cal S\hookrightarrow B^{s}_{p,q},\,F^{s}_{p,q}
\hookrightarrow\cal S'$ are continuous. Moreover, the image of 
$\cal S$ is dense in $B^{s}_{p,q}$ and in $F^{s}_{p,q}$
when both $p$ and $q<\infty$, and $C^\infty$ is so in 
$B^{s}_{\infty,q}$ for $q<\infty$ (where the latter assertion 
is inferred from Triebel's proof of the former \cite{T2}). 

The definitions imply that $B^{s}_{p,p}=F^{s}_{p,p}$,
and they imply the existence of {\em simple\/} embeddings for 
$s\in\R,\,\,p\in\left]0,\infty\right]$ 
and $o$ and $q\in\left]0,\infty\right]$, 
  \begin{gather}
  B^{s}_{p,q}\hookrightarrow B^{s}_{p,o},
          \quad F^{s}_{p,q}\hookrightarrow F^{s}_{p,o},
                              \quad\text{when $q\le o$}, 
  \label{1.19}\\ 
  B^{s}_{p,q}\hookrightarrow 
                  B^{s-\varepsilon}_{p,o},
          \quad F^{s}_{p,q}\hookrightarrow 
                  F^{s-\varepsilon}_{p,o},\quad\text{when
                  $ \varepsilon>0,$}
   \label{1.19'} \\ 
  B^{s}_{p,\min(p,q)}\hookrightarrow F^{s}_{p,q}
     \hookrightarrow B^{s}_{p,\max(p,q)}.
 \label{1.19''}
 \end{gather}
There are Sobolev embeddings if $s-\fracc{n}p\ge
t-\fracc{n}r$ and $r>p$, more specifically
 \begin{gather}
 B^{s}_{p,q}\hookrightarrow B^{t}_{r,o},
 \quad\text{ provided $q\le o$ when $ s-\fracc{n}p
 =t-\fracc{n}r$},
 \label{1.20} \\
 F^{s}_{p,q}\hookrightarrow F^{t}_{r,o},
 \quad\text{ for any $o$ and $q\in\,]0,\infty].$}
 \label{1.20'}
 \end{gather}
Furthermore, Sobolev embeddings also exist between the two scales,
in fact under the assumptions $\infty\ge p_1>p>p_0>0$ and 
$s_0-\fracc{n}{p_0}=s-\fracc{n}p=s_1-\fracc{n}{p_1}$  one has that
 \begin{equation}
  B^{s_0}_{p_0,q_0}\hookrightarrow F^{s}_{p,q}
  \hookrightarrow B^{s_1}_{p_1,q_1},\quad
  \text{for $q_0\le p$ and $p\le q_1$.}  
  \label{1.21} 
 \end{equation}
This is obtained from \eqref{1.20}, \eqref{1.20'} and \eqref{1.19''} 
except for the cases with equality, which are interpolation results due 
to J. Franke \cite{F3} and B. Jawerth \cite{J}, respectively.
 
\bigskip

By use of \eqref{1.16}, \eqref{1.17} and \eqref{1.20}, it is found when
$0<p,q\le\infty$ that
 \begin{equation}
  \begin{gathered}
  B^{s}_{p,q}\hookrightarrow B^{0}_{\infty, 1}\hookrightarrow
  C\hookrightarrow L_\infty\hookrightarrow B^{0}_{\infty,\infty}
  ,\\
  \text{if $s>\fracc{n}p$, or if $s=\fracc{n}p$ and $q\le1$}. 
  \end{gathered}
  \label{1.22} 
 \end{equation} 
Then \eqref{1.21} gives for the Triebel--Lizorkin spaces that for 
$0<q\le\infty$,
 \begin{equation}
  \begin{gathered}
  F^{s}_{p,q}\hookrightarrow B^{0}_{\infty,1}\hookrightarrow
  C\hookrightarrow L_\infty,\\
  \text{if $s>\fracc{n}p$, or if $s=\fracc{n}p$ and $p\le1$}.
  \end{gathered}
  \label{1.23}
 \end{equation}
Moreover, when $n(\fracc{1}p-1)_+\le s<\fracc{n}p$
one has, with $\tfrac{n}{t}=\fracc{n}p-s$, that
 \begin{equation}
 \begin{gathered}
  F^{s}_{p,q}\hookrightarrow\bigcap\{\, L_r\mid
  p\le r\le t\,\},
 \\
  \text{ provided $q \le 1+\kd{1<p} $ if $s=0$}.
 \end{gathered}
 \label{1.24}
 \end{equation}
See \cite{JJ93} or \cite{JJ94mlt} for a proof of this and 
of the corresponding fact that 
 \begin{equation}
 B^{s}_{p,q}\hookrightarrow\bigcap\{\, L_r\mid
 p\le r< t\,\},
 \label{1.25}
 \end{equation}
where $r=t$ can be included in general when $q\le t$. For $s=0$ one has
$B^{s}_{p,q}\hookrightarrow L_p$ for $q\le\min(2,p)$ and $p\ge1$. (Cf. \cite[p.~97]{T3}
for the pitfalls in the case $p<1$.) 

\bigskip

For an open set $\Omega\subset\Rn$ the space 
$B^{s}_{p,q}(\overline{\Omega})$ is defined by restriction,
 \begin{gather}
 B^{s}_{p,q}(\overline{\Omega})=r_\Omega B^{s}_{p,q}=
 \{\,u\in\cal D'(\Omega)\mid \exists v\in B^{s}_{p,q}\colon 
 r_\Omega v=u\,\} 
 \label{1.34} \\
 \norm{u}{B^{s}_{p,q}(\overline{\Omega})}=\inf \bigl\{\,
 \norm{v}{B^{s}_{p,q}}\bigm| r_\Omega v=u\,\bigr\},
 \label{1.34'}
 \end{gather}
and $F^{s}_{p,q}(\overline{\Omega})$ is defined analogously.
By the definitions all the~embeddings
in \eqref{1.19}--\eqref{1.25} carry over to the corresponding 
scales over $\Omega$. 

Moreover, when $\Omega$ is a suitable set of finite measure and 
$\infty\ge p\ge r>0$ the inclusion $L_p(\Omega)\hookrightarrow 
L_r(\Omega)$ carries over to the embeddings
 \begin{equation}
 B^{s}_{p,q}(\overline{\Omega})\hookrightarrow 
 B^{s}_{r,q}(\overline{\Omega}), \qquad
 F^{s}_{p,q}(\overline{\Omega})\hookrightarrow 
 F^{s}_{r,q}(\overline{\Omega}).
 \label{1.38}
 \end{equation} 
When $\Omega$ is bounded this is shown in \cite[3.3.1]{T2}, 
except for the case $q=\infty$ for the $F^{s}_{p,q}$ spaces. In
\cite{JJ94mlt} there is a (simpler) proof of \eqref{1.38} in
its full generality.

For $m\in\Z$ the order-reducing operator $\Xi^{m}:=
\cal F^{-1}\ang{\xi}\cal F$ is bounded
 \begin{equation}
 \Xi^{m}\colon B^{s}_{p,q}\overset\sim\rightarrow B^{s-m}_{p,q},
 \qquad
 \Xi^{m}\colon F^{s}_{p,q}\overset\sim\rightarrow F^{s-m}_{p,q}
 \label{Xi'}
 \end{equation}
and bijective for any $(s,p,q)$, cf.~\cite{T2}. On $\R^{n-1}$ the
corresponding operator is denoted $\Xi^{\prime\ m}$.

\subsection{Convergence theorems} \label{Ythm-ssect} 
Yamazaki's theorems are recalled from the article \cite{Y1}, where the 
convergence of the series in the following two theorems was proved first.

\begin{thm} \label{Y1-thm} \ Let $s\in\R$, $p$ and 
$q\in\,]0,\infty]$ and suppose $u_j\in\cal S'(\Rn)$ satisfies
 \begin{equation}
  \supp \hat u_j\subset\bigl\{\,\xi\bigm| \kd{j>0}A^{-1}2^j\le |\xi|\le
 A2^j\,\bigr\},\quad\text{ for}\quad j\in\N_0,
 \label{1.30}
 \end{equation}
for some $A>0$. Then the following holds, if $p<\infty$ in \rom{(2)}:
 \begin{itemize}\addtolength{\itemsep}{\jot}
 \item[{\rm (1)\/}] If $\Norm{\{2^{sj}
                     \norm{u_j}{L_p}\}^\infty_{j=0}}{\ell_q}=B<\infty$, 
 then the series
 $\sum_{j=0}^\infty u_j$ converges in $\cal S'(\Rn)$ to a limit $u\in
 B^{s}_{p,q}(\Rn)$ and the estimate $\norm{u}{B^{s}_{p,q}}\le CB$
 holds for some constant $C=C(n,A,s,p,q)$. 

 \item[{\rm (2)\/}]  If 
 $\Norm{\norm{\{2^{sj}u_j\}^\infty_{j=0}}{\ell_q}(\cdot)}{L_p}=B<\infty$, 
 then the series
 $\sum_{j=0}^\infty u_j$ converges in $\cal S'\!(\Rn)$ to a limit $u\in
 F^{s}_{p,q}(\Rn)$ and the estimate $\norm{u}{F^{s}_{p,q}}\le CB$
 holds for some constant $C=C(n,A,s,p,q)$. 
 \end{itemize}
\end{thm}
 
Hence, when $q<\infty$ the series in \eqref{1.16} converges in
$B^{s}_{p,q}$ for $u\in B^{s}_{p,q}$, and similarly for $u\in F^{s}_{p,q}$.

The second of these theorems states that the spectral conditions on
the series $\sum_{j=0}^\infty u_j$ can be relaxed if the smoothness
index $s$ is sufficiently large.

\begin{thm} \label{Y2-thm} \ Let $s\in\R$, $p$ and 
$q\in\,]0,\infty]$ and suppose $u_j\in\cal S'(\Rn)$ satisfies
 \begin{equation}
 \supp \hat u_j\subset\bigl\{\,\xi\bigm| |\xi|\le A2^j\,\bigr\},\quad
 \text{ for}\quad j\in\N_0,
 \label{1.31}
 \end{equation}
for some $A>0$. Then the following holds, if $p<\infty$ in \rom {(2)}:
 \begin{itemize}\addtolength{\itemsep}{\jot}
\item[{\rm (1)\/}] If $s>n(\fracp-1)_+$ and
if $\Norm{\{2^{sj}\norm{u_j}{L_p}\}^\infty_{j=0}}{\ell_q}=B<\infty$, 
then the series
$\sum_{j=0}^\infty u_j$ converges in $\cal S'(\Rn)$ to a limit $u\in
B^{s}_{p,q}(\Rn)$ and the estimate $\norm{u}{B^{s}_{p,q}}\le CB$
holds for some constant $C=C(n,A,s,p,q)$. 

\item[{\rm (2)\/}] If $s>n(\tfrac1{\min(p,q)}-1)_+$, and
if $\Norm{\norm{\{2^{sj}u_j\}^\infty_{j=0}}{\ell_q}(\cdot)}{L_p}=B<\infty$, 
then the series
$\sum_{j=0}^\infty u_j$ converges in $\cal S'(\Rn)$ to a limit $u\in
F^{s}_{p,q}(\Rn)$ and the estimate $\norm{u}{F^{s}_{p,q}}\le CB$
holds for some constant $C=C(n,A,s,p,q)$. 
\end{itemize}
\end{thm}

For the proofs of Theorems \ref{Y1-thm} and \ref{Y2-thm} the reader 
is referred to \cite{Y1}. In part Theorem \ref{Y2-thm} is based on 
\cite[Lemma 3.8]{Y1}, which for later reference is stated for
$s<0$ in a slightly generalised version (that is proved analogously):

\begin{lem} \label{Y-lem} 
For each $s<0$ and $q$ and $r\in\,]0,\infty]$ there exists a $c<\infty$ 
such that for any sequence $\{a_j\}_{j=0}^\infty$ of complex numbers 
 \begin{equation}
 \Norm{ \bigl\{2^{sj}({\textstyle\sum_{k=0}^j}|a_k|^r)^{\fracci1r}\bigr\}
 ^\infty_{j=0} }{\ell_q}
 \le c\norm{ \{2^{sj}a_j\}^\infty_{j=0} }{\ell_q}
 \label{3.7'} 
 \end{equation} 
(with
modification for $r=\infty$).
\end{lem}

\subsection{Tensor products} \label{tensor-ssect}
As a tool in connection with the Poisson operators in 
Section \ref{pois-ssect} a boundedness result for the operator that
tensorises with the delta-distribution $\delta_0$ is included here.
 
\begin{prop} \label{tensor-prop} 
Let $p$ and $q\in\,]0,\infty]$ and suppose that 
$s+1-\fracp<0$. Then 
 \begin{align}
 \norm{f\otimes\delta_0}{B^{s}_{p,q}(\Rn)}&\le
        c\norm{\delta_0}{B^{\fracpi-1}_{p,\infty}(\R)}        
         \norm{f}{B^{s+1-\fracpi}_{p,q}(\R^{n-1})}, 
 \label{11} \\
        \norm{f\otimes\delta_0}{F^{s}_{p,q}(\Rn)}&\le
        c(p,q) \norm{f}{B^{s+1-\fracpi}_{p,p}(\R^{n-1})}, 
 \label{12}
 \end{align} 
when $p<\infty$ holds in \eqref{12}.
\end{prop}

\begin{proof}
Let $f\in B^{s+1-\fracpi}_{p,q}(\R^{n-1})$ and introduce the
decompositions 
 \begin{alignat}{2}
 f(x')&= \sum_{k=0}^\infty \cal F^{-1}_{\xi'\to x'}
                              \Phi'_k \cal F_{x'\to\xi'} f(x')\,&&=:\,
             \sum_{k=0}^\infty f_k(x'), 
 \label{20} \\
 \delta_0(x_n)&= \sum_{k=0}^\infty \cal F^{-1}_{\xi_n\to x_n}
                               \Phi^{(n)}_k(\xi_n)
             \,&&=:\, \sum_{k=0}^\infty \eta_k(x_n),
 \label{21}   
 \end{alignat}
where $\Phi'_k$ and $\Phi^{(n)}_k$ denote the $k^{th}$ element in the
partition of unity associated with the $x'$- and  $x_n$-space, respectively.
In \eqref{20} and \eqref{21}, and in the following,\linebreak[4]
$f_k:=\cal F^{-1}_{\xi'\to x'}
\Phi'_k \cal F_{x'\to\xi'} f(x')$ and $\eta_k:=\cal F^{-1}_{\xi_n\to x_n}
\Phi^{(n)}_k(\xi_n)$, whereas with superscripts $f^k=\cal F^{-1}_{\xi'\to x'}
\Psi'_k \cal F_{x'\to\xi'} f(x')$ and $\eta^k=\cal F^{-1}_{\xi_n\to x_n}
\Psi^{(n)}_k(\xi_n)$. 

This is used for the central relation 
 \begin{equation}
 f(x')\otimes\delta_0(x_n)=
 \smash[t]{
   \sum_{k=0}^\infty f_k\eta^{k-1} + \sum_{k=0}^\infty f^k\eta_k,
 }
 \label{22}
 \end{equation}
which holds since it is shown below that each of the two sums on the
right hand side converges in $\cal S'$. Indeed, given this convergence
it follows that 
 \begin{align}
 \quad\ \lim_N {\sum_{k=0}^N} f_k\eta^{k-1} +
 \lim_N {\sum_{k=0}^N} f^k \eta_k 
 &=\lim_N {\sum_{k,l=0}^N} f_k\eta_l
 \label{23} \\
 &=\,\lim_N \cal F^{-1}(\Psi'_N\otimes\Psi^{(n)}_N)\cal
  F(f\otimes\delta_0) 
 \notag \\
 &=f(x')\otimes\delta_0(x_n),
 \notag
 \end{align}
since $\Psi'_N\otimes\Psi^{(n)}_N(\xi)$ equals the
$C^\infty_0$ function $\Psi'_0(2^{-N}\xi')\Psi^{(n)}_0(2^{-N}\xi_n)$.

In the following, Theorem \ref{Y1-thm} is applied to each sum in \eqref{22}.
The first step is to note the~spectral conditions, 
 \begin{align}
 \supp \cal F(f_k\eta^{k-1})&\subset
 \{\,\xi\mid\tfrac{11}{10}2^{k-1}\le|(\xi',0)|\le\tfrac{13}{10}2^k,\
             |(0,\xi_n)|\le\tfrac{13}{10}2^{k-1}\,\}
 \notag \\
 &\subset\{\,\xi\mid\tfrac{11}{20}2^k\le|\xi|\le\tfrac{39}{20}2^k\,\},
 \label{24} \\
 \supp \cal F(f^k\eta_k)&\subset
 \{\,\xi\mid\tfrac{11}{20}2^k\le|\xi|\le\tfrac{26}{10}2^k\,\}.
 \label{25}
 \end{align}
Secondly the $\ell^s_q(L_p)$ norms of the sums are estimated. 
From $\eta^{l}=2^{l}\eta_0(2^{l}x_n)$ it is seen that 
$\norm{f_k\eta^{k-1}}{L_p(\Rn)}=
2^{(k-1)(1-\fracpi)}\norm{\eta_0}{L_p(\R)}\norm{f_k}{L_p(\R^{n-1})}$,
since the $L_p$ norm is multiplicative. Then, with
$c=\norm{\delta_0}{B^{\fracpi-1}_{p,\infty}}$, 
 \begin{equation}
 \Norm{ \{2^{ks}\norm{f_k\eta^{k-1}}{L_p} \}^\infty_{k=0} }{\ell_q}\le
 2^{\fracpi-1}c
 \Norm{ \{2^{k(s+1-\fracpi)}\norm{f_k}{L_p}\}^\infty_{k=0} }{\ell_q},
 \label{26}
 \end{equation} 
cf.~Example~\ref{delta-exmp}.
Concerning the second sum one finds in a similar way that
 \begin{align}
 2^{ks}\norm{f^k\eta_k}{L_p}&=
  2^{ks}\norm{f_0+\dots+f_k}{L_p}  \norm{\eta_k}{L_p}
 \label{27} \\
 &\le\,2^{k(s+1-\fracpi)}(\norm{f_0}{L_p}^{r}+\dots+
  \norm{f_k}{L_p}^r)^{\fracci1r}
  \norm{\delta_0}{B^{\fracpi-1}_{p,\infty}}, 
 \notag
 \end{align}
when $r=\min(1,p)$. The assumption $s+1-\fracp<0$ in 
Proposition \ref{tensor-prop} now allows an application of 
Lemma \ref{Y-lem} above, leading to the estimate
 \begin{equation}
 \Norm{ \{2^{ks}\norm{f^k\eta_k}{L_p}\}^\infty_{k=0} }{\ell_q}
 \le c
 \Norm{ \{2^{k(s-1+\fracpi)}\norm{f_k}{L_p}\}^\infty_{k=0} }{\ell_q},
 \label{28}
 \end{equation}
with $c=c'\norm{\delta_0}{B^{\fracpi-1}_{p,\infty}}$  
for some $c'<\infty$ depending on $p$ and $q$.

From \eqref{24}, \eqref{25}, \eqref{26} and \eqref{28} it follows  
by  Theorem \ref{Y1-thm} that the series in \eqref{22} 
converge in $\cal S'$, and that the sums belong to 
$B^{s}_{p,q}(\Rn)$ with norms estimated by constants times the
right hand sides of \eqref{26} and \eqref{28}, respectively. By use of
\eqref{22} it follows that also $f\otimes\delta_0\in B^{s}_{p,q}$,
and by application of the quasi-triangle inequality 
this implies \eqref{11}.

\smallskip

For $p\le q$ the estimate in \eqref{12} follows from \eqref{11} by use of the
embedding $B^{s}_{p,p}\hookrightarrow F^{s}_{p,q}$. The case
$q<p$ is obtained like \eqref{11} by application of Theorem~\ref{Y1-thm} to
the sums in \eqref{22}. However, the necessary estimates of the
$L_p(\ell^s_q)$ norms are substantially more complicated than \eqref{26} and
\eqref{28}. But with (mainly) notational changes one can proceed as
in \cite[p.~136]{T2}, where estimates analogous to \eqref{11} and
\eqref{12} are shown for a right inverse of $\tilde\gamma_0$ 
(cf.~Section~\ref{gamm-ssect} below).

[To be more specific one can treat the first sum in \eqref{22} by letting
$a_k$ in \cite[2.7.2/31]{T2} be equal to $f_k$, and for simplicity replace the
reference to \cite[Thm.\ 1.6.3]{T2} by an application of Theorem~%
\ref{Y1-thm}. Concerning the second sum in \eqref{22} one can start
by showing an analogue of \cite[2.7.2/34]{T2} for $\eta_k$ and then
proceed as before except with $a_k=f^k$ instead; in suitable late stages of
the various estimates one can then introduce
$\norm{f^k}{L_p}\le(\norm{f_0}{L_p}^r+\dots+\norm{f_k}{L_p}^r)^{\fracci1r}$,
for $r=\min(1,p)$, together with Lemma \ref{Y-lem}. ]
\end{proof}

\subsection{Traces} \label{gamm-ssect}
In preparation for Section \ref{trac-ssect} below on general trace
operators in the Boutet de Monvel calculus some
well-known facts about restriction to hyperplanes is modified to suit
the purposes there.

The basic trace operator is the two-sided restriction operator, which
takes $v(x)$ in $C^\infty\!(\Rn)$ to $v(x',0)$; it is denoted by
$\tilde\gamma_0v$. The properties of $\tilde\gamma_0$ are investigated in 
numerous papers, see \cite{FJ2}, e.g., and the~references therein.

For $u\in C^\infty(\Rp^n)$, the {\em one-sided\/} restriction
operator $\gamma_0 u$ is defined  by letting
 \begin{equation}
 \gamma_0u=\tilde\gamma_0v,\quad\text{ when}\quad r^+v=u
 \quad\text{holds for}\quad v\in C^\infty(\Rn).
 \label{g0} 
 \end{equation}
Evidently one has the
intrinsic description $\gamma_0u(x')=u(x',+0)$. Moreover, let
$\tilde\gamma_j=\tilde\gamma_0D^j_{x_n}$ and $\gamma_j=\gamma_0D^j_{x_n}$.

Henceforth the following simplifying notation is employed: for $k\in\Z$
the parameter $(s,p,q)$ is said to belong to the set $\Dm_k$ if  
 \begin{equation}
 s>k+\max(\fracp-1,\fracc np-n),
 \label{Dkineq}
 \end{equation}
cf.\  Figure \ref{D0-fig} below. 
That $s\ge k+\max(\fracp-1,\fracc np-n)$ means that $(s,p,q)$ belongs 
to the closure of $\Dm_k$, so $(s,p,q)\in\overline{\Dm}_k$ is written then.

For the one-sided trace operator $\gamma_j$ the following result is
needed below (whereas those for $\tilde\gamma_0$ in \cite{T2,T3} do
not suffice).

\begin{lem} \label{gamm1-lem} 
For each $j\in\N_0$ the trace $\gamma_j$ extends uniquely to a bounded
operator (when $p<\infty$ in \eqref{-51})
 \begin{alignat}{2}
 \gamma_j&\colon  B^{s}_{p,q}(\Rp^n) \to B^{s-j-\fracpi}_{p,q}(\R^{n-1}),
 \quad&\text{ for}\quad (s,p,q)&\in \Dm_{j+1}, 
 \label{-50} \\
 \gamma_j&\colon  F^{s}_{p,q}(\Rp^n) \to F^{s-j-\fracpi}_{p,p}(\R^{n-1}),
 \quad&\text{ for}\quad (s,p,q)&\in \Dm_{j+1}.
 \label{-51}
 \end{alignat}
Moreover, when $(s,p,q)\notin \overline{\Dm}_{j+1}$, there is not any
extension of $\gamma_j$ with the continuity properties in \eqref{-50} 
or \eqref{-51}.
\end{lem}
 
\begin{proof} It is known that $\tilde\gamma_0$ has continuity
properties corresponding to \eqref{-50} and \eqref{-51}, 
cf.\  \cite[2.7.2]{T2}, so it 
is sufficient to see that $\gamma_0$ is well defined by \eqref{g0}, i.e.,
$\tilde\gamma_0v=0$ should hold whenever $v$ belongs to
$B^{s}_{p,q;0}(\overline{\R}^n_-)$ or
$F^{s}_{p,q;0}(\overline{\R}^n_-)$.
It suffices to treat $q<\infty$, and by the continuity of 
$\tilde\gamma_0$, it is enough to prove that $\{\,\varphi\in\cal
S(\Rn)\mid\overline{\R}^n_-\supset\supp\varphi\,\}$ is dense in 
$B^{s}_{p,q;0}(\overline{\R}^n_-)$ and $F^{s}_{p,q;0}(\overline{\R}^n_-)$
for $p<\infty$, and that $\{\,\varphi\in
C^\infty(\Rn)\mid\overline{\R}^n_-\supset\supp\varphi\,\}$ is dense in
$B^{s}_{\infty,q;0}(\overline{\R}^n_-)$.

However, if $\tau_h f=f(\cdot-he_n)$, $\tau_h\to1$
strongly on $\cal S(\Rn)$ and on $C^\infty(\Rn)$ for $h\to0$. 
By use of the denseness and the relations $\cal
F^{-1}\Phi_j\cal F\tau_hf=\tau_h\cal F^{-1}\Phi_j\cal Ff$, it is
seen that $\tau_h\to1$ in the strong operator topology on
$B^{s}_{p,q}(\Rn)$ and $F^{s}_{p,q}(\Rn)$\,---\,for any admissible $(s,p,q)$.

For $u\in B^{s}_{p,q;0}(\overline{\R}^n_-)$ and $\varepsilon>0$ we 
take $h<0$ so small that $\norm{u-\tau_hu}{B^{s}_{p,q}}<
\tfrac{\varepsilon}{2}$, and let $g\in C^\infty(\Rn)$ satisfy $\supp
g\subset\overline{\R}^n_-$ and $g=1$ on $\supp \tau_hu$. Because
multiplication by $g$ is continuous in $B^{s}_{p,q}(\Rn)$, we obtain that
 \begin{equation}
 \norm{u-gv_k}{B^{s}_{p,q}}\le\tfrac{\varepsilon}{2}+
 \norm{g\tau_hu-gv_k}{B^{s}_{p,q}}\le\varepsilon
 \label{-52}
 \end{equation}
holds eventually,
when $v_k\in\cal S\!(\Rn)$ (resp.\  $C^\infty\!(\Rn)$ for $p=\infty$)
converges to $\tau_hu$ in $B^{s}_{p,q}(\Rn)$. When $u\in
F^{s}_{p,q;0}(\overline{\R}^n_-)$ one can proceed in the same manner.

For $j>0$ it is now obvious that the composite $\gamma_0D^j_{x_n}$ is
bounded as in \eqref{-50} and \eqref{-51}. The uniqueness follows from
the denseness of $\cal S$ or $C^\infty$ when $q<\infty$.

When $(s,p,q)\notin\overline{\Dm}_{j+1}$ the non-extendability of 
$\gamma_j$ follows from Lemma \ref{gamm2-lem} below
regardless of the~choice of $u\ne0$ and $z'$. Indeed, for $p<1$ the
existence of $r^+v_k$ shows that $\gamma_j$ is not continuous at 0 from
any space $B^{s}_{p,q}(\Rp^n)$ or $F^{s}_{p,q}(\Rp^n)$ when $s<\fracc
np -(n-1)+j$. For $1\le p<\infty$ the existence of $r^+u_k$ yields the
same conclusion (for spaces with $s<\fracp+j$), while in the case
$p=\infty$ and $s<j$ a Sobolev embedding $B^{\fracci
1r+j}_{r,q}\hookrightarrow B^{s}_{\infty,q}$ reduces the question to the
case $p<\infty$.
\end{proof}

The next lemma was used in the proof above, and it will later on
provide counterexamples that are strong enough to
show that each trace operator $T$, that {\em has\/} class
$j\in\Z$, is not extendable to spaces with $(s,p,q)\ne
\overline{\Dm}_j$, cf.~Theorem \ref{trac-thm} below.

\begin{lem} \label{gamm2-lem} 
For each $j\in\N_0$, $u\in\cal S(\R^{n-1})\setminus\{\,0\,\}$ and 
$z'\in\R^{n-1}$ there exists two sequences with elements $u_k$ and 
$v_k\in\cal S(\Rn)$ with the properties 
 \begin{gather}
 {\left\{
  \begin{aligned}
  \tilde\gamma_ju_k(x')&=u(x') \text{ for each } k\in\N, 
  \\
  \lim_{k\to\infty}u_k&=0 \text{ in }
  B^{\fracpi+j}_{p,q}(\Rn) \text{ for } 1<q\le\infty,
  \\
  \lim_{k\to\infty}u_k&=0 \text{ in }
  F^{\fracpi+j}_{p,q}(\Rn) \text{ for } 1<p<\infty,
  \end{aligned}
 \right.} 
 \label{-53}
 \\[2\jot]
 {\left\{ 
  \begin{aligned}
  &\lim_{k\to\infty}\tilde\gamma_jv_k(x')=\delta_{z'}(x')\text{ in $\cal
  S'(\R^{n-1})$},
  \\
  &\lim_{k\to\infty} v_k=0 \text{ in } B^{\fracci np-(n-1)+j}_{p,q}(\Rn)
     \text{ for } 1<q\le\infty,
 \end{aligned}
 \right.} 
 \label{-54}
 \end{gather}
provided $0<p<\infty$ in \eqref{-53} respectively $0<p\le1$ in \eqref{-54}.
\end{lem}

\begin{proof} In the deduction of \eqref{-53}, recall the fact from
\cite{F3} that for $s>0$ 
 \begin{equation}
 \norm{f(x')\otimes g(x_n)}{B^{s}_{p,q}(\Rn)}\le c
 \norm{f}{B^{s}_{p,q}(\R^{n-1})}\norm{g}{B^{s}_{p,q}(\R)}.
 \label{-55}
 \end{equation}
Here $f$ will play the role of the given $u\in\cal S(\R^{n-1})$,
while for $g$ we shall take $w_k(x_n)=k^{-1}\sum_{l=1}^k2^{-lj}w(2^lx_n)$ 
for some auxiliary function $w\in\cal S(\R)$ satisfying $\supp \hat
w\subset\{\,\frac34\le|\xi_n|\le1\,\}$ and $\int\xi^j\hat w=2\pi$.
Observe that $\tilde\gamma_jw=1$.

One can let $u_k(x)=u(x')\cdot w_k(x_n)$, for $\tilde\gamma_j u_k=u$
and, since $\supp
\cal F(w(2^l\cdot))\subset\{\,\Phi^{(n)}_l\equiv 1\,\}$ where
$\Phi^{(n)}_l$ is as in \eqref{21} ff., 
 \begin{equation}
 \aligned 
 \norm{u_k}{B^{\fracpi+j}_{p,q}}&\le
 c\norm{u}{B^{\fracpi+j}_{p,q}(\R^{n-1})}
 \Norm{\{2^{\frac{l}{p}}\norm{w(2^l\cdot)}{L_p(\R)}\}_{l=1}^k}{\ell_q} k^{-1}
 \\
 &=c\norm{u}{B^{\fracpi+j}_{p,q}}\norm{w}{L_p} k^{\fracci1q-1}
 \endaligned
 \label{-55'}
 \end{equation}
by \eqref{-55}. Here $k^{\fracci 1q-1}\to 0$ for $k\to\infty$, when $q>1$.

When $p>1$ there is an embedding $B^{t}_{r,r}(\R)\hookrightarrow
F^{\fracpi+j}_{p,q}(\R)$ for an $r\in\,]1,p[\,$ and $t-\fracc1r=j$. Since
there is an analogue of \eqref{-55} for the $F^{s}_{p,q}$ spaces,
cf.~\cite{F3}, there is an estimate
$\norm{u_k}{F^{\fracpi+j}_{p,q}}\le c'\norm{u}{F^{\fracpi+j}_{p,q}}
\norm{w}{L_r} k^{\fracci1r-1}$. Because $r>1$, $u_k\to 0$ in
$F^{\fracpi+j}_{p,q}$. 

To obtain \eqref{-54} for $z'=0$ we take $f\in\cal S(\R^{n-1})$ and
$g\in\cal S(\R)$ satisfying
 \begin{equation}
 \alignedat2
 \supp \hat f&\subset\{\,|\xi'|\le\tfrac{1}{2}\,\},&\quad &\hat f(0)=1,
 \\ 
 \supp \hat g&\subset\{\,|\xi_n|\le\tfrac{1}{2}\,\},&\quad
    &\int\xi_n^j \hat g(\xi_n)\,d\xi_n=2\pi,
 \endalignedat
 \label{-55''}
 \end{equation}
and let
$v_{k}=\frac{1}{k}\sum_{l=k+1}^{2k} 2^{l(n-1-j)}f(2^lx')g(2^lx_n)$.
Now $\tilde\gamma_jv_k=\frac{1}{k}\sum 2^{l(n-1)}f(2^l\cdot)$
and, by a modification of the usual proof of the fact that 
$2^{k(n-1)} f(2^k\cdot)*\cdot\to1$ in the strong
operator topology on $C(\R^{n-1})$, it is verified that
$\tilde\gamma_j v_k*\cdot\to 1$ strongly on $C(\R^{n-1})$; in
particular this implies
$\tilde\gamma_j v_k(x')\to\delta_0(x')$ in $\cal S'(\R^{n-1})$.

Since
$\supp \cal F(fg(2^{l}\cdot))\subset\{\,|\xi|\le2^l\,\}$ and since
$p\le1$, Theorem \ref{Y2-thm} can be applied
to the sum defining $v_k$, which gives a constant $c$, independent
of $k$, such that
 \begin{equation}
 \norm{v_{k}}{B^{\fracci np-(n-1)+j}_{p,q}}\le
 \tfrac{c}{k}\Norm{ \{2^{l\fracci np}\norm{fg(2^l\cdot)}{L_p}\}_{l=k+1}
               ^{2k} }{\ell_q}
 =c\norm{f g}{L_p} k^{\fracci 1q-1},
 \label{-56 }
 \end{equation}
so for $z'=0$ the properties of the $v_k$ are proved. For
$z'\ne0$ one can simply translate.
\end{proof}

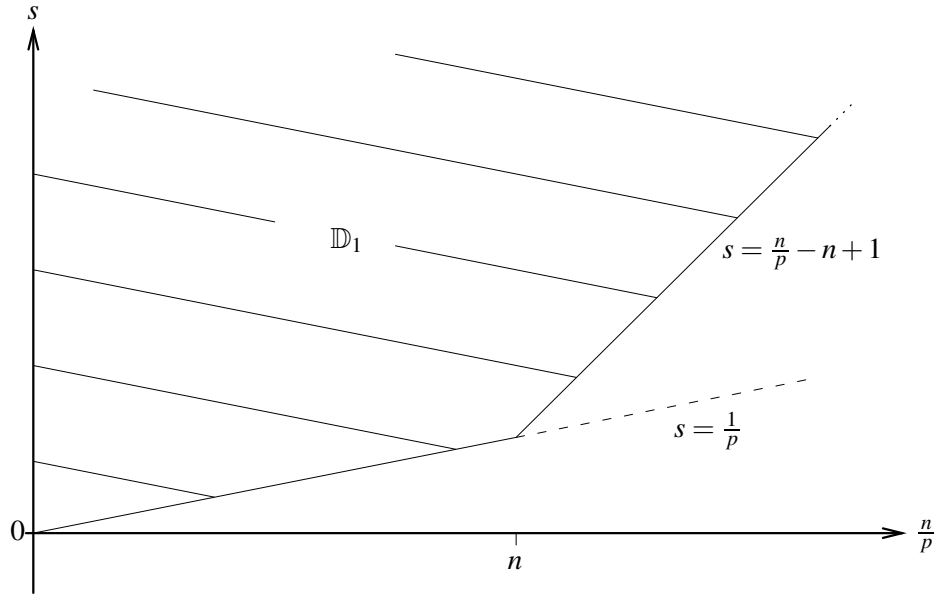
\begin{figure}[htbp]
\hfil
\setlength{\unitlength}{0.0125in}
\begin{picture}(395,261)(0,0)
\path(10,25)(210,65)
\dashline{4.000}(210,65)(330,89)
\path(210,25)(210,20)
\path(210,65)(340,195)
\dottedline{4}(340,195)(350,205)
\path(10,55)(85,40)        
\path(10,95)(185,60)
\path(10,135)(235,90)
\path(10,175)(110,155)      
\path(160,145)(268.333,123.333)
\path(35,210)(301.667,156.667)
\path(160,225)(335,190)
\thicklines
\path(7,25)(370,25)
\path(362.000,23.000)(370.000,25.000)(362.000,27.000)
\path(10,0)(10,235)
\path(12.000,227.000)(10.000,235.000)(8.000,227.000)
\put(0,22){\makebox(0,0)[lb]{\raisebox{0pt}[0pt][0pt]
{\shortstack[l]{{$0$}}}}}
\put(135,140){\makebox(10,15){$\Dm_1$}}
\put(295,140){\makebox(0,0)[lb]{\raisebox{0pt}[0pt][0pt]
{\shortstack[l]{{$s=\fracc np-n+1$}}}}}
\put(275,65){\makebox(0,0)[lb]{\raisebox{0pt}[0pt][0pt]
{\shortstack[l]{{$s=\fracp$}}}}}
\put(206,10){\makebox(0,0)[lb]{\raisebox{0pt}[0pt][0pt]
{\shortstack[l]{{$n$}}}}}
\put(375,22){\makebox(0,0)[lb]{\raisebox{0pt}[0pt][0pt]
{\shortstack[l]{{$\fracc np$}}}}}
\put(7,240){\makebox(0,0)[lb]{\raisebox{0pt}[0pt][0pt]
{\shortstack[l]{{$s$}}}}}
\end{picture}
\hfil
\caption[The borderline cases for $\tilde\gamma_0$]{The borderline 
 cases for $\tilde\gamma_0$ (when $n=5$)}
  \label{g0-fig}
\end{figure}

\begin{rem} \label{gamm-rem} 
When applied to $\gamma_j$ Lemma \ref{gamm2-lem} gives a
little more than stated in Lemma \ref{gamm1-lem}. In fact, if we for simplicity
consider $\tilde\gamma_0$, for each Hausdorff topological space $X$, there does
not exist any extension $\tilde\gamma_0\colon B^{s}_{p,q}(\Rn)\to X$, that
is continuous when $s=\fracp$ and $q>1$, or $s<\fracp$, cf.~\eqref{-53}. 
Moreover, there does not exist any
continuous extension $\tilde\gamma_0\colon F^{s}_{p,q}(\Rn)\to X$ when
$s<\fracp$, since $B^{\fracpi}_{p,\infty}\hookrightarrow F^{s}_{p,q}$
then, and for $1<p<\infty$ even $s=\fracp$ is excluded by \eqref{-53}.

These counterexamples, that were inspired by \cite[Lemma 2.2]{G3},
are sharper than previous ones obtained by H.\  Triebel, in that
$X=\cal S'(\Rn)$ is not assumed and the borderline cases $s=\fracp$ 
are included, cf.~\cite[2.7.2 Rem.~4]{T2}.

Moreover, the counterexamples provided by the $u_k$ for the Besov
spaces are optimal for $p\in[1,\infty[\,$, since 
$\tilde\gamma_0(B^{\fracpi}_{p,q}(\Rn))=L_p(\R^{n-1})$ for
$q\le\min(1,p)$ when $p\in\,]0,\infty[\,$, cf.\  \cite{FJ1}. 
For $p<1$ it is shown in \cite{FJ1} that one can not take 
$X=L_p+L_\infty$ when $q>p$. Contrary to this
$\tilde\gamma_0(B^{0}_{\infty,1}(\Rn))\subset C(\R^{n-1})$.
In the cases with $p=\infty$ and $p<q\le 1$ a strengthening of the $u_k$-%
counterexample above would be appropriate.

For the Triebel--Lizorkin spaces it was obtained in \cite{FJ2} that
$\tilde\gamma_0(F^{\fracpi}_{p,q}(\Rn))=L_p(\R^{n-1})$, independently
of $q$ when $p\le1$, and that one can not take $X=\cal S'(\R^{n-1})$ for
any $p\in\,]1,\infty[\,$ and $q\in\,]0,\infty]$. So for these spaces the
$u_k$-example removes the restriction on the space $X$.

For $p<1$ the $v_k$ yield stronger conclusions in the sense that even
$s\ge\fracc np-(n-1)$ is necessary for \eqref{-50} and \eqref{-51}. 
On the other hand, the conclusions are weaker in the sense that when $s<\fracc
np-(n-1)$ any $X$ with a continuous embedding $X\hookrightarrow\cal
D'$ is impossible, while other choices of $X$ might work. 
Indeed, from the~results quoted above it is seen that one can take 
the non-locally convex space $X=L_p$ when
$\fracp\le s\le \fracc np-(n-1)$ (except when $q>p$ for $s=\fracp$ in
the Besov case). Furthermore, in the borderline case $s=\fracc np-(n-1)$
for $p<1$, $\tilde\gamma_0(F^{\fracci np-(n-1)}_{p,r}(\Rn))+
\tilde\gamma_0(B^{\fracci np-(n-1)}_{p,q}(\Rn))\subset
\tilde\gamma_0(B^1_{1,1}(\Rn))=L_1(\R^{n-1})$ are well defined
subspaces of $\cal S'(\R^{n-1})$ when $q\le1$ and $r\in\,]0,\infty]$;
this shows the optimality of the $v_k$-counterexample in this case.  

\smallskip

However, the range spaces for $\tilde\gamma_0$ considered on
$B^{0}_{\infty,q}$, $B^{\fracci np-n+1}_{p,q}$ and $F^{\fracci np-n+1}_{p,r}$
with $q\le1$ and $p<1$ seems to be undetermined yet.
\end{rem}

\subsection{Extension by zero}  \label{ext-ssect}
For a function $f(x)\in L_2(\Rn_+)$ there is an extension by zero to the 
whole space, for example $e^+f(x)=\kd{x_n>0}f(x',|x_n|)$. 

The boundedness of $e^+\colon L_2(\Rn_+)\to L_2(\Rn)$ and the properties that
 \begin{equation}
 r^+e^+=1 \quad\text{ and}\quad r^-e^+=0
 \label{ext1}
 \end{equation}
extend to spaces with parameters $(s,p,q)$ in a whole region around
$(0,2,2)$.

In fact, when $(s,p,q)$ satisfies 
 \begin{equation}
 \max(\fracp-1,\fracc np-n)< s< \fracp,
 \label{ext2}
 \end{equation}
the operator $e^+$ can be given a sense as a bounded operator
 \begin{equation}
 e^+\colon B^{s}_{p,q}(\Rp^n)\to B^{s}_{p,q}(\Rn), \qquad
 e^+\colon F^{s}_{p,q}(\Rp^n)\to F^{s}_{p,q}(\Rn)
 \label{ext3} 
 \end{equation}
(when $p<\infty$ in the $F$ case), which has the properties in 
\eqref{ext1}.

\bigskip

For this it is convenient to use the product $\pi(u,v)$ defined in
\cite{JJ94mlt} for $u$ and $v\in\cal S'(\Rn)$ as 
 \begin{equation}
 \pi(u,v)=\lim_{k\to\infty}\cal F^{-1}(\psi(2^{-k}\cdot)\hat u)
                           \cal F^{-1}(\psi(2^{-k}\cdot)\hat v)
 \label{ext4}
 \end{equation}
when the limit exists in the w$^*$-topology on $\cal D'(\Rn)$ for each
$\psi\in C^\infty_0(\Rn)$ that equals $1$ on a neighbourhood of the
origin. Here the limit is required to be independent of $\psi$.

Now, for $u$ in $B^{s}_{p,q}(\Rp^n)$ or $F^{s}_{p,q}(\Rp^n)$ with \eqref{ext2}
satisfied by $(s,p,q)$ one can {\em define\/}, with $\chi(x)=\kd{x_n>0}$,
 \begin{equation}
 e^+u=\pi(\chi,v)\quad\text{when}\quad r^+v=u
 \label{ext5}
 \end{equation}
for $v\in B^{s}_{p,q}(\Rn)$ and $v\in F^{s}_{p,q}(\Rn)$ respectively.

Note that $v\in L_t(\Rn)$ when $\frac nt=\fracc np-s$ holds for $s>0$ in
addition to \eqref{ext2}. Then
\cite[Prop.~3.8]{JJ94mlt} gives that $\pi(\chi,v)=\chi v$, so $e^+u$
has the usual meaning. Moreover, it was proved in
\cite[Thm.~2.8.7]{T2} and \cite[Cor.~3.4.6]{F3} that 
 \begin{equation}
 \pi(\chi,\cdot)\colon B^{s}_{p,q}(\Rn)\to B^{s}_{p,q}(\Rn)
 \quad\text{and}\quad
 \pi(\chi,\cdot)\colon F^{s}_{p,q}(\Rn)\to F^{s}_{p,q}(\Rn)
  \label{ext7} 
 \end{equation}
are bounded when \eqref{ext2} holds. By taking the infimum over $v$ in
\eqref{ext5} the boundedness of $e^+$ follows.

However, \eqref{ext5} needs to be justified. So, if $v_1$ belongs to
the same space as $v$ and $r^+v_1=u$, then $\pi(\chi,v)$ and
$\pi(\chi,v_1)$ coincide except at $\{\,x\mid x_n=0\,\}$. For with 
$w=v-v_1$ at least one of the factors in $\pi(\chi,w)$ is $0$ in $\Rn_-$ 
and in $\Rn_+$, so that $r^\pm\pi(\chi,w)=0$ by \cite[Prop.~3.7]{JJ94mlt}.

For $s>0$ or otherwise when $B^{s}_{p,q}$ and $F^{s}_{p,q}$ only contain 
functions, $\pi(\chi,w)=0$ is necessary. When $s\le0$ this conclusion is
obtained from the inequality $s>\fracp-1$ by duality: from the definition of 
$(e^+)^*\colon B^{-s}_{p',q'}(\Rp^n)\to B^{-s}_{p',q';0}(\Rp^n)$ 
(where $e^+$ refers to one of the already covered cases with 
$0\le s<\fracp$) it follows when $1<q'<\infty$ that
 \begin{equation}
 \dual{(e^+)^*r^+\psi}{\overline{r^+f}}=\int_{\Rn} \psi\chi\overline{f}
 =\dual{\chi\psi}{\overline{r^+f}}
 \label{ext10}
 \end{equation}
when $\psi\in C^\infty(\Rn)\cap B^{-s}_{p',q'}$ and 
$f\in B^{s}_{p,q}(\Rn)$, so by closure
the identity $(e^+)^*r^+\psi=\pi(\chi,\psi)$ holds for every $\psi\in
B^{-s}_{p',q'}(\Rn)$. This shows that $\pi(\chi,v)$ also in these
cases only depends on $r^+v$. The cases with $q'=\infty$ are covered
by simple embeddings.

Finally it should be mentioned that the second part of \eqref{ext1} is
a direct consequence of \cite[Prop.~3.7]{JJ94mlt}. The first part
also follows from this when combined with \cite[Prop.~3.6]{JJ94mlt}:
$r^+e^+u=r^+\pi(\chi,v)=r^+\pi(1,v)=r^+v$. Altogether the desired properties
of $e^+$ as defined in \eqref{ext5} has been obtained.

\begin{rem} \label{ext-rem}
The extension operator $e^+$ has been defined with care above, albeit a
definition as a self-adjoint operator by \eqref{ext10} is simpler. The present
definition is more flexible, however, for it allows an analysis by means of
para-multiplication, which is crucial for the proof of Theorem~\ref{psdo-thm}
below. 
\end{rem}

\subsection{Interpolation}  \label{intp-ssect}
For the proof of Theorem~\ref{psdo-thm} below it is necessary to have
interpolation available. In addition to the real method, described in
\cite[2.4.2]{T2}, properties similar to those of the complex method are
needed. 

Here it is on one hand well known that the usual complex interpolation method
due to A.~P.~Calder\'on \cite{Cr} does not extend to quasi-Banach spaces.
On the other hand, the so-called $+-$-method may serve as a substitute,
as was pointed out to me by both W.~Sickel and J.~Marshall. 

The $+-$-interpolation of two quasi-Banach spaces $A_0$ and
$A_1$\,---\,both lying inside some Hausdorff topological vector space
$X$\,---\,is defined by J.~Gustavsson and J.~Peetre \cite{GuPe}, and
it is usually denoted by $\langle A_0, A_1, \rho\rangle$. Here the
function $\rho$ will be $t\mapsto t^\theta$ for some
$\theta\in\,]0,1[$, and $(A_0, A_1)_{\pm, \theta}:=\langle A_0, A_1,
t^\theta\rangle$ in order to avoid confusion with dualities. In
general $(A_0, A_1)_{\pm, \theta}$ is a quasi-Banach space.

Moreover, the interpolation property was proved in \cite{GuPe}. That
is to say, when $T$ is a linear operator defined on $X$ (or a subspace) 
such that
 \begin{equation}
 T\colon A_j\to B_j
 \label{intp1}
 \end{equation}
is bounded for $j=0$ and $1$, it is so also for $j=2$, when $A_2=(A_0,
A_1)_{\pm, \theta}$ and $B_2=(B_0, B_1)_{\pm, \theta}$. Here $\|
T\|_2\le \| T\|_0^{1-\theta}\| T\|_1^{\theta}$ holds for the operator
quasi-norms.

For the $F^{s}_{p,q}$ scale it was shown by Frazier and Jawerth
\cite{FJ2} that 
 \begin{equation}
 (F^{s_0}_{p_0,q_0}(\Rn), F^{s_1}_{p_1,q_1}(\Rn))_{\pm, \theta}=
  F^{s_2}_{p_2,q_2}(\Rn)
 \label{intp2}
 \end{equation} 
for each $\theta\in\,]0,1[\,$ and any admissible parameters provided
 \begin{equation}
 \begin{gathered}
 s_2=(1-\theta)s_0+\theta s_1;  \\
 \fracc1{p_2}=\fracc{1-\theta}{p_0}+\fracc\theta{p_1}, \qquad
 \fracc1{q_2}=\fracc{1-\theta}{q_0}+\fracc\theta{q_1}.
 \end{gathered}
 \label{intp3}
 \end{equation}
This result is also valid for open sets $\Omega\subset\Rn$, when they
have the extension property. This means that for each $N\in\N$ there
exists an operator $\ell_\Omega$, which is bounded
 \begin{equation}
 \ell_\Omega\colon B^{s}_{p,q}(\overline{\Omega})\to B^{s}_{p,q}(\Rn),\qquad
 \ell_\Omega\colon F^{s}_{p,q}(\overline{\Omega})\to F^{s}_{p,q}(\Rn)
 \label{intp4}
 \end{equation}
for $|s|<N$ and $\frac1N<p,q\le\infty$, and for which the composite
 \begin{equation}
 F^{s}_{p,q}(\overline{\Omega})\xrightarrow{\;\ell_\Omega\;}
 F^{s}_{p,q}(\Rn)\xrightarrow{\;r_\Omega\;}
 F^{s}_{p,q}(\overline{\Omega})
 \label{intp6}
 \end{equation}
equals the identity, with a similar property for the Besov spaces.

The formulae \eqref{intp2} and \eqref{intp6} and the
interpolation property now give

\begin{prop} \label{intp-prop}
Let $\Omega\subset\Rn$ be an open set with the extension property,
let $\theta\in\,]0,1[$ and let $s_j\in\R$, $p_j\in\,]0,\infty]$ and
$q_j\in\,]0,\infty]$ for $j=0$ and $1$.

When $(s_2,p_2,q_2)$ satisfies \eqref{intp3}, then
 \begin{equation}
 F^{\Mi,s_2}_{p_2,q_2}(\overline{\Omega})=
 (F^{\Mi,s_0}_{p_0,q_0}(\overline{\Omega}),
 F^{\Mi,s_1}_{p_1,q_1}(\overline{\Omega}))_{\pm, \theta}
 \label{intp5}
 \end{equation}
holds with equivalent quasi-norms.
\end{prop}

That $\Omega$ has the extension property when it is bounded and
$C^\infty$ smooth or $\Omega=\Rn_+$ was proved in \cite{F3} (and with
some restrictions for the $F$ case also in \cite{T2}).
By the general result in \cite[Prop.~6.1]{GuPe} the interpolation 
property holds for the
$F^{s}_{p,q}(\overline{\Omega})$ spaces too.

\begin{rem} \label{intp-rem}
It deserves to be mentioned, that the $B^{s}_{p,q}(\Rn)$ and
$F^{s}_{p,q}(\Rn)$ scales are invariant under a complex interpolation
based on $\cal S'$-analytical functions,\linebreak[5] cf.~\cite{T2}. However, for
this method the interpolation property has only been verified for
$\max(p_2,q_2)<\infty$ in \cite{F2}, and under the assumption that
there is continuity from, say, $B^{s_2-\varepsilon}_{p_2,q_2}$ to
$\cal S'(\Rn)$ the case $q_2=\infty$ was included there too. 

An overview of this is contained in \cite{JJ93}, even with a removal of
the restriction to $p_2<\infty$. Although this approach works equally
well for the application in \cite{F2}, and thus in the present paper
too, the $+-$-interpolation is preferred here because of the available
references. Ultimately the proofs are also more structured and less
technical, then.
\end{rem}

\section{Operators on $\Rp^n$} \label{Rnp-sect}

To begin with the operators are defined on the spaces $\cal S(\Rp^n)$ and $\cal
S(\R^{n-1})$. More general spaces are introduced afterwards in
Section~\ref{cont-sect}. 

Since the inclusion of the $B^{s}_{\infty,q}$ spaces requires the 
Definitions~\ref{trac-defn} and \ref{sgo-defn} of the operators 
(because $\cal S$ is not dense there), the 
exposition in Sections~\ref{Rnp-sect} and \ref{cont-sect} is 
intended to be fairly detailed. 

In particular proofs are given for
Propositions~\ref{pois-prop}--\ref{sgo-prop}, albeit the contents are
essentially known. However, none of the references apply directly, and
at least the presented proofs should be of interest in view of their
elementary nature.
 
For a general introduction to the Boutet de~Monvel calculus the reader
is referred to the exposition in \cite{G2} and to Section~1.1 ff.\ in
\cite{G1}. 

\subsection{Review of the operators} \label{revw-ssect}
Recall that a truncated pseudo-differential operator $P_{+}$, a
Poisson operator $K$, a trace operator $T$ and a singular Green operator 
$G$, cf.~\eqref{i2} ff., act in the following way on $u\in\cal
S(\Rp^n)$ and $v\in\cal S(\R^{n-1})$\,---\,when $T$ and $G$ are 
of {\em  class zero\/}:
 \begin{align}
   P_{@!@!+}u(x)&=r^+(2\pi)^{-n}\!\smash[t]{\int_{\Rn}} e^{ix\cdot\xi}p(x,\xi)
                                      \widehat{e^+u}(\xi)\,d\xi,
 \label{-20} \\
 Kv(x)&=(2\pi)^{1-n}\!\int_{\R^{n-1}} e^{ix'\cdot\xi'}\tilde k(x',x_n,\xi')
                                      \hat v(\xi')\,d\xi',
 \label{-21} \\
 Tu(x')&=(2\pi)^{1-n}\!\int_{\R^{n}_+}\!
               e^{ix'\cdot\xi'}\tilde t(x',y_n,\xi')
                                      \acute u(\xi',y_n)\,dy_nd\xi',
 \label{-22} \\
 Gu(x)&=(2\pi)^{1-n}\!\int_{\R^{n}_+}\!
               e^{ix'\cdot\xi'}\tilde g(x',x_n,y_n,\xi')
                                      \acute u(\xi',y_n)\,dy_nd\xi',
 \label{-23}
 \end{align}
The fifth kind of 
operators in the calculus are the pseudo-differential operators $S$ 
acting on $v\in\cal S(\R^{n-1})$ in the usual way, cf.~\eqref{-18}.
The definition of class $r\in\Z$ of $T$, $G$ and $P_++G$ is recalled 
in Subsections~\ref{trac-ssect}--\ref{PG-ssect} below. 

For $P_{+}$ the uniform two sided transmission condition will be 
employed to assure that $P_{+}u$ belongs to $C^\infty(\Rp^n)$ when 
$u\in\cal S(\Rp^n)$, see \cite{GK} and \cite{GH} for a
discussion of this condition. 

The starting point is the uniform class
$S^{d}_{1,0}(\Rn\times\Rn)$ given with the seminorms\linebreak[4]
$\norm{p}{S^d_{1,0},\alpha,\beta}:=C_{\alpha,\beta}$ in \eqref{i15}
above. While the symbol of $S$ is taken in
$S^d_{1,0}(\R^{n-1}\times\R^{n-1})$, that of $P$ is required to belong to
$S^d_{1,0,\op{uttr}}(\Rn\times\Rn)$:

\begin{defn} \label{uttr-defn} 
For $d\in\R$ the space $S^d_{1,0,\op{uttr}}(\Rn\times\Rn)$ consists of the
symbols $p(x,\xi)\in S^{d}_{1,0}(\Rn\times\Rn)$ satisfying the {\em uniform
two-sided\/} transmission condition (at $x_n=0$), i.e.,
for every $\alpha,\beta\in\N^n_0$ and $l,m\in\N_0$ the condition
 \begin{equation}
 C_{\alpha,\beta,l,m}(p):=\sup
 |z^l_nD^m_{z_n}\cal F^{-1}_{\xi_n\to z_n}D^\alpha_\xi D^\beta_x
 p(x',0,\xi)| <\infty
 \label{-2}
 \end{equation} 
holds for each $\xi'$ when the supremum is taken over
$(x',z_n)\in\R^n\setminus\{\,z_n=0\,\}$.
\end{defn}

In formulae \eqref{-21} and \eqref{-23} the symbol-kernels
$\tilde k$ and $\tilde g$ can belong to the uniform spaces
$S^{d-1}(\R^{n-1}\times\R^{n-1},\cal S(\Rp))$ and $S^{d-1}(\R^{n-1}
\times\R^{n-1},\cal S(\Rpp))$ respectively. This means that for all indices 
$\alpha'$ and $\beta'\in\N^{n-1}_0$ and $l,m,l'$ and $m'\in\N_0$  
the following seminorms are finite:
 \begin{gather}
   \norm{\tilde k}{S^{d-1}_{1,0},\alpha',\beta',l,m} \,:=\,
 \sup\ang{\xi'}^{-(d-|\alpha'|-l+m)}|x^l_nD^m_{x_n}D^{\alpha'}_{\xi'} 
        D^{\beta'}_{x'} \tilde k(x',x_n,\xi')|,
 \label{-24} \\
{ \aligned
 &\norm{\tilde g}{S^{d-1}_{1,0},\alpha',\beta',l,m,l'm'}
   :=\,\sup\bigl(\ang{\xi'}^{-(d+1-|\alpha'|-l+m-l'+m')}\times
 \\
 &\hphantom{ \norm{\tilde g}{S^{d-1}_{1,0},\alpha',\beta',l,m,l'm'}:= \sup\bigl((}
 |x^l_nD^m_{x_n}y^{l'}_nD^{m'}_{y_n}
       D^{\alpha'}_{\xi'}D^{\beta'}_{x'} \tilde g(x',x_n,y_n,\xi')|\bigr),    
  \endaligned} 
 \label{-25}
 \end{gather}
when the supremum is taken over 
$(x',x_n,\xi')\in\R^{n-1}\times\Rp\times\R^{n-1}$ respectively over\linebreak
$(x',x_n,y_n,\xi')$ in $\R^{n-1}\times\Rp\times\Rp\times\R^{n-1}$. 
The symbol-kernel $\tilde t$ is usually taken in\linebreak
$S^{d}_{1,0}(\R^{n-1}\times\R^{n-1},\cal S(\Rp))$ (yet here the normal
variable is integrated out, and hence denoted by $y_n$, cf.~\eqref{-17}
below). 

Occasionally we shall use the equivalent family of seminorms
 \begin{equation}
 \norm{p}{S^{d}_{1,0},k}=\max\bigl\{\,\norm{p}{S^{d}_{1,0},\alpha,\beta}
 \bigm| |\alpha|,|\beta|\le k\,\bigr\},\quad k\in\N_0
 \label{-1'} 
 \end{equation}
When the meaning is clear the symbol space is suppressed, i.e.\ $\norm{p}{k}:=
\norm{p}{S^{d}_{1,0},k}$, and instead of $S^{d}_{1,0}(\Rn\times\Rn)$
we write $S^{d}_{1,0}$ and $S^{-\infty}:=\cap_d S^{d}_{1,0}$. 
Similar abbreviations are used for the symbol-kernel spaces. 
Endowed with the topology of the introduced systems of
seminorms, $S^{d}_{1,0}(\Rn\times\Rn)$,
$S^{d}_{1,0}(\R^{n-1}\times\R^{n-1},\cal S(\Rp))$ and
$S^{d}_{1,0}(\R^{n-1}\times\R^{n-1},\cal S(\Rpp))$ are Fr\'echet spaces.

\bigskip

With the symbol-kernels belonging to the indicated spaces it is 
seen at once that the integrals in \eqref{-21}--\eqref{-23} above are 
convergent, and hence $Kv$, $Tu$ and $Gu$ are well defined:

\begin{prop} \label{S-prop}
Let $\tilde k\in S^{d-1}_{1,0}(\cal S(\Rp))$ with
$d\in\R$, and define for $v(x')$ in $\cal S(\R^{n-1})$ the function
$Kv(x',x_n)=\op{OPK}(\tilde k)v$ by the formula \eqref{-21}.

Then $(\tilde k,v)\mapsto Kv=\op{OPK}(\tilde k)v$
is continuous as a mapping
 \begin{equation}
 S^{d-1}_{1,0}(\cal S(\Rp))\times\cal S(\Rp^n)\xrightarrow{\;\op{OPK}\;}\cal
 S(\Rp^n).
 \label{S1}
 \end{equation}
Similarly the mappings 
 \begin{align}
 S^{d}_{1,0}(\cal S(\Rp))\times\cal S(\Rp^n)&\xrightarrow{\;\op{OPT}\;}\cal S(\Rp^n)
 \label{S2} \\
 S^{d-1}_{1,0}(\cal S(\overline{\R}^2_{++}))\times\cal S(\Rp^n)
  & \xrightarrow{\;\op{OPG}\;}\cal S(\Rp^n)
 \label{S3}
 \end{align}
defined by \eqref{-22} and \eqref{-23} are continuous.
\end{prop}

\begin{proof} By use of \eqref{-24}, $\op{OPK}(\tilde k)v$
is in $\cal S(\Rp^n)$: for any multiindices $\alpha$ and $\beta\in\N_0^n$ the
seminorm $\sup\{\,x^\alpha D^\beta_x Kv\mid x\in\Rp^n\,\}$ is
finite and dominated by
 \begin{multline}
   (2\pi)^{1-n}\smash[b]{ \sum_{\omega'\le\alpha'}\sum_{\gamma'\le\beta'} }
 \textstyle{ \binom{\alpha'}{\omega'}\binom{\beta'}{\gamma'} }
 \norm{\tilde k}
 {S^{d-1}_{1,0},\alpha'-\omega',\beta'-\gamma',\alpha_n,\beta_n}
 \\ \times\norm{(1\mlap)^Nx^{\omega'}D^{\gamma'}v}{L_1}
 \int\ang{\xi'}^{d-|\alpha'-\omega'|-\alpha_n+\beta_n-2N}\,d\xi',
 \label{5} 
 \end{multline}
when $N$ is so large that
$d-|\alpha'-\omega'|-\alpha_n+\beta_n-2N<-(n-1)$.

$\op{OPT}$ and $\op{OPG}$ can be treated in a similar fashion.
\end{proof}

Contrary to this, the formula \eqref{-20} does not make sense for every
$u\in\cal S(\Rp^n)$ as it stands, so it should rather be read as $P_{+}u=
r^+\op{OP}(p)e^+u$, where
 \begin{equation}
 \op{OP}(p(x,\xi))\psi=(2\pi)^{-n}\!\int\!e^{ix\cdot\xi}p(x,\xi)
 \hat\psi(\xi)\,d\xi,\quad\text{ for}\quad\psi\in\cal S(\Rn).
 \label{-18}
 \end{equation}
Then $P_{+}u$ is well defined in view of \eqref{i14}. More precisely:

When $P=\op{OP}(p)$ for $p(x,\xi)\in S^{d}_{1,0}$, direct 
calculations show the continuity of\linebreak $P\colon \cal S(\Rn)\to\cal S(\Rn)$.
Since, by consideration of the sesqui-linear duality
$\dual{u}{\overline{\varphi}}$ for $u\in\cal S'(\Rn)$ and
$\varphi\in\cal S(\Rn)$,
 \begin{equation}
 \op{OP}(p(x,\xi))^*=\op{OP}(q(x,\xi)),\quad
 \text{with}\quad  q(x,\xi)=e^{iD_x\cdot D_\xi}\overline{p}(x,\xi),
 \label{-16}
 \end{equation} 
the continuity of $P\colon \cal S'(\Rn)\to\cal S'(\Rn)$ follows, cf.\  
\cite[Sect.\ 18.1]{H}. Here $e^{iD_x\cdot D_\xi}$ is a homeomorphism
on $S^{d}_{1,0}(\Rn\times\Rn)$. 

Recall that $P$ has the boundedness properties in \eqref{i14}, which
in particular apply when $p(x,\xi)$ belongs to the subclass
$S^d_{1,0,\op{uttr}}$. 

For $p(x,\xi)\in S^{d}_{1,0}$ it follows that $P_+=r^+Pe^+$ is bounded from 
$L_2(\Rn_+)$ to $F^{-d}_{2,2}(\Rp^n)$, so in particular $P_+u$ is defined 
for $u\in\cal S(\Rp^n)$. When in addition $p\in
S^{d}_{1,0,\op{uttr}}$, one has $P_+u\in\cal S(\Rp^n)$ then,
cf.~Proposition~\ref{psdo-prop} below. (The result there supplements
Proposition~\ref{S-prop}).  

Moreover, letting $\op{OP}(q(x',y_n,\xi))u:=
(2\pi)^{-n}\iint e^{i(x-y)\cdot\xi}q(x',y_n,\xi)u(y)\,dyd\xi$
for $u\in\cal S(\Rn)$, the technique in \eqref{-16} shows that
 \begin{equation}
 P=\op{OP}(q(x',y_n,\xi))
 \quad \text{for}\quad  q(x',y_n,\xi)=
 e^{-iD_{x_n}\cdot D_{\xi_n}}p(x,\xi)\!\bigm|_{x_n=y_n},
 \label{-17}
 \end{equation}
and $P$ is then said to be given in $(x',y_n)$-form. It is also known,
cf.\  \cite{G3}, that $p\in S^{d}_{1,0,\op{uttr}}$ 
implies that $q(x',y_n,\xi)\in S^{d}_{1,0,\op{uttr}}$, i.e., 
\eqref{-2} holds when $p(x',0,\xi)$ is replaced by $q(x',0,\xi)$.

For the symbol-kernel spaces one has results analogous to \eqref{-16}
above, and they follow from the pseudo-differential case by freezing 
$x_n$ and $y_n$:

\begin{lem} \label{sk-lem}
Let $\tilde k\in S^{d_1-1}_{1,0}(\cal S(\Rp))$,
$\tilde t\in S^{d_2}_{1,0}(\cal S(\Rp))$ and $\tilde g\in 
S^{d-1}_{1,0}(\Rpp)$, and let there be defined symbol-kernels by
 \begin{align}
 \tilde k^*(x'\!,x_n,\xi')&=e^{iD_{x'}\cdot D_{\xi'}}
  \overline{\tilde k}(x'\!,x_n,\xi')
 \label{sk1} \\
 \tilde g^*(x'\!,x_n,y_n,\xi')&=e^{iD_{x'}\cdot D_{\xi'}}
  \overline{\tilde g}(x'\!,y_n,x_n,\xi')
 \label{sk2} \\
 \tilde k\circ\tilde t(x'\!,x_n,y_n,\xi')&=
 e^{iD_{y'}\cdot D_{\eta'}}\tilde k(x'\!,x_n,\eta')
  \tilde t(y'\!,y_n,\xi')\!\!\bigm|_{\text{\rlap{$y'=x',\, \eta'=\xi'$}}}
 \hphantom{y'=x',\,}
 \label{sk3}
 \end{align} 
The mappings $\tilde k\mapsto\tilde k^*$ and 
$\tilde g\mapsto\tilde g^*$ define homeomorphisms on
$S^{d_1-1}_{1,0}(\cal S(\Rp))$ and\linebreak $S^{d-1}_{1,0}(\cal S(\Rpp))$ respectively,
and the bilinear mapping given by
$(\tilde k,\tilde t)\mapsto \tilde k\circ\tilde t$ is continuous from
$S^{d_1-1}_{1,0}(\cal S(\Rp))\times S^{d_2}_{1,0}(\cal S(\Rp))$
to $S^{d_1+d_2-1}_{1,0}(\cal S(\Rpp))$. 

In particular, for each $j\in\N$ there is a constant $c$ and a $j'$ such that
 \begin{align}
 \norm{\tilde k^*}{S^{d_1-1}_{1,0}(\cal S(\Rp)),j}&\le
 c\norm{\tilde k}{S^{d_1-1}_{1,0}(\cal S(\Rp)),j'}
 \label{sk5} \\
 \norm{\tilde g^*}{S^{d-1}_{1,0}(\cal S(\Rpp)),j}&\le
 c\norm{\tilde g}{S^{d-1}_{1,0}(\cal S(\Rpp)),j'}
 \label{sk6} \\
 \norm{\tilde k\circ\tilde t}{S^{d_1+d_2-1}_{1,0}(\cal S(\Rpp)),j}&\le
 c\norm{\tilde k}{S^{d_1-1}_{1,0}(\cal S(\Rp)),j'}
 \label{sk4} \\[-1\jot]
 &\hphantom{\le c\,\|\tilde k|}
   \smash[b]{ \times\norm{\tilde t}{S^{d_2}_{1,0}(\cal S(\Rp)),j'} }
 \notag
 \end{align}
hold for every $\tilde k$, $\tilde t$ and $\tilde g$ in the
considered spaces. 
\end{lem}
\begin{proof} Obviously $x^l_nD^m_{x_n}\tilde k(\cdot,x_n,\cdot)
\in S^{d_1-l+m}_{1,0}(\R^{n-1}\times\R^{n-1})$ for each $x_n$, and therefore\linebreak
$x^l_nD^m_{x_n}\tilde k^*=e^{iD_{x'}\cdot D_{\xi'}}
\overline{x^l_n(-D_{x_n})^m\tilde k}$ belongs to 
$S^{d_1-l+m}_{1,0}$. Moreover, for each $|\alpha'|$, $|\beta'|$,
$l$, and $m$ there exist $c$ and $N\ge|\alpha'|,|\beta'|$ such that,
with $N'=\max(N,l,m)$,
 \begin{equation}
 \begin{aligned}
 \ang{\xi'}^{-(d_1-l+m-|\alpha'|)}|D^{\beta'}_{x'}D^{\alpha'}_{\xi'}
 x^l_nD^m_{x_n}\tilde k^*|&\le c\norm{x^l_nD^m_{x_n}
 \tilde k(\cdot,x_n,\cdot)}{S^{d_1-l+m}_{1,0},N}
 \\
 &\le c\norm{\tilde k}{S^{d_1-1}_{1,0}(\cal S(\Rp)),N'}.
 \end{aligned}
 \label{sk7}
 \end{equation}
The statements on $\tilde g^*$ and $\tilde k\circ\tilde t$ carry
over from the pseudo-differential case in the same manner.
\end{proof}

\subsection{The transmission condition}  \label{uttr-ssect}
The requirement of the uniform two-sided transmission condition 
in \eqref{-2} is not as innocent as it looks,
with a seemingly arbitrary $\xi'$ dependence of
$C_{\alpha,\beta,l,m}$: Indeed, \eqref{-2} is equivalent to a rather 
special $\xi$-dependence of $p(x,\xi)$, cf.\ (ii) in Proposition~%
\ref{uttr-prop} below. Furthermore, there is also
equivalence with the condition (iii) below, that
implies a slowly increasing behaviour of $C_{\alpha,\beta,l,m}(\xi')$.

\begin{prop} \label{uttr-prop} 
When $p(x,\xi)\in S^d_{1,0}(\Rn\times\Rn)$ 
for $d\in\R$, the following conditions on $p(x,\xi)$ are equivalent:
 \begin{itemize}
 \item[(i)] $p\in S^d_{1,0,\op{uttr}}(\Rn\times\Rn)$, 

 \item[(ii)]For all $\alpha$ and $\beta\in \N^n_0$ there exist 
$s_{j,\alpha,\beta}(x',\xi')\in 
S^{d-j-|\alpha|}_{1,0}(\R^{n-1}\times\R^{n-1})$, for
$j\in \Z$ with $j\le d-|\alpha|$, such that for every $m\in\N_0$
 \begin{equation}
 |\xi_n^mD^\alpha_\xi D^\beta_x p(x',0,\xi)-
 \sum_{-m\le j\le d-|\alpha|} s_{j,\alpha,\beta}(x',\xi')\xi^{j+m}_n| 
 \le C\ang{\xi'}^{d+1-|\alpha|+m}\ang\xi^{-1}
 \label{-3}
 \end{equation}
holds with a constant $C$ independent of
$(x',\xi)\in\R^{n-1}\times\Rn$.

 \item[(iii)] For all $\alpha,\beta\in\N^n_0$ and 
$l,m\in\N_0$ the symbol $p(x,\xi)$ satisfies
 \begin{equation}
 \sup\,\ang{\xi'}^{-(d+1-|\alpha|-l+m)}
 |z^l_nD^m_{z_n}\cal F^{-1}_{\xi_n\to z_n}D^\alpha_\xi D^\beta_x
 p(x',0,\xi)|\,<\,\infty,
 \label{-4} 
 \end{equation}
when the supremum is taken over $x'$ and $\xi'$ in $\R^{n-1}$ and $z_n\ne0$.
 \end{itemize} 
In the affirmative case, the symbols $s_{j,\alpha,\beta}(x',\xi')$
are uniquely determined, and they are polynomials in 
$\xi'\in\R^{n-1}$ of degree $\le d-j-|\alpha|$.
\end{prop}

Here and in the following $C$ denotes a `global' constant 
(independent of variables like $x$ and $\xi$), while $c$ is a
`local' constant (that might depend on $x$, say). The constants may
differ on each occurrence, as usual.

In the rest of this section $e^+r^++e^-r^-$ is denoted $\erd$,
where $\overset{\lower 1pt\hbox{\large.}}r$ stands for restriction
to the~set $\R\setminus\{\,0\,\}$. One has $\erd=\cal F^{-1}h_{-1}\cal
F$ on $\overset{\lower 1pt\hbox{\,\large.}}{\cal S}(\R)$,      
when $h_{-1}$ denotes the projection of $\cal H$ onto $\cal
H_{-1}$. See \cite[Sect.\ 2.2]{G1} where this terminology, that
is used in the following without further mention, is explained.

\begin{proof} It is obvious that $\text{(iii)}\Rightarrow\text{(i)}$,
and $\text{(ii)}\Rightarrow\text{(iii)}$ follows by use of the
Parseval--Plancherel identity together with an 
application of the inequality
 \begin{equation}
 \sup_t |f(t)|\le\sqrt2\norm{f}{L_2}^{\frac12}
 \norm{ \erd D_tf}{L_2}^{\frac12},
 \label{-5}
 \end{equation}
valid for functions $f\in e^+W^1_2(\Rp)+e^-W^1_2(\overline{\R}_-)$,
to the function defined for each $(x',\xi')$ as 
$f(z_n)=\erd z^l_nD^m_{z_n}\cal F^{-1}_{\xi_n\to z_n}
D^\alpha_\xi D^\beta_x p(x',0,\xi)$. Indeed, for $\norm{f}{L_2}$ in
\eqref{-5} one finds
 \begin{align}
 \norm{ f(z_n) }{L_{2}(\R)}=&\norm{ h_{-1}\overline D^l_{\xi_n}\xi^m_n
      D^\alpha_\xi D^\beta_x p(x',0,\xi) }{L_{2}(\R)}
 \label{-6} \\
 \le\,&C\ang{\xi'}^{d+1-|\alpha|-l+m}\norm{ \ang\xi^{-1}}{L_{2}(\R)}
 \le\, C'\ang{\xi'}^{d+\frac{1}{2}-|\alpha|-l+m}
 \notag
 \end{align}
when \eqref{-3} is applied after Leibniz' rule. An estimate of 
$\erd D_{z_n}f(z_n)$ can be derived from \eqref{-6} with $m+1$ instead of $m$.

In the proof of $\text{(i)}\Rightarrow\text{(ii)}$ one observes first that 
$\cal F^{-1}_{\xi_n\to z_n}D^\alpha_\xi D^\beta_xp(x',0,\xi)$ for each
$(x',\xi')$ belongs to 
$\overset{\lower 1pt\hbox{\,\large.}}{\cal S}(\R)$, since the only
distributions supported by $\{\,z_n=0\,\}$ are the finite linear
combinations of derivatives of $\delta_0(z_n)$. Hence $D^\alpha_\xi 
D^\beta_xp(x',0,\xi)\in\cal H$, i.e., there exist numbers $s_{j,\alpha,
\beta}$ for $j\in\Z$, such that for $|\xi_n|\ge1$ and $l,N\in\N_0$,
 \begin{equation}
 \bigl|D^{l}_{\xi_n}(D^\alpha_\xi D^\beta_xp(x',0,\xi)-
 \sum_{d-|\alpha|-N<j\le d-|\alpha|} s_{j,\alpha,\beta}\xi_n^j)\bigr|
 \le c|\xi_n|^{\bignt d-|\alpha|-l-N}.
 \label{-7}
 \end{equation}
Such numbers are necessarily unique\,---\,and zero for
$j>d-|\alpha|$\,---\,hence {\em functions\/}\linebreak
$s_{j,\alpha,\beta}(x',\xi')$.

\medskip

The construction of the $s_{j,\alpha,\beta}$ is completed, and it 
remains to be shown by a bootstrap-method that they are symbols with 
the desired properties.

From (i) and the well-known fact that (with $\gamma_0^\pm
v=\lim_{z_n\to0_\pm} v(z_n)$) one has 
 \begin{equation}
  s_{-1-k,\alpha,\beta}(x',\xi')=
 -i(\gamma_0^+-\gamma_0^-)D^{k}_{z_n}\cal F^{-1}_{\xi_n\to z_n}
 D^\alpha_\xi D^\beta_xp(x',0,\xi)
 \label{-7'}
 \end{equation} 
for $k\in\N_0$, it follows that
$s_{j,\alpha,\beta}(\cdot,\xi')\in C^\infty(\R^{n-1})$ for each
$\xi'$ when $j<0$. 

The next step is to show that, with $C$ independent of $x'$ and $\xi_n$,
 \begin{equation}
 \bigl|D^l_{\xi_n}(\xi_n^mD^\alpha_\xi D^\beta_xp(x',0,\xi)-
 \sum_{-m\le j\le d-|\alpha|}
 s_{j,\alpha,\beta}(x',\xi')\xi_n^{j+m})\bigr| \le C\ang{\xi_n}^{-1-l}.
 \label{-8}
 \end{equation}
Observe that the left hand side is equal to $D^l_{\xi_n}h_{-1}\xi_n^m
D^\alpha_\xi D^\beta_x p(x',0,\xi)$, which is bounded in $x'$ and 
$\xi_n$ by \eqref{-5} since, e.g., for the $L_2$ norm in $\xi_n$
 \begin{equation}
 \aligned 
 \norm{ D^{l}_{\xi_n}h_{-1}\xi^m_nD^\alpha_\xi 
        D^\beta_xp(x',0,\xi) }{L_2}&= 
 \norm{ \erd z_n^lD^m_{z_n}\cal F^{-1}_{\xi_n\to z_n}
 D^\alpha_\xi D^\beta_xp(x',0,\xi)}{L_2} \\
 &\le\, \norm{ (1+|z_n|)^{-1}}{L_2} \sum_{k=l,l+1} C_{\alpha,\beta,k,m} .
 \endaligned
 \label{-9}
 \end{equation}
Moreover $\xi_n^{l+1}D^l_{\xi_n}h_{-1}\xi_n^mD^\alpha_\xi D^\beta_x 
p(x',0,\xi)$ is bounded with respect to $x'$ and $\xi_n$, since
 \begin{multline}
 \xi_n^{l+1}D^l_{\xi_n}h_{-1}\xi_n^mD^\alpha_\xi D^\beta_xp(x',0,\xi)=
 (-1)^ll! s_{-m-1,\alpha,\beta}(x',\xi') 
 \\
 +h_{-1}\xi_n^{l+1}D^l_{\xi_n}
 \xi_n^{m}D^\alpha_\xi D^\beta_x p(x',0,\xi).
 \label{-10}
 \end{multline}
Hence $(1+|\xi_n|^{l+1})D^l_{\xi_n}h_{-1}\xi_n^mD^\alpha_\xi 
D^\beta_x p(x',0,\xi)$ is bounded, so \eqref{-8} is obtained.

A consequence of \eqref{-8} is that $s_{j,\alpha,\beta}(\cdot,\xi')
\in C^\infty(\R^{n-1})$ for $j\ge0$. Indeed,
 \begin{equation}
 s_{j,\alpha,\beta}(x',\xi')=\tfrac{1}{j!}\partial^j_{\xi_n}(
 D^\alpha_\xi D^\beta_xp(x',0,\xi)-h_{-1}D^\alpha_\xi D^\beta_xp(x',0,\xi))
 \!\bigm|_{\xi_n=0},
 \label{10'}
 \end{equation}
and here the fact that $p\in S^{d}_{1,0}$ can be applied together with
\eqref{-8}.

The rest is similar to \cite[Thm.\ 1.9]{G2}:
 Only the case $\alpha=\beta=0$ will be considered
since $p$ and $d$ can be replaced by $D^\alpha_\xi D^\beta_xp$
and $d-|\alpha|$ in the following. For $d<-m$ there is nothing to
show in \eqref{-3} so $d\ge -m$ is assumed. Let $\gamma':=
(\gamma_1,\dots,\gamma_{n-1},0)$.

At this place the goal is to prove, for $j>-m$ when $m\in\N_0$, 
that with $N=\smlnt{d\,}+1+m$
 \begin{equation}
 s_{j,0,0}(x',\xi')=\sum_{|\gamma'|<N,\ |\gamma'|\le d-j}
 s_{j,\gamma',0}(x',0){\xi}^{\gamma'}.
 \label{-11}
 \end{equation}
For every $j\le d$ the function $s_{j,0,0}(x',\xi')$ would then 
be a polynomial of degree $\bignt d-j$ in 
$\xi'$ with coefficients in $C^\infty(\R^{n-1})$\,---\,i.e.\ 
$s_{j,0,0}\in S^{d-j}_{1,0}$\,---\,so in addition only \eqref{-3}
would still require a proof. 

For \eqref{-3} and \eqref{-11} it suffices to show 
 \begin{equation}
 \bigl|\xi^m_np(x',0,\xi)-
 \smash[t]{
 \sum_{|\gamma'|<N}\sum_{-m\le j\le d-|\gamma'|} } 
 s_{j,\gamma,0}(x',0){\xi}^{\gamma'}\xi_n^{j+m}\bigr|
 \le C\ang{\xi'}^{d+1+m}\ang\xi^{-1},
 \label{-12}
 \end{equation}
for on one hand \eqref{-12} and \eqref{-8} would imply that
 \begin{equation}
 \bigl|\sum_{j=-m}^{\bignt d-|\gamma'|} s_{j,\gamma',0}(x',\xi')\xi_n^{j+m}
 -\sum_{|\gamma'|<N}\sum_{j=-m}^{\bignt d-|\gamma'|} 
 s_{j,\gamma',0}(x',0){\xi}^{\gamma'}\xi_n^{j+m}\bigr|
 \le C\ang{\xi_n}^{-1},
 \label{-13}
 \end{equation}
and here the $\xi_n$-polynomial on the left hand side is identical to
zero precisely when \eqref{-11} holds. On the other hand, \eqref{-12} would then
be the estimate required in \eqref{-3}.
 
When $\ang{\xi_n}\le|\xi'|$ both terms are $\cal O(\ang{\xi'}^{d+1+m}
\ang{\xi}^{-1})$ on the left hand side of \eqref{-12}, since
$\ang{\xi'}\sim\ang{\xi}$ there. In the other region,
$\ang{\xi_n}\ge|\xi'|$, one shows by use of a Taylor expansion, 
cf.~\cite{G2}, the uniform estimate
 \begin{align}
 \bigl|\xi_n^mp(x',0,\xi)-\sum_{|\gamma'|<N} &\partial^{\gamma'}_{\xi'}
 p(x',0,0,\xi_n)\tfrac{{\xi}^{\gamma'}}{\gamma'!}\xi_n^m\bigr|
 \notag \\
&\le\bigl(\sum_{|\gamma'|<N}\tfrac{N}{\gamma'!}
 \norm{p}{\gamma',0}\bigr) |\xi'|^N\ang{\xi_n}^{d-N+m} 
 \notag \\
&\le\,C\ang{\xi'}^{d+1+m}
 \ang{\xi}^{-1}(\tfrac{|\xi'|}{\ang{\xi_n}})^{\smlnt{d\,}-d}.
 \label{-14}
\end{align}
Now \eqref{-14} and \eqref{-8} applied to 
$\xi_n^m\partial^{\gamma'}_{\xi}p(x',0,0,\xi_n)$ lead to \eqref{-12}.
 
It was obtained during the course of the proof that $s_{j,\alpha,\beta}$ is
uniquely determined and is a polynomial of degree 
$\le d-|\alpha|-j$ in $\xi'$ as claimed. The proof is complete. 
\end{proof}

The contents of Proposition~\ref{uttr-prop} are to some extent known. In
fact the equivalence of (i) and (ii) was claimed but not proved in
\cite{GK}, so the proof of \cite[Thm.~1.9]{G2} has been modified into the one 
above with the appropriate uniform estimates.

Note that the essential thing is to show \eqref{-8} and \eqref{-12},
since the proper $x'$ and $\xi'$ behaviour of the $s_{j,\alpha,\beta}$
is a gratis consequence, cf.~\eqref{10'} and \eqref{-11}.

The equivalence with (iii) fits in very naturally, so it seems
reasonable to have the short proof of this available. Indeed, 
(iii) states that $r^+\cal F^{-1}_{\xi_n\to x_n}p(x',0,\xi)$ 
is the symbol-kernel of a Poisson operator of order $d+1$, and this
property is used in Proposition~\ref{pois-prop} below.

\section{Continuity on $\Rp^n$}  \label{cont-sect}
With the preparations made in the section above, the continuity
properties of the operators introduced in \eqref{-20}--\eqref{-23} above
shall now be described. 

\subsection{Poisson operators} \label{pois-ssect}
The treatment of Poisson operators
given here follows the line of thought in \cite{G3}. 
Some observations are collected in the following proposition, where 
the proofs of \eqref{3} and \eqref{4} are intended to be more elementary than
those of the corresponding facts in \cite{G1} and \cite{G3}. 

\begin{prop} \label{pois-prop} 
$1^\circ$ Let $v\in\cal S(\R^{n-1})$ and
$w\in\cal S(\R)$ satisfy $v(0)=1$ together with $\int w(x_n)\,dx_n=1$
and $\supp w\subset\{\,x_n\mid-1\le x_n\le0\,\}$. 

Then it follows for every $\tilde k\in S^d_{1,0}(\cal S(\Rp))$ and $d'>d$ that
 \begin{gather}  
 v(\varepsilon\xi')(w_\varepsilon*_n\tilde k)(x',x_n,\xi')
 \in S^{-\infty}_{1,0}(\cal S(\Rp)), 
 \label{1 }\\
 v(\varepsilon\cdot) w_\varepsilon*_n\tilde k\to \tilde k
 \quad\text{ in $ S^{d'}_{1,0}(\cal S(\Rp))$},
 \label{2}
 \end{gather}
when $w*_n\tilde k(x',x_n,\xi')=r^+\!\int\! e^+\tilde k(x',x_n-y_n,\xi')
w(y_n)\,dy_n$ and $w_\varepsilon(\cdot)=\tfrac1\varepsilon 
w(\tfrac1\varepsilon \cdot)$.

$2^\circ$ When $P=\op{OP}(q(x',y_n,\xi))$ is given in $(x',y_n)$-form
with $q\in S^{d}_{1,0,\op{uttr}}$, then\linebreak $\tilde k=
r^+\cal F^{-1}_{\xi_n\to x_n}q(x',0,\xi)$ in 
$S^{d}_{1,0}(\cal S(\Rp))$ and 
 \begin{equation}
 r^+P(u\otimes\delta_0)=
 \op{OPK}(\tilde k) u
 \quad\text{ holds for}\quad u\in\cal S(\R^{n-1}).
 \label{3}
 \end{equation} 

$3^\circ$ For each $\tilde k\in S^{d-1}_{1,0}(\cal S(\Rp))$ there
exists a $p(x',\xi)\in S^{d-1}_{1,0,\op{uttr}}$ such that
 \begin{equation}
 Kv=r^+\op{OP}(p)(v\otimes\delta_0),\quad\text{ for}
 \quad v\in\cal S(\R^{n-1}).
 \label{4}
 \end{equation}
\end{prop}

\begin{proof} $1^\circ$ The support condition on $w$ implies that 
\begin{equation}
 w(y_n)r^+x_n^lD^m_{x_n}(e^+\tilde k(x',x_n-y_n,\xi'))=
 w(y_n)x_n^lD^m_{x_n}\tilde k(x',x_n-y_n,\xi'),
 \label{5'}
 \end{equation}
so one shows straightforwardly that $\norm{v(\varepsilon\cdot)
w_\varepsilon*_n\tilde k}{S^{-N}_{1,0},\alpha',\beta',l,m}$ is
$<\infty$ for each $N\in\N$.
Now $\norm{\tilde k-v(\varepsilon\cdot)\tilde k}{S^{d'}_{1,0},0}
\le \norm{\tilde k}{S^{d}_{1,0},0}\sup_{\xi'}\ang{\xi'}^{d-d'}
|1-v(\varepsilon\xi')|\to 0$ for $\varepsilon\to 0$, and
 \begin{equation}
 \norm{v(\varepsilon\cdot)(\tilde k-w_\varepsilon*_n
 \tilde k)}{S^{d'}_{1,0},0}\le
 2\norm{\tilde k}{S^{d'}_{1,0},0}
 \cdot \norm{v}{L_\infty}\int_{-\varepsilon}^0|w|,
 \label{6'}
 \end{equation}
so $\norm{\tilde k-v(\varepsilon\cdot)w_\varepsilon*_n
\tilde k}{S^{d'}_{1,0},0} \to 0$ for $\varepsilon\to 0$.
The other seminorms can be handled in a manner similar to this;
for $\alpha'\ne0$ terms with $D^{\gamma'}(v(\varepsilon\cdot))$
obviously $\to0$ for $\varepsilon\to0$.

$2^\circ$
The formula \eqref{3} is first verified for $d=-\infty$, since Fubini's
theorem then permits the following calculation, where $v\in
C^\infty_0(\Rn_+)$, and $w_k\in C^\infty_0(\Rn)$ satisfy
$w_k\to\delta_0$ in $\cal S'$,
 \begin{align}
 \dual{r^+P(u\otimes\delta_0)}{v}&=
      \lim_{k\to\infty}\iiint \!e^{i(x-y)\cdot\xi}q(x',y_n,\xi)
 \notag \\[-2\jot]
 &\hphantom{ =\lim_{k\to\infty}\iiint \!
                     e^{i(x-y)\cdot\xi}q(x', }
 \times u(y')w_k(y_n)e^+v(x)\,dy\dbar\xi dx 
 \notag \\
 &=\,\dual{u(y')}{\!\iint\! e^{i(x'-y')\cdot\xi'}e^{ix_n\xi_n}
          q(x',0,\xi)e^+v\,dx\dbar\xi} 
 \notag \\
 &=\,\dual{\op{OPK}(r^+\cal F^{-1}_{\xi_n\to x_n}q(x',0,\xi'))u}{v};
 \label{6} 
 \end{align}
that $\tilde{\cal K}q:=r^+\cal F^{-1}_{\xi_n\to x_n}q(x',0,\xi')$ is in
$S^{d}_{1,0}(\cal S(\Rp))$ follows from (iii) in Proposition~\ref{uttr-prop}.

For $d\in\R$ the relation \eqref{3} follows from \eqref{6} by regularisation,
since $P(u\otimes\delta_0)$ and $Ku$ depend
continuously on $q$ and $\tilde k$, respectively. 
 
More precisely, take $v$ and $w$ as in $1^\circ$, and define
$q_\varepsilon=v(\varepsilon\xi')\hat w(\varepsilon\xi_n)
q(x',y_n,\xi)$ in $S^{-\infty}_{1,0}$.
Then $q_\varepsilon\to q$ in $S^{d'}_{1,0}$ when $d'>d$ and, 
as verified below, $q_\varepsilon\in S^{d'}_{1,0,\op{uttr}}$ and
$\tilde{\cal K}q_\varepsilon\to\tilde k$ in 
$S^{d'}_{1,0}(\cal S(\Rp))$ for $\varepsilon\to 0$. 
Then \eqref{6} and \eqref{S1} give, with limits taken in $\cal D'(\R^n_+)$, 
 \begin{equation}
 r^+P(u\otimes\delta_0)=
 \lim_{\varepsilon\to0} r^+\op{OP}(q_\varepsilon)(u\otimes\delta_0) 
 =\lim_{\varepsilon\to0}\op{OPK}(\tilde{\cal K}q_\varepsilon)u=Ku.
 \label{7}
 \end{equation}

To show $q_\varepsilon\in S^{d'}_{1,0,\op{uttr}}$, one may write 
$\cal F^{-1}_{\xi_n\to x_n}q(x'\!,0,\xi)$ as $\tilde q(x'\!,x_n,\xi')=
s_{\bignt d}D^{\bignt d}\delta_0\linebreak[4]+\dots+s_0\delta_0+\erd\tilde q$ by
(ii) in Proposition \ref{uttr-prop}. 

Then $r_\pm z_n^lD^m_{z_n}D^{\beta'}_{x'}D^{\alpha'}_{\xi'}
\widetilde{q_\varepsilon} (x',z_n,\xi')$ equals
 \begin{multline}
 r_\pm z_n^lD^m_{z_n}D^{\beta'}_{x'}D^{\alpha'}_{\xi'}
 \sum_{0\le k\le d}s_{k}(x',\xi')v(\varepsilon\xi')
                   D^{k}_{z_n}w_\varepsilon(z_n)
 \\
 +r_\pm{\sum_{\gamma'\le\alpha'}} {\textstyle\binom{\alpha'}{\gamma'}}
 \!\int D^{\alpha'-\gamma'}_{\xi'} (v(\varepsilon\xi'))
 z^lD^{m}_{z_n}(w_\varepsilon(z_n-y_n))\erd D^{\beta'}_{x'}
 D^{\gamma'}_{\xi'}\tilde q(x',y_n,\xi')dy_n, 
 \label  8
 \end{multline}
and using that $q\in S^{d}_{1,0,\op{uttr}}$ majorisations global in 
$(x',z_n)$ can be obtained. 

It remains to show that $\tilde{\cal K}q_\varepsilon\to\tilde k$. 
But $\tilde{\cal K}q_\varepsilon=v(\varepsilon\xi')w_\varepsilon
*_n\tilde{\cal K}q$, since $\supp w \subset[-1,0]$, so $1^\circ$ gives
the rest.

$3^\circ$ To show the existence of $p(x',\xi)$ one can proceed
as in \cite{G3}
by extending $\tilde k(x',x_n,\xi)$ for $x_n<0$ to a function 
$\tilde p(x',x_n,\xi')$ by Seeley's method in \cite{Se}, and let
$p(x',\xi)=\cal F_{x_n\to\xi_n}\tilde p$. It can be checked
that $p\in S^{d-1}_{1,0,\op{uttr}}$, where in
particular the uniform two-sided transmission condition is satisfied
since for each $l$ and $m\in\N_0$ and $\alpha$ and 
$\beta\in\N_0^n$ the functions
 \begin{equation}
 r_\pm z_n^l D^{m}_{z_n}\cal F^{-1}_{\xi_n\to z_n} 
 D^\beta_xD^\alpha_\xi p(x',\xi)=r_\pm z_n^l D^m_{z_n}(-z_n)^{\alpha_n}
 D^\beta_x D^{\alpha'}_{\xi'}\tilde p(x',z_n,\xi')
 \label{9}
 \end{equation} 
are bounded on $\R^{n-1}\times\overline\R_\pm$ for each 
$\xi'$ by the construction of $\tilde p$. \eqref{4} holds by use of
$2^\circ$ since $\tilde{\cal K}p=r^+\cal F^{-1}_{\xi_n\to
x_n}p(x',\xi)=r^+\tilde p(x',x_n,\xi')=\tilde k$.
\end{proof}

Since the composite $r^+P(\cdot\otimes\delta_0)$ is continuous from $\cal
S'(\R^{n-1})$ to $\cal S'(\Rp^n)$ and $\cal S$ is dense in $\cal
S'$ we can obviously make the following

\begin{defn} \label{pois-defn}
For $v\in\cal S'(\R^{n-1})$ the action of a Poisson operator $K$ with
symbol-kernel in $S^{d-1}_{1,0}(\R^{n-1}\times\R^{n-1},\cal S(\Rp))$ 
is defined as $Kv=r^+P(v\otimes\delta_0)$, where $P$ is any
pseudo-differential operator as in $3^\circ$ in Proposition~\ref{pois-prop}.
\end{defn} 

According to its definition $K$ is a continuous operator
 \begin{equation}
 K\colon\cal S'(\R^{n-1})\to\cal S'(\Rp^n).
 \label{S'K}
 \end{equation}
To show that this extended definition of $K$ has good continuity 
properties in the scales of Besov and Triebel--Lizorkin spaces also for $p<1$
one can make use of Proposition~\ref{tensor-prop} concerning the operator 
$f(x')\mapsto f(x')\otimes\delta_0(x_n)$:

\begin{thm} \label{pois-thm} 
Let $K$ be a Poisson operator of order $d\in\R$ and let $s\in\R$
and $p$ and $q\in\,]0,\infty]$. Then the~operator $K$ is bounded
 \begin{align}
 K\colon B^{s}_{p,q}(\R^{n-1})&\to
          B^{s-d+\fracpi}_{p,q}(\overline{\R}^n_+),
 \label{13} \\
        K\colon F^{s}_{p,p}(\R^{n-1})&\to
          F^{s-d+\fracpi}_{p,q}(\overline{\R}^n_+),
 \label{14} 
 \end{align}
when $p<\infty$ holds in \eqref{14}.
\end{thm}

\begin{proof} The symbol-kernel of $K$ is denoted by
$\tilde k(x',x_n,\xi')\in S^{d-1}_{1,0}(\cal S(\Rp))$ and
Definition~\ref{pois-defn} is applied to write
$K=r^+P(\cdot\otimes\delta_0)$ for some $P\in\op{OP} 
(S^{d-1}_{1,0,\op{uttr}})$.

$1^\circ$ For any $s<0$, Proposition~\ref{tensor-prop} and \eqref{i14} give
the boundedness of 
 \begin{alignat}{3}
           &B^{s}_{p,q}(\R^{n-1})&\, \xrightarrow{\;\cdot\otimes\delta_0\;}\,
           &B^{s-1+\fracpi}_{p,q}(\Rn)&\overset P{\,\longrightarrow\,}
           &B^{s-d+\fracpi}_{p,q}(\Rn),
 \label{16} \\
           &B^{s}_{p,p}(\R^{n-1})&\, \xrightarrow{\;\cdot\otimes\delta_0\;}\,
           &F^{s-1+\fracpi}_{p,q}(\Rn)&\overset P{\,\longrightarrow\,}
           &F^{s-d+\fracpi}_{p,q}(\Rn)\quad(p<\infty).
 \label{17}
 \end{alignat}
Hence \eqref{13} and \eqref{14} follow for $s<0$ for every Poisson 
operator $K$.

$2^\circ$ For a given $s\ge 0$ it follows for any $m\in\R$ that 
on $B^{s}_{p,q}(\R^{n-1})$
 \begin{equation}
 K=r^+P((\Xi^{\prime \,(-m)}\cdot)\otimes\delta_0)\Xi^{\prime\, m},
 \label{18}
 \end{equation}
cf.~\eqref{Xi'}. By $1^\circ$, if $m>s\ge 0$ is fixed, it suffices for 
the conclusion of \eqref{13} and \eqref{14}  to show that the operator 
$r^+P(\Xi^{\prime(-m)}\cdot\otimes\delta_0)$ acts on 
$B^{s-m}_{p,q}(\R^{n-1})$ as a Poisson operator $K'$ of order $d-m$.
However, first it is seen from \eqref{-21} that for $v\in\cal S(\R^{n-1})$,
 \begin{equation}
 r^+P(\Xi^{\prime\,(-m)} v\otimes\delta_0)=K\Xi^{\prime\,(-m)}v=
 \op{OPK}(\tilde k(x',x_n,\xi')\ang{\xi'}^{-m}) v=K' v,
 \label{19}
 \end{equation}
where $\tilde k(x,\xi')\ang{\xi'}^{-m}\in S^{d-m-1}_{1,0}(\cal S(\Rp))$. 
Secondly the formula \eqref{19} extends to every $v$ in 
$\cal S'(\R^{n-1})$ by the denseness of $\cal S(\R^{n-1})$.
\end{proof}

The proof above of Theorem \ref{pois-thm} seems to be the first 
to cover the full scales of Besov and Triebel--Lizorkin spaces, since 
the (somewhat different) arguments in \cite{F2} rely on an article 
that has not appeared in Mathematische Nachrichten as announced.
The proof is similar to the~one in \cite{G3}, but in the present
context it is an important point to show that \eqref{19} holds also 
when $\cal S$ is not dense in $B^{s}_{p,q}$. 

Partly for this reason Definition~\ref{pois-defn} and 
Propositions~\ref{uttr-prop} and \ref{pois-prop} are stated explicitly. 
Another step in the above extension of the arguments in \cite{G3} is
to show \eqref{11} and \eqref{12}, since it seems impossible to carry 
through the duality arguments from \cite{G3} for $p<1$ or $q<1$. 

For later reference an observation on the operator norms of $K$ is included.

\begin{cor} \label{pois-cor}
For a Poisson operator $K=\op{OPK}(\tilde k)$ of order $d$ 
the operator norms in \eqref{13} and \eqref{14} satisfy the inequality
 \begin{equation}
 \norm{K}{\Bbb L(B^{s}_{p,q},B^{s-d+\fracpi}_{p,q})}+
 \norm{K}{\Bbb L(F^{s}_{p,p},F^{s-d+\fracpi}_{p,q})}\le
 c\norm{\tilde k}{S^{d-1}_{1,0},j}
 \label{13'} 
 \end{equation}
for some $(s,p,q)$-dependent $c<\infty$ and $j\in\N$
(when the $F$-term is omitted for $p=\infty$).
\end{cor}
\begin{proof}When $s<0$ it is clear from \eqref{16} and \eqref{17} that 
$\|K\|\le c'(s,p,q)\|P\|$ holds for the operator norms.
Here $\|P\|\le c''\norm{p}{S^{d-1}_{1,0},j'}$ when $j'$ is
large enough (depending on $s$), see the formulation of \eqref{i14} in
\cite{Y1}. Since $p$ is a Seeley extension of $\tilde k$, 
$\norm{p}{S^{d-1}_{1,0},j'}\le \norm{\tilde k}{S^{d-1}_{1,0}(\cal S(\Rp)),j}$.
Finally, when $s\ge0$ one has for $K'$ in the proof above that
$\norm{\tilde k\ang{\xi'}^{-m}}{S^{d-m-1}_{1,0},j}\le c(j)
\norm{\tilde k}{S^{d-1}_{1,0},j}$, so it can be used that $K$ acts as
$K'\Xi^{\prime\, m}$.
\end{proof}

\subsection{Truncated pseudo-differential operators, $P_{+}$} 
    \label{psdo-ssect}
The results for the $P_{+}$ operators are obtained for spaces with
$p<1$ by a combined application of
interpolation and para-multiplication due to Franke. 

Recall the extended definition of $e^+$ in Section~\ref{ext-ssect}. Since a
truncated pseudo-differential operator is defined as $P_+=r^+Pe^+$ it
is clear that $P_+$ is defined for certain singular distributions (in
spaces with $\fracp-1<s\le 0$).

\begin{thm} \label{psdo-thm} 
Let $p(x,\xi)\in
S^d_{1,0,\op{uttr}}(\Rn\times\Rn)$ for some $d\in\R$, and let
$p$ and $q\in\,]0,\infty]$. If $s>\max(\fracp-1,\fracc np-n)$ the operator
$P_{+}=r^+\op{OP}(p)e^+$ is bounded
 \begin{align}
 P_{+}&\colon B^{s}_{p,q}(\Rp^n)\to B^{s-d}_{p,q}(\Rp^n),
 \label{30} \\
 P_{+}&\colon F^{s}_{p,q}(\Rp^n)\to F^{s-d}_{p,q}(\Rp^n),
 \label{31}
 \end{align}
where in addition $p<\infty$ is assumed in \eqref{31}.
\end{thm}

\begin{proof} The cases $1\le p\le\infty$ are covered first.
When $\fracp-1<s<\fracp$, \eqref{30} and \eqref{31} follow from \eqref{i14}
and \eqref{ext3}. For $s>\fracp$ the induction argument as presented 
in \cite{G3} can be used to cover the Besov as well as the Triebel--%
Lizorkin cases with $s-\fracp\notin\N_0$ when one uses Proposition~%
\ref{pois-prop} $2^\circ$. Here the equivalent norms for these spaces 
given in \cite[3.3.5]{T2} are needed; the unnecessary restriction in
\cite[3.3.5/2]{T2} is removable by \cite[Thm.\ 4.1.1]{F3}. 
The cases $s-\fracp\in\N_0$ are
then covered by use of real interpolation, cf.~Theorem~2.4.2 and
Proposition~2.4.1 in \cite{T2}.

It remains to consider the case $0<p<1$, where it by real
interpolation suffices to prove \eqref{31}. Let $u\in F^{s}_{p,q}(\Rp^n)$
be given and take $v\in F^{s}_{p,q}(\Rn)$ such that $r^+v=u$. Then $v$
is an $L^t$ function (for some $t>1$) and $e^+u=\chi v$ as seen above
\eqref{ext7}. 

The product $\chi v=\pi(\chi,v)$ may be analysed by means of 
the para-multipli\-ca\-tion operators $\pi_j(\cdot,\cdot)$ with $j=1$, $2$
and $3$ (in the sense of \cite{Y1}), provided these exist. In fact it
is obtained then, cf.~\cite[(3.6)]{JJ94mlt}, that
 \begin{equation}
 r^+P\pi(\chi,v)=r^+P\pi_1(\chi,v)+r^+P\pi_2(\chi,v)+r^+P\pi_3(\chi,v).
 \label{32}
 \end{equation}
From \eqref{i14} and the results for the $\pi_j(\cdot,\cdot)$ 
it follows that the operators
 \begin{equation}
 r^+P\pi_j(\chi,\cdot)\colon F^{s}_{p,q}(\Rn)\to F^{s-d}_{p,q}(\Rp^n),
 \quad\text{ with $j=1,2,3$,}
 \label{33}
 \end{equation}
are bounded when $s\in\R$, $s>\max(0,\fracc np-n)$ and $s<0$, respectively:
for $j=1$ \cite[(5.1)]{JJ94mlt} applies, since $\chi\in
L_\infty$; for $j=2$ and $q\ge p$ formula (5.10) there is easily
modified to give a version for $B^{0}_{\infty,\infty}\oplus
F^{s}_{p,q}$, and generally the proof of \eqref{ext7} in 
\cite[Thm.~3.4.2]{F3} show
the property; for $j=3$ a variant of \cite[(5.9)]{JJ94mlt} may be used.

By Proposition~\ref{intp-prop} it would be enough to show that
$r^+P\pi_3(\chi,\cdot)$ is bounded between the spaces in \eqref{33} for,
say, $s\ge0$ and $p=q=2$. Indeed, in this case it would follow by
$+-$-interpolation that $r^+P\pi_3(\chi,\cdot)$ is bounded between
the spaces in \eqref{33} for any $s$, $p$ and $q$, and then, by \eqref{32},
boundedness would hold for $r^+P\pi(\chi,v)$ for $p<1$ when $s>\fracc
np-n$. Clearly \eqref{31} follows from this by taking the infimum over $v$.

Therefore we shall derive the continuity of $r^+P\pi_3(\chi,\cdot)$ 
in \eqref{33} for $s>\fracp-1$ and $1\le p<\infty$ from the fact
that \eqref{31} holds for $1\le p<\infty$. 
First note that \eqref{31} implies that the operator
$r^+P\pi(\chi,\cdot)$ is bounded between the spaces in \eqref{33} when
$s>\fracp-1$ for some $1\le p<\infty$. From \eqref{32} and \eqref{33} it 
then follows that $r^+P\pi_3(\chi,\cdot)$ has the desired property.
\end{proof}

The theorem above contains an improvement over \cite{F2}, in that
for $p=\infty$ it is not assumed that the operators are properly
supported.

From Theorem~\ref{psdo-thm} it follows that $P_+(\cal
S(\Rp^n))\subset C^\infty(\Rp^n)$, and we even have

\begin{prop} \label{psdo-prop}
Let $P$ be a pseudo-differential operator with symbol $p(x,\xi)$ in
$S^{d}_{1,0,\op{uttr}}$. Then $P_+\colon\cal S(\Rp^n)\to\cal S(\Rp^n)$ is 
continuous.
\end{prop} 
\begin{proof} Recall the commutator identities $[D^\alpha_x,P]=
\op{OP}(D^\alpha_xp)$ and $[x^\alpha,P]=\linebreak[4]\op{OP}(D^\alpha_\xi p)$ valid
on $\cal S'(\Rn)$ and $[D^\alpha_x,e^+]u=-\kd{\alpha_n=1}i\gamma_0u(x')\otimes
\delta_0(x_n)$ valid for $u\in\cal S(\Rp^n)$ when $|\alpha|=1$. By use
of these it is seen that $x^\alpha D^\beta P_+u$ is a sum of terms
either of the form $Q_+x^\gamma D^\omega u$, with $Q$ in
$\op{OP}(S^{d}_{1,0,\op{uttr}})$, $\gamma\le\alpha$ and $\omega\le\beta$,
or of the form $K\gamma_0(\xi^{\gamma'}D^{\omega'}u)$, where
$K\in\op{OPK}(S^{d+\alpha_n}_{1,0}(\cal S(\Rp^n)))$,
$\gamma'\le\alpha'$ and $ \omega'\le\beta'$. Hence $x^\alpha D^\beta P_+u\in
C(\Rp^n)$ with $\norm{x^\alpha D^\beta P_+u}{L_\infty}\le
C\norm{u}{\cal S(\Rp^n),N}$ for appropriate constants $C$ and $N$.
\end{proof}

\subsection{Trace operators} \label{trac-ssect}
A trace operator of {\em class} $r\in\Z$ and order $d\in\R$ is of the form
 \begin{equation}
 Tu(x')=\sum_{0\le j<r_+} S_j\gamma_ju(x')+T_0u(x'), 
\quad\text{ for}\quad u\in\cal S(\Rp^n),
 \label{50}
 \end{equation}
where each $S_j=\op{OP}'(s_j)$ is a pseudo-differential operator on $\R^{n-1}$,
with symbol $s_j(x',\xi')$ in $S^{d-j}_{1,0}(\R^{n-1}\times\R^{n-1})$, and
the sum is void when $r<1$.
$T_0=\op{OPT}(\tilde t_0)$ given as in \eqref{-22} is the part of 
class $\le0$ with $\tilde t_0\in S^{d}_{1,0}(\cal S(\Rp))$.

$T$ is of class $r<0$  when (the sum is void and)
one of the equivalent conditions in Proposition \ref{trac-prop} $3^\circ$ 
below is satisfied. To prepare for these, let
 \begin{equation}
 \cal S_m(\Rp)=\{\,f\in\cal S(\Rp)\mid
 \gamma_0f=\dots=\gamma_{m-1}f=0\,\},
 \label{51}
 \end{equation}
where the index $m\in\N$ counts the number of traces required to vanish.
(This should not be confounded with $\cal S_0(\Rp^n)$ that consists of
functions on $\Rn$ supported by $\Rp^n$.) The conditions in $3^\circ$ below 
for negative class have been introduced by Franke and Grubb, cf.~%
\cite{F1,F2}, \cite{G2} and \cite{GK}.

\bigskip

The analysis of the trace operators departs from a description of the
standard traces $\gamma_j$ that enter in \eqref{50} above. 
See Section \ref{gamm-ssect} for the definitions and the basic results.

Recall in particular the $\Dm_k$-notation, cf.\  Figure \ref{D0-fig}. 
It is chosen as a reminder of the fact that $\Dm_k$ is a domain
consisting of numbers (rather than of vectors).
Observe that Theorem~\ref{psdo-thm} states that $P_{+}$ satisfies 
\eqref{30} and \eqref{31} when $(s,p,q)\in \Dm_0$.  

The aim in the following is to show that when, say, a trace operator 
$T$ is of class $r\in\Z$ then it is bounded from spaces with
parameters $(s,p,q)$ in $\Dm_r$. 

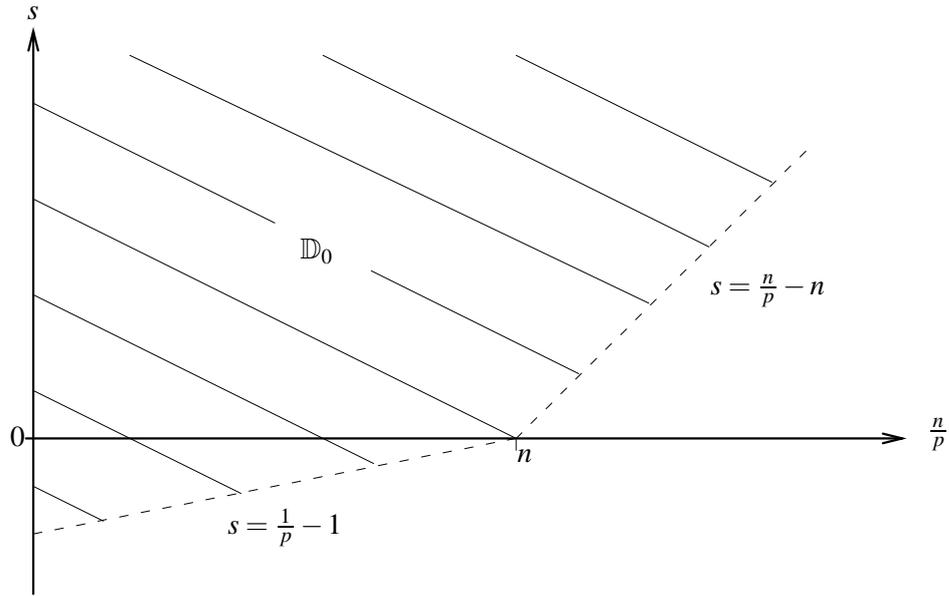
\begin{figure}[htbp]
\hfil
\setlength{\unitlength}{0.0125in}
\begin{picture}(395,261)(0,-10)
\dashline{4.000}(10,25)(210,65)
\dashline{4.000}(210,65)(330,185)
\path(210,65)(210,60)
\path(10,45)(39,30.5)
\path(10,85)(96,42)
\path(10,125)(151,54.5)
\path(10,165)(210,65)
\path(50,225)(265,121.5)
\path(130,225)(290,145)
\path(210,225)(316,172)
\path(10,205)(110,155)
\path(150,135)(236,92)
\thicklines
\path(7,65)(370,65)
\path(362.000,63.000)(370.000,65.000)(362.000,67.000)
\path(10,0)(10,235)
\path(12.000,227.000)(10.000,235.000)(8.000,227.000)
\put(380,65){\makebox(0,0)[lb]{\raisebox{0pt}[0pt][0pt] 
{\shortstack[l]{{$\fracc np$}}}}}
\put(90,25){\makebox(0,0)[lb]{\raisebox{0pt}[0pt][0pt]
{\shortstack[l]{{$s=\fracp-1$}}}}}
\put(290,125){\makebox(0,0)[lb]{\raisebox{0pt}[0pt][0pt]
{\shortstack[l]{{$s=\fracc np-n$}}}}}
\put(7,240){\makebox(0,0)[lb]{\raisebox{0pt}[0pt][0pt]
{\shortstack[l]{{$s$}}}}}
\put(0,62){\makebox(0,0)[lb]{\raisebox{0pt}[0pt][0pt]
{\shortstack[l]{{$0$}}}}}
\put(210,55){\makebox(0,0)[lb]{\raisebox{0pt}[0pt][0pt]
{\shortstack[l]{{$n$}}}}}
\put(120,140){\makebox(0,0)[lb]{\raisebox{0pt}[0pt][0pt]
{\shortstack[l]{{$\Dm_0$}}}}}
\end{picture}
\hfil
\caption[The definition of $\protect\Dm_0$.]{The definition of 
             $\protect\Dm_0$ (for $n=5$).}
 \label{D0-fig}
\end{figure}

Recall also that the dual of $\cal S(\Rp^n)$ is $\cal S'_0(\Rp^n)$,
with $\dual{u}{\varphi}=\dual{u}{\psi}$ when $\varphi=r^+\psi$ for
$\psi\in\cal S(\Rn)$ and $u\in\cal S'_0(\Rp^n)$. Similarly $\cal
S'(\Rp^n)'=\cal S_0(\Rp^n)$.

Among the statements in Proposition \ref{trac-prop} below 
fairly elementary proofs of $1^\circ$ and $2^\circ$ are given
(until now simple explanations are available in the $x'$-independent case). 

Observe that in $1^\circ$ below, $r^+\overline{\cal F}^{-1}_{\xi_n\to
y_n}q(x',\xi)=r^-\cal F^{-1}_{\xi_n\to
z_n}q(x',\xi)\!\bigm|_{y_n=-z_n}$ belongs to $\cal S(\Rp)$ as a function
of $y_n$, because the transmission condition is two-sided, i.e.,  in
\eqref{-2} the supremum is also taken over $z_n\in\Rn_-$.

\begin{prop} \label{trac-prop}  
$1^\circ$ Every $q(x'\!,\xi)\in S^{d}_{1,0,\op{uttr}}(\Rn\!\times\!\Rn)$
satisfies the relation, for $u\in\cal S\!(\Rp^n)$,
 \begin{equation}
 \gamma_0\op{OP}(q)_{+}u=\sum_{0\le j\le d}\op{OP}'(s_j)\gamma_ju
 +\op{OPT}(r^+\overline{\cal F}^{-1}_{\xi_n\to y_n}q)u,
 \label{52}
 \end{equation}
when the symbols $s_j(x',\xi')\in S^{d-j}_{1,0}(\R^{n-1}\times\R^{n-1})$ 
are determined from $q$ by Proposition~\ref{uttr-prop}. 
Hence $\gamma_0\op{OP}(q)_+$ is a trace operator of class $\le (d+1)_+$.

 $2^\circ$ For each symbol-kernel $\tilde t\in S^{d}_{1,0}(\cal S(\Rp))$ 
there exists a $p(x',\xi)\in S^d_{1,0,\op{uttr}}$ and a
Poisson operator $K$ such that
 \begin{equation}
 Tu=\op{OPT}(\tilde t)u=K^*e^+u =\gamma_0\op{OP}(p)_{+}u
 \quad\text{ holds for}\quad u\in\cal S(\Rp^n).
 \label{52'}
 \end{equation}
Here $K=\op{OPK}(e^{iD_{x'}\cdot D_{\xi'}}\overline{\tilde t})$ and 
$r^+\overline{\cal F}^{-1}_{\xi_n\to y_n}p(x',\xi)=\tilde
t(x',y_n,\xi')$. Moreover, the continuous operator $K^*\colon
\cal S'_0(\Rp^n)\to\cal S'(\R^{n-1})$ is uniquely determined by \eqref{52'}.

$3^\circ$ Let $S^{(k)}_{T}=\op{OP}'(i\overline{D}^k_{y_n}\tilde t(x',0,\xi'))$ 
for $k\in\N_0$, whenever $T$ is a class 0 trace operator with 
symbol-kernel $\tilde t(x',y_n,\xi')\in S^{d}_{1,0}(\cal S(\Rp))$.
Then, for each $m\in\N$, the following conditions are equivalent:
 \begin{itemize} \addtolength{\itemsep}{\jot}
 \item[(i)] $\tilde t(x',\cdot,\xi')\in \cal S_m(\Rp)$ 
for each $x'$ and $\xi'$.
 \item[(ii)] $t(x',\xi',\xi_n):=\overline{\cal F}_{y_n\to\xi_n}
e^+_{y_n}\tilde t(x',y_n,\xi')\in\cal H^-_{-1-m}$ as a function of $\xi_n$, for
each $(x',\xi')$.
 \item[(iii)] $S^{(0)}_{T}=\dots=S^{(m-1)}_T=0$.
 \item[(iv)] $TD^\alpha$ is a trace operator of class 0 for each
$|\alpha|\le m$.
 \end{itemize}
In the affirmative case, $T$ is said to be of class $-m$, and when
this holds for every $m\in\N$, the class of $T$ is said to be $-\infty$.
\end{prop}

\begin{proof} $1^\circ$ According to Proposition~\ref{uttr-prop} there is
a decomposition
 \begin{equation}
 q(x',\xi)=\sum_{0\le j\le d} s_j(x',\xi')\xi_n^j+
 h_{-1,\xi_n}q(x',\xi).
 \label{53}
 \end{equation}
Here $\gamma_0\op{OP}(\sum s_j\xi_n^j)_{+}u$ equals 
$\sum \op{OP}'(s_j)\gamma_ju$, because
$s_j(x',\xi')\xi_n^j$ is a polynomial in $\xi$ (so
$\op{OP}(s_j\xi_n^j)$ acts on $\cal S'(\Rn)$ as a differential operator).

Thus $q=h_{-1}q$ can be assumed. Any $\eta\in C^\infty(\R)$ with
$\eta(t)=0$ for $t<\tfrac{1}{2}$ and $\eta=1$ for $t>1$ can now be
used to approximate $e^+u$ in $\cal S'(\Rn)$ by 
$\eta(kx_n)u\in\cal S(\Rn)$ for $k\in\N$, so 
 \begin{equation}
 \dual{\op{OP}(q)e^+u}{\psi}=\lim_{k\to\infty} \iint e^{ix\cdot\xi}
 q(x',\xi)\cal F(\eta(k\cdot)e^+u)(\xi)\psi(x)\,\dbar\xi dx,
 \label{54}
 \end{equation}
when $\psi\in\cal S(\R^{n})$. By Fubini's theorem the $\xi_n$ 
variable can be integrated first, and since
 \begin{equation}
 \dual{\tfrac{e^{ix_n\xi_n}}{2\pi}q(\xi_n)}{ \cal
 F_{y_n\to\xi_n}(\eta(ky_n) e^+\acute u(y_n))}=
 \dual{(\overline{\cal F}^{-1}_{\xi_n\to y_n}q)(\cdot-x_n)}{
 \eta(k\cdot)e^+\acute u(\cdot)}
 \label{55}
 \end{equation}
for each $x'$ and $\xi'$ when $\dual{\cdot}{\cdot}$ denotes the duality
between $\cal S'(\R)$ and $\cal S(\R)$, it is found that
 \begin{align}
 \dual{\op{OP}(q)e^+u}{\psi}&=
 \smash[b]{ \lim_k \iiint \! e^{ix'\cdot\xi'}  }
 (\overline{\cal F}^{-1}_{\xi_n\to y_n}q)(x',y_n-x_n,\xi')
 \label{56} 
 \\[-1\jot]
 &\hphantom{ =\lim_k \iiint \!e^{ix\cdot\xi}
             (\overline{\cal F}^{-1}_{\xi_n\to } }
 \times \eta(ky_n)e^+\acute u(\xi',y_n)\psi(x)\,dy_n\dbar\xi' dx 
 \notag \\
 &=\,\dual{\iint \!e^{ix'\cdot\xi'}
 (\overline{\cal F}^{-1}_{\xi_n\to y_n}q)(x',y_n-x_n,\xi')
 e^+\acute u(\xi',y_n)\,dy_n\dbar\xi'}{\psi}.
 \notag
 \end{align}
Indeed, the limit is calculated by a majorisation using that
 \begin{equation}
 \begin{gathered}
 |\overline{\cal F}^{-1}_{\xi_n\to y_n}q|(x',z_n,\xi')\le
 C\ang{\xi'}^{d+1}\le C\ang{\xi'}^{2|d+1|}, \\
 \begin{aligned}
 |\ang{\xi'}^{l}\ang{y_n}^{2}\acute u(\xi',y_n)|&\le
 \sup_{y_n}\int|(1-\lap_{y'})^{n+|d+1|}(1+y_n^2)u(y',y_n)|\,dy'
 \\
 &\le\,C\norm{u}{\cal S(\Rp^n),l}\int\ang{y'}^{-2n}\,dy'.
 \end{aligned}  
 \end{gathered}
 \label{57}
 \end{equation}
when $l=2(n+|d+1|)$. This procedure shows, in fact, that 
 \begin{equation}
 L_\infty(\Rn)\ni \iint e^{ix'\cdot\xi'}
 (\overline{\cal F}^{-1}_{\xi_n\to y_n}q)(x',y_n-x_n,\xi')
 e^+\acute u(\xi',y_n)\,dy_n\dbar\xi',
 \label{58}
 \end{equation}
which justifies the last relation in \eqref{56}. Hence this function 
equals $\op{OP}(q)e^+u$.

However, by application of \eqref{57} it is seen that
$r^+\op{OP}(q)e^+u$ is continuous in $x\in\Rp^n$ (since the integrand
in \eqref{58} converges a.e.\  when $x$ belongs to a convergent sequence).
Thus $r^+\op{OP}(q)e^+u$ may be restricted to a hyperplane
$\{\,x_n=a\,\}$, $a>0$, and the limit for $x_n\to0_+$ of this continuous
function of $x'$ is calculated for $\varphi\in C^\infty_0(\R^{n-1})$
by majorised convergence (using \eqref{57}):
 \begin{equation}
 \begin{aligned}
 \dual{\op{OP}(q)e^+u(\cdot,x_n)}{\varphi}&=
 \smash[b]{ \iiint \!e^{ix'\cdot\xi'} }
 (\overline{\cal F}^{-1}_{\xi_n\to y_n}q)(x',y_n-x_n,\xi') \\
 &\hphantom{=\iiint \!e^{ix\cdot\xi}{\cal F}
           }
 \times  e^+\acute u(\xi',y_n)\varphi(x')\,dy_n\dbar\xi'dx' \\
 & \xrightarrow[\;x_n\to0_+\;]{~} \dual{\op{OPT}(r^+\overline{\cal F}^{-1}_{\xi_n\to
 y_n}q(x',\xi)u}{\varphi}.
 \end{aligned}
 \label{59}
 \end{equation}
It follows that $\gamma_0\op{OP}(q)_{+}u=\op{OPT}(r^+\overline{\cal
F}^{-1}_{\xi_n\to y_n}q)u$.

$2^\circ$ The symbol $p(x',\xi)$ can be taken as $\overline{\cal
F}_{y_n\to\xi_n}\tilde p(x',y_n,\xi')$, where $\tilde p$ denotes a
Seeley extension to $y_n<0$ of $\tilde t(x',y_n,\xi')$, cf.\ 
Proposition \ref{pois-prop} $3^\circ$ or \cite{G3}. 
Observe that the smoothness of
$\tilde p$ in $y_n$ implies that $h_{-1,\xi_n}p=p$, so that
$\gamma_0\op{OP}(p)_{+}u=\op{OPT}(\tilde t)u$ by \eqref{52}.

Let $K$ denote the Poisson operator with the symbol-kernel 
$\tilde k=\tilde t^*$ as in the proposition. For $u\in\cal S(\Rp^n)$
and $\psi\in\cal S(\R^{n-1})$ we have
 \begin{equation}
 \dual{Tu}{\overline{\psi}}=\dual{\int_0^\infty
  \op{OP}'(\tilde t (\cdot,y_n,\cdot))\acute u(\cdot,y_n)
              \,dy_n}{\overline{\psi}}
 \label{60}
 \end{equation}
since the majorisation
 \begin{multline}
 |\tilde t(x',y_n,\xi')\acute u(\xi',y_n)|\le
 2\norm{\tilde t}{S^{d}_{1,0},2}\ang{\xi'}^{2|d+1|}\ang{y_n}^{-2}
 |\acute u(\xi',y_n)|
 \\
 \le 2\norm{\tilde t}{S^{d}_{1,0},2}
  \int_{\Rn_+}(1-\lap_{y'})^{n+|d+1|}u(y)\,dy
  \cdot \ang{\xi'}^{-2n}\ang{y_n}^{-2} 
 \label{60'}
 \end{multline}
allows a change in the order of integration in the definition of $Tu$. Then 
 \begin{equation}
 \dual{Tu}{\overline{\psi}}=\int_0^\infty
  \dual{\op{OP}'(\tilde t (\cdot,y_n,\cdot))\acute u(\cdot,y_n)
    }{\overline{\psi}}\,dy_n,
 \label{60''}
 \end{equation}
for $(\op{OP}'(\tilde t(\cdot,y_n,\cdot))u(\cdot,y_n))(x')\overline{\psi}(x')$
has by \eqref{60'} a majorant that is integrable with respect to $(x',y_n)$.
However, with $(\cdot,\cdot)=\dual{\cdot}{\overline{\,\cdot\,}}$ it is
found from \eqref{60''} that
 \begin{multline}
 \dual{Tu}{\overline{\psi}}=\int_0^\infty
 (u(\cdot,y_n),\op{OP}'(e^{iD_{x'}\cdot D_{\xi'}}
     \overline{\tilde t}(\cdot,y_n,\cdot))\psi)\,dy_n
 \\ 
 =\int_{\Rn_+}u(x)\overline{K\psi(x)}\,dx =\dual{e^+u}{\overline{K\psi}},
 \label{61}
 \end{multline}
and this shows that $Tu=K^*e^+u$ for $u\in\cal S(\Rp^n)$.

Moreover, since $e^+C^\infty_0(\Rn_+)$ is dense in $\cal S'_0(\Rp^n)$
it follows from this relation that the continuous operator 
$K^*\colon\cal S'_0(\Rp^n)\to\cal S'(\R^{n-1})$ is uniquely determined.

$3^\circ$ That $\text{(i)}\iff\text{(ii)}$ is clear from the $\cal
H$-theory, for $\overline{\cal F}\colon e^+\cal S(\Rp)\to\cal
H^-_{-1}$ is a bijection with the property that $s_{-1-k}$ of
$\overline{\cal F}e^+u$ equals $-\gamma_ku$, when $u\in\cal S(\Rp)$.

$\text{(i)}\Rightarrow\text{(iii)}$ is trivial, and when
$S^{(k)}_T=0$ one has $i\overline{D}^k_{y_n}\tilde
t(x',0,\xi')\equiv0$, since in the uniform calculus there is a
bijective correspondence between operators and symbols, cf.\ 
\cite[18.1]{H}. Hence $\text{(iii)}\Rightarrow\text{(i)}$.
$\text{(iii)}\iff\text{(iv)}$ since \cite[Prop.\ 2.6]{G2} is
valid also in the uniform case.
\end{proof}

The restriction to
$x_n$-independent symbols $q(x',\xi)$ in $1^\circ$ above was made
partly because this generality is sufficient for the
application in the proof of $2^\circ$; and partly because it requires
extra techniques to handle symbols $q(x,\xi)$, since a
decomposition like that in \eqref{53} holds only for $x_n=0$, then. 

\bigskip

It is an important result in Proposition~\ref{trac-prop} $2^\circ$
that for each trace operator $T$ of class 0
 \begin{equation}
 Tu=K^*e^+u,\quad\text{ for}\quad u\in\cal S(\Rp^n).
 \label{62'}
 \end{equation}
Indeed, this fact may be used together with Section~\ref{ext-ssect} 
to make the following

\begin{defn}  \label{trac-defn}
Let $u$ belong to $B^{s}_{p,q}(\Rp^n)$ or to
$F^{s}_{p,q}(\Rp^n)$ for some $(s,p,q)$ in $\Dm_0$ (with $p<\infty$ in the
Triebel--Lizorkin case), and let the trace operator $T$ have class~$0$
and symbol-kernel $\tilde t\in S^{d}_{1,0}(\cal S(\Rp))$.

Then the action of $T$ on $u$ is defined as $Tu=K^*e^+u$, whereby $K=
\op{OPK}(\tilde t^*)=\linebreak\op{OPK}(e^{iD_{x'}\cdot D_{\xi'}}\overline{\tilde t})$. 
\end{defn}

The justification is, of course, that the action of $K^*$ is determined 
by $T$. The definition is natural when compared to a pseudo-differential
operator $P\colon\cal S(\Rn)\to\cal S(\Rn)$ that is extended to $\cal
S'(\Rn)$ as $P=\op{OP}(e^{iD_x\cdot D_\xi}\overline{p})^*$, cf.~\eqref{-16}. 

In the following we shall for $(s,p,q)$ in $\Dm_0$
derive the continuity in \eqref{66} and \eqref{67} below for $r=0$.

\smallskip

The idea is to show that $K^*e^+$ for $d\le-1$ acts as
$\gamma_0P_+$, when $P=\op{OP}(p)$ is chosen according to $2^\circ$ in 
Proposition~\ref{trac-prop}. This is useful because
Lemma~\ref{gamm1-lem} and Theorem~\ref{psdo-thm} give the boundedness of
 \begin{equation}
 B^{s}_{p,q}(\Rp^n)\xrightarrow{\;P_+\;} B^{s-d}_{p,q}(\Rp^n)\xrightarrow{\;\gamma_0\;} 
 B^{s-d-\fracpi}_{p,q}(\R^{n-1})
 \label{62}
 \end{equation}
for every $(s,p,q)\in\Dm_0$, and similarly for the $F^{s}_{p,q}$
spaces.

For $(s_1,p_1,q_1)\in\Dm_0$ there is always an embedding
$B^{s_1}_{p_1,q_1}+F^{s_1}_{p_1,q_1}\hookrightarrow B^{s}_{p,q}$ where
$p$ and $q\in\,]1,\infty]$ and $\fracp-1<s<\fracp$. Thus it suffices to
check that $K^*e^+u=\gamma_0P_+u$ for $u\in B^{s}_{p,q}(\Rp^n)$ when $(s,p,q)$
satisfies these requirements. Moreover, since $e^+\colon
B^{s}_{p,q}(\Rp^n)\to B^{s}_{p,q;0}(\Rp^n)$ is bounded then, it
will be enough to check that $K^*=\gamma_0r^+\!P$ holds on
$B^{s}_{p,q;0}(\Rp^n)$ for the specified $(s,p,q)$.

To carry out this programme one can show that for $(s,p,q)\in\Dm_0$ with $p$ 
and $q\in\,]1,\infty]$ there is a commutative diagram
 \begin{equation}
  \begin{CD}
  B^{s}_{p,q;0}(\Rp^n) @>\iota>> \cal S'_0(\Rp^n) \\
  @V\gamma_0r^+P VV               @VVK^*V  \\
   B^{s-d-\fracpi}_{p,q}(\R^{n-1}) @>\iota>>  \cal S'(\R^{n-1})
 \end{CD}\,.
 \label{63}
 \end{equation}
However, when $e^+C^\infty_0(\Rn_+)\subset B^{s}_{p,q;0}(\Rp^n)$ is not dense, 
this is not trivial.

Note that $P^*$ given in $(x',y_n)$-form has the symbol
 \begin{equation}
 q(x',y_n,\xi')=e^{-iD_{x_n}D_{\xi_n}}e^{iD_{x}\cdot D_\xi}
 \overline{p}(x',\xi)=e^{iD_{x'}\cdot D_{\xi'}}\overline{p}(x',\xi),
 \label{63'}
 \end{equation}
which is $y_n$-independent. Hence the Poisson operator 
$r^+P^*(\cdot\otimes\delta_0)$ has the symbol-kernel 
 \begin{equation}
 r^+\cal F^{-1}_{\xi_n\to x_n}q(x',\xi)=
 e^{iD_{x'}\cdot D_{\xi'}}r^+\overline{\overline{\cal F}^{-1}_{\xi_n
 \to x_n}p(x',\xi)}= e^{iD_{x'}\cdot D_{\xi'}}
 \overline{\tilde t}=\tilde k.
 \label{63''}
 \end{equation}

One has that $B^{s}_{p,q;0}(\Rp^n)=(B^{-s}_{p',q'}(\Rp^n))'$, for
$\fracp+\tfrac{1}{p'}=1$ and $\fracc1q+\tfrac{1}{q'}=1$, so
\eqref{63} is obtained from the commutative diagram  
 \begin{equation}
  \begin{CD}
 B^{-s}_{p',q'}(\Rp^n) @<\iota<< \cal S(\Rp^n) \\
 @Ar^+P^*(\cdot\otimes\delta_0)AA       @AAKA  \\ 
 B^{-s+d+\fracpi}_{p',q'}(\R^{n-1})     @<\iota<< \cal S(\R^{n-1})
 \end{CD}
 \label{64}
 \end{equation}
by taking adjoints. Indeed, for $v\in B^{s}_{p,q;0}(\Rp^n)$ and $w\in
B^{-s+d+\fracpi}_{p',q'}(\R^{n-1})$
 \begin{equation}
 \dual{r^+P^*(w\otimes\delta_0)}{\overline{v}}=
 \dual{w\otimes\delta_0}{\overline{Pv}}=
 \dual{w}{\overline{\tilde\gamma_0Pv}};
 \label{65}
 \end{equation}
here the last relation is obtained by closure from the case with $w\in
C^\infty_0(\R^{n-1})$ and $v\in C^\infty(\Rp^n)\cap B^{s}_{p,q;0}$, for 
$q<\infty$ suffices and $\tilde\gamma_0$ makes sense on $P(B^{s}_{p,q;0})$ 
when $(s-d,p,q)\in\Dm_1$. By definition $\gamma_0r^+P=\tilde\gamma_0P$. 

For $d\le-1$ this shows \eqref{66} and \eqref{67}.
When $d>-1$ note that $T=\Xi^{'(d+1)}K_1^*e^+$
with $K_1:=\op{OPK}(\ang{\xi'}^{-d-1}\tilde k)=K\Xi^{'(-d-1)}$,
simply because $K_1^*=\Xi^{'(-d-1)}K^*$. 
Since $K_1^*e^+$ acts as a trace operator of order $-1$, 
it follows that the formulae hold also in this case (but still for $r=0$).

We shall now lift these considerations to a much stronger result, that
in special cases can be found in \cite{G3}. By and large we modify the
proofs there.
 
\begin{thm} \label{trac-thm} 
A trace operator $T$ of order $d\in\R$ and class $r\in\Z$ is continuous 
 \begin{alignat}{2}
 T&\colon B^{s}_{p,q}(\Rp^n)\to 
 \smash[t]{B^{s-d-\fracpi}_{p,q}(\R^{n-1}),
          }
 &\quad\text{ for}\quad &(s,p,q)\in \Dm_r, 
 \label{66} \\
 T&\colon F^{s}_{p,q}(\Rp^n)\to F^{s-d-\fracpi}_{p,p}(\R^{n-1}),
 &\quad\text{ for}\quad &(s,p,q)\in \Dm_r, 
 \label{67}
 \end{alignat}
when in \eqref{67} also $p<\infty$ holds.

Moreover, if $T$ is continuous from either $B^{s}_{p,q}(\Rp^n)$ or
$F^{s}_{p,q}(\Rp^n)$ to $\cal D'(\R^{n-1})$ for some 
$(s,p,q)\notin\overline{\Dm}_r$, 
then the class of $T$ is  $\le r-1$.

$T$ is continuous from $\cal S'(\Rp^n)$ to $\cal
S'(\R^{n-1})$ if and only if $T$ has class $-\infty$.
\end{thm}

\begin{proof} For $r=0$ the proof of the first part has been conducted 
above, and for $r>0$ one can treat the sum in \eqref{50} by use of
Lemma~\ref{gamm1-lem} and \eqref{i14}. Operators of
negative class $r=-m$, where $m\in\N$, can be handled with 
Proposition~\ref{trac-prop} $3^\circ$(iv) as the point of departure: 
for $u$ in $B^{s_1}_{p_1,q_1}(\Rp^n)$ respectively $F^{s_1}_{p_1,q_1}(\Rp^n)$
and $(s_1,p_1,q_1)$ arbitrary there is a decomposition, 
 \begin{equation}
 u=\sum_{|\alpha|\le m} D^\alpha v_\alpha, \quad\text{ where each
 $ v_\alpha\in B^{s_1+m}_{p_1,q_1}(\Rp^n)$ resp.\  
  $F^{s_1+m}_{p_1,q_1}(\Rp^n)$},
 \label{68} 
 \end{equation}
in such a way that each operator $u\mapsto v_\alpha$ is bounded from
$B^{s_1}_{p_1,q_1}$ to $B^{s_1+m}_{p_1,q_1}$ and from 
$F^{s_1}_{p_1,q_1}$ to $F^{s_1+m}_{p_1,q_1}$. (As usual this can
be seen by expansion of the identity $1=(1+\xi_1^2+\dots+\xi_n^2)^m
\ang{\xi}^{-2m}$.)
Since $(s+m,p,q)\in \Dm_0$ and $TD^\alpha$ is of class 0, it follows
that $u\mapsto \sum_{|\alpha|\le m}TD^\alpha v_\alpha$ has the 
boundedness properties in \eqref{66} and \eqref{67}. For
$(s_1,p_1,q_1)\in\Dm_0$ this operator equals $T$ on $B^{s_1}_{p_1,q_1}$ and
$F^{s_1}_{p_1,q_1}$ that are dense in $B^{s}_{p,q}$ and $F^{s}_{p,q}$ when
$q<\infty$. Hence this extension is unique.

On the other hand, let $T$ be continuous on $B^{s}_{p,q}(\Rp^n)$ 
for some $(s,p,q)\in \Dm_{r_1}\!\setminus\overline{\Dm}_r$ such that
$(s,p,q)\notin \overline{\Dm}_{r_1+1}$ (the
argument is the same in the Triebel--Lizorkin case). 
If $r>r_1\ge0$ the operator $T$ has the form in \eqref{50}, so 
$S_{r-1}\gamma_{r-1}=T-\sum_{0\le j<r-1} S_j\gamma_j-T_0$.
Since the case $p<1$ is a novelty we begin with this. 

Obviously 
$T\colon B^{\fracci np -n+r}_{p,2}\to\cal D'(\R^{n-1}) $ is
continuous, hence $S_{r-1}\gamma_{r-1}$ has the same property and we
shall deduce from this fact that $S_{r-1}=0$.
According to Lemma~\ref{gamm2-lem} there exists for each $z'\in\R^{n-1}$ a
sequence $v_k\in\cal S(\Rp^n)$ such that $v_k\to0$ in
$B^{\fracci np-n+r}_{p,2}(\Rp^n)$ while $\gamma_{r-1}v_k\to
\delta_{z'}$ in $\cal S'(\R^{n-1})$. 
Because of the continuity of $S_{r-1}\gamma_{r-1}$
and of $S_{r-1}$ this implies $S_{r-1}\delta_{z'}=0$. Since
$z'\in\R^{n-1}$ is arbitrary, the identity
 \begin{equation}
 0=\dual{S_{r-1}\delta_{z'}}{\overline{\psi}}=
       \dual{\delta_{z'}}{\overline{S^*_{r-1}\psi}},
 \quad\text{ where}\quad\psi\in\cal S(\R^{n-1}),
 \label{69}
 \end{equation}
gives that both $S^*_{r-1}$ and $S_{r-1}$ are equal to 0. In the case $1\le
p\le\infty$ one concludes that the operator
$S_{r-1}\gamma_{r-1}\colon B^{\fracpi
+r-1-\varepsilon}_{p,1}\to\cal D'(\R^{n-1})$ is continuous for some
$\varepsilon>0$, and then it is inferred from Lemma \ref{gamm2-lem} that
$S_{r-1}$ is 0.

This procedure can be repeated until all the terms in \eqref{50} with
$j\ge r_1$ are shown to be 0; then $T$ is of class $r_1$. If
$r_1=-m<0$, $0\le |\alpha|\le m$, the operators $TD^\alpha$ are 
trace operators, cf.\  \cite{G2}, that are continuous 
$TD^\alpha\colon B^{s+m}_{p,q}(\Rp^n)\to B^{s-d-\fracpi}_{p,q}(\R^{n-1})$,
and since $(s+m,p,q)\in \Dm_0$, each $TD^\alpha$ is of class zero
according to the preceding argument. By use of Proposition
\ref{trac-prop} $3^\circ$(iv) this implies that $T$ is of class $-m=r_1$.

If $T=K^*e^+$ has class $-\infty$, the symbol-kernel of $K$ vanishes
of infinite order at $x_n=0$, implying that $e^+K$ is bounded $\cal
S(\R^{n-1})\to \cal S_0(\Rp^n)$. For $u\in C^\infty(\Rp^n)$ and $v\in\cal
S(\R^{n-1})$ we see that $\dual{u}{\overline{e^+Kv}}=\int u\overline{Kv}
=\dual{e^+u}{\overline{Kv}}=\dual{Tu}{\overline{v}}$. 
Because $C^\infty(\Rp^n)$ is dense in $B^{-N}_{\infty,1}(\Rp^n)$ 
for any $N>0$ and in $\cal S'(\Rp^n)$ it follows that  
$T\subset(e^+\!K)^*\colon\cal S'(\Rp^n)\to\cal S'(\R^{n-1})$ 
with uniqueness of the extension.
Conversely, if $T$ is continuous from $\cal S'(\Rp^n)$, it is 
so from $F^{-N}_{2,2}$, hence of class $-N$, for any $N$.
\end{proof}

The statement above is somewhat more general than the corresponding
one in \cite{F2}. First of all because Definition~\ref{trac-defn}
allowed the inclusion of the case $p=\infty$ without assuming that the 
operator is  properly supported. Secondly the
result for operators of class $-\infty$ seems to be new. Moreover, there are
the limitations on the $(s,p,q)$ parameters in terms of $T$'s class 
(that except for the sharpness generalise the corresponding ones in \cite{G3}).

Like for the Poisson operators the operator norms are estimated.
\begin{cor} \label{trac-cor} 
For each trace operator $T=\op{OPT}(\tilde t)$ of class 0  and order
$d$ the operator norms in \eqref{66} and \eqref{67} satisfy the inequality
 \begin{equation}
 \norm{T}{\Bbb L(B^{s}_{p,q},B^{s-d-\fracpi}_{p,q})}+
 \norm{T}{\Bbb L(F^{s}_{p,q},F^{s-d-\fracpi}_{p,p})}\le
 c\norm{\tilde t}{S^{d}_{1,0},j}
 \label{13''}
 \end{equation}
for some $(s,p,q)$-dependent $c<\infty$ and $j\in\N$ 
(when the $F$-term is omitted for $p=\infty$).
\end{cor}
\begin{proof} By means of the closed graph theorem, which is applicable
by Remark~\ref{tvs-rem}, it is easy to show
that, say, $S^{d}_{1,0}(\cal S(\Rp))\to 
\Bbb L(B^{s}_{p,q},B^{s-d-\fracpi}_{p,q})$ given by $\tilde t\mapsto
\op{OPT}(\tilde t)$ is continuous. Indeed, when $\tilde t_\nu\to\tilde t$
and $\op{OPT}(\tilde t_\nu)=:T_\nu\to T$ we let $\tilde k_\nu=
e^{iD_{x'}\cdot D_{\xi'}}\overline{\tilde t}_\nu$ and 
$K_\nu=\op{OPK}(\tilde k_\nu)$;
$\tilde k$ and $K$ are defined similarly. Then $\tilde k_\nu
\to\tilde k$ in $S^{d}_{1,0}(\cal S(\Rp))$ by \eqref{sk5}.
For $u\in B^{s}_{p,q}(\Rp^n)$ and $\psi\in\cal S(\R^{n-1})$ one has
 \begin{equation}
 \dual{Tu}{\psi} = \lim_\nu\dual{e^+u}{K_\nu\psi}
 =\dual{\op{OPT}(\tilde t)u}{\psi},
 \label{13'''}
 \end{equation}
when the limit is calculated using \eqref{S1}. 
This shows that $T=\op{OPT}(\tilde t)$.
\end{proof}

\begin{rem} \label{adj-rem} 
It should be emphasised that $K^*$ in the formula \eqref{62'} is 
the adjoint of $K\colon\cal S(\R^{n-1})\to\cal S(\Rp^n)$, and as such it is
continuous $K^*\colon\cal S'_0(\Rp^n)\to\cal S'(\R^{n-1})$. 
However, it is a result in the calculus that the trace operators
of order $d$ and class 0 constitute 
precisely the adjoints of the Poisson operators of order $d+1$.
Seemingly this contadicts the fact that $K^*$ acts on $\cal
S'_0(\Rp^n)$ whereas $T$ acts on spaces over $\Rn_+$.

But it is understood in the cited result that $K\colon 
F^{d+\frac12}_{2,2}(\R^{n-1})\to L_2(\Rn_+)$ is the operator that has a
trace operator $T\colon L_2(\Rn_+)\to F^{-d-\frac12}_{2,2}(\R^{n-1})$
as adjoint, cf.\ \cite[(1.2.34)]{G1}.

More generally, one can restrict the distributional adjoint $K^*$ to a 
bounded operator from $F^{s}_{p,q;0}(\Rp^n)$
to $F^{s-d-\fracpi}_{p,p}(\R^{n-1})$ when $(s,p,q)\in\Dm_0\!\setminus
\overline{\Dm}_1$ with $1<p,q\le\infty$. Then $e^+$ is a bijection 
$F^{s}_{p,q}(\Rp^n)\to F^{s}_{p,q;0}(\Rp^n)$, and by
Definition~\ref{trac-defn} the {\em composite\/} $K^*e^+$ is a
bounded trace operator $T\colon F^{s}_{p,q}(\Rp^n)\to 
F^{s-d-\fracpi}_{p,p}(\R^{n-1})$. Identifying $F^{s}_{p,q;0}(\Rp^n)$ 
with $F^{s}_{p,q}(\Rp^n)$ corresponds to omitting $e^+$ (as above
where $L_2(\Rn_+)'=L_2(\Rn_+)$), that is, to identify $K^*$ with $T$.

With this explanation the formula $Tu=K^*e^+u$ will be used 
throughout since the `$e^+$' there is a reminder
of the fact that $T$ need not be defined for every $u\in\cal S'(\Rp^n)$.
\end{rem}

\begin{rem} \label{trac-rem} 
Contrary to the treatment of the subscales $B^{s}_{p,p}$ and $F^{s}_{p,2}$
with $1<p<\infty$ in \cite{G3}, the borderline cases 
$s=r+\max(\fracp-1,\fracc np-n)$ for a trace operator of class $r$ are left 
open here. 

It is quite tempting, though, to treat these cases beginning with the 
analysis of $\gamma_0$ in Remark~\ref{gamm-rem}. However, since
$\gamma_0(B^{\fracpi}_{p,q})=L_p$ when $1\le p<\infty$, it is not even
sufficient to use the scales $B^{s}_{p,q}$ and $F^{s}_{p,q}$ 
simultaneously (which is required since $F^{0}_{p,2}=L_p$ for 
$1<p<\infty$); for $p=1$ it is necessary to go outside of these scales 
since $\gamma_0\colon B^{1}_{1,1}(\Rp^n)\to L_1(\R^{n-1})$ is surjective. 

Besides this, an investigation of the borderline cases for truncated\linebreak[5]
pseudo-differential operators $P_{+}$ would also have to be
carried out. Altogether the exposition would be heavily burdened by
consideration of these borderline cases, so this topic is left for the future.
\end{rem}

\subsection{Singular Green operators} \label{sgo-ssect}
A singular Green operator\,---\,abbreviated as a \sgo in the sequel\,---\,of 
{\em class} $r\in\Z$ and order $d\in\R$ is of the form
 \begin{equation}
 Gu(x')=\sum_{0\le j<r_+} K_j\gamma_ju(x')+G_0u(x'), 
\quad\text{ for}\quad u\in\cal S(\Rp^n),
 \label{70}
 \end{equation}
where each $K_j$ is a Poisson operator of order $d-j$.
Here $G_0=\op{OPG}(\tilde g_0)$  is the part of class $\le0$ with
$\tilde g_0\in S^{d-1}_{1,0}(\cal S(\Rpp))$, given as in \eqref{-23}.

$G$ is of class $r<0$  when (the sum is void and)
one of the~equivalent conditions in Proposition \ref{sgo-prop} $3^\circ$ 
below is satisfied.

Below some known basic results on s.g.o.s, including a
certain Laguerre expansion, shall be modified.
Concerning the conditions for negative class the reader is referred to
the same sources as in Subsection \ref{trac-ssect}. 

\begin{prop} \label{sgo-prop}
$1^\circ$ If $K_\nu$ and $T_\nu$ have symbol-kernels in 
$S^{d_1-1}_{1,0}(\cal S(\Rp))$ respectively $S^{d_2}_{1,0}(\cal S(\Rp))$ 
that for each $j$ and $N\in\N_0$ satisfy
 \begin{equation}
 \norm{\tilde k_\nu\circ\tilde t_\nu}{S^{d_1+d_2-1}_{1,0}(\cal S(\Rp)),j} 
 =\cal O(\nu^{-N})
 \quad\text{ for}\quad \nu\to\infty,
 \label{71} 
 \end{equation}
then the series $\sum_{\nu=0}^\infty \tilde k_\nu\circ\tilde t_\nu$ is rapidly
convergent in $S^{d_1+d_2-1}_{1,0}(\cal S(\Rpp))$ to a limit\linebreak
$\tilde g(x'\!,x_n,y_n,\xi')$, i.e.,
 \begin{equation}
 \norm{\tilde g-\sum_{\nu=0}^l \tilde k_\nu\circ\tilde t_\nu
       }{S^{d_1+d_2-1}_{1,0}(\cal S(\Rpp)),j}=
 \cal O(l^{-N})\quad\text{ for}  \quad l\to\infty,
 \label{73}
 \end{equation}
for each $j$ and $N$ in $\N_0$.

$2^\circ$ For $\tilde g(x',x_n,y_n,\xi')\in 
S^{d-1}_{1,0}(\cal S(\Rpp))$ there exist sequences
$(\tilde k_\nu(x',x_n,\xi'))$ in\linebreak $S^{d-\frac12}_{1,0}(\cal S(\Rp))$ 
and $(\tilde t_\nu(x',y_n,\xi'))$ in $S^{-\frac12}_{1,0}(\cal S(\Rp))$ 
such that \eqref{71} and \eqref{73} hold.

Moreover, for such sequences $(\tilde k_\nu)$ and $(\tilde t_\nu)$ one
has, when $G=\op{OPG}(\tilde g)$, $K_\nu=\op{OPK}(\tilde k_\nu)$
and $T_\nu=\op{OPT}(\tilde t_\nu)=L_\nu^*e^+$ (cf.\ \eqref{62'}), that
 \begin{equation}
 Gu=\sum_{\nu=0}^\infty K_\nu T_\nu u=\sum_{\nu=0}^\infty K_\nu L_\nu^*e^+u,
 \quad\text{ for}\quad u\in\cal S(\Rp^n),
 \label{74}
 \end{equation} 
with convergence of the series in $\cal S(\Rp^n)$.

Furthermore, $\sum_{0}^\infty K_\nu L_\nu^*$ converges weakly
on $\cal S'_0(\Rp^n)$ to $r^+G_1^*\colon\cal S_0'(\Rp^n)\to
\cal S'(\Rp^n)$, where the~\sgo $G_1$ has symbol-kernel
$\tilde g_1=e^{iD_{x'}\cdot D_{\xi'}}
\overline{\tilde g}(x',y_n,x_n,\xi')$.

$3^\circ$ Let $K^{(k)}_{G}=\op{OPK}(i\overline{D}^k_{y_n}
\tilde g(x',x_n,0,\xi'))$ 
for $k\in\N_0$, whenever $G$ is a class 0 \sgo with symbol-kernel 
$\tilde g(x',x_n,y_n,\xi')\in S^{d-1}_{1,0}(\cal S(\Rpp))$.
Then, for each $m\in\N_0$, the following conditions are equivalent:
 \begin{itemize}\addtolength{\itemsep}{\jot}
 \item[(i)] $\tilde g(x',x_n,\cdot,\xi')\in \cal S_m(\Rp)$ 
for each $x'$, $x_n$ and $\xi'$.
 \item[(ii)] $g(x',\xi',\xi_n,\eta_n):=\overline{\cal F}_{y_n\to\eta_n}
\cal F_{x_n\to\xi_n}e_{\R^2_{++}}\tilde g(x',x_n,y_n,\xi')\in
\cal H^+\hat\otimes\cal H^-_{-1-m}$ as a function of $(\xi_n,\eta_n)$, for
each $(x',\xi')$.
 \item[(iii)] $K^{(0)}_{G}=\dots=K^{(m-1)}_G=0$.
 \item[(iv)] $GD^\alpha$ is a \sgo of class 0 for each
$|\alpha|\le m$.
 \end{itemize}
In the affirmative case, $G$ is said to be of class $-m$, and when
this holds for every $m\in\N$, the~class of $G$ is said to be $-\infty$.
\end{prop}

\begin{proof} 
$1^\circ$ Each composite $K_\nu T_\nu$ maps $\cal S(\Rp^n)$ into
itself by Proposition~\ref{S-prop}, and $K_\nu T_\nu u$ equals
 \begin{equation}
 \int e^{ix'\cdot\xi'}\tilde k_\nu(x',x_n,\xi')
 \iiint e^{iy'\cdot(\eta'-\xi')}\tilde t_\nu(y',y_n,\eta')\acute u(\eta',y_n)\,
 dy_n\dbar\eta'dy'\dbar\xi',
 \label{72}
 \end{equation}
for each $u\in\cal S(\Rp^n)$. First it is inferred that, with $\tilde k_\nu
\circ\tilde t_\nu$ given as in \eqref{sk3},
 \begin{equation}
 K_\nu T_\nu u=\iint e^{ix'\cdot\xi'}\tilde k_\nu\circ\tilde t_\nu
 (x',x_n,y_n,\xi')\acute u(\xi',y_n)\, dy_n\dbar\xi'.
 \label{72'}
 \end{equation}
The idea is to let the $y_n$-integration be the last one in \eqref{72}, 
and then apply the result for composition of two pseudo-differential
operators on $\R^{n-1}$; in this case for each parameter value $x_n$
and $y_n$. Then \eqref{72'} is obtained, for there one can integrate
in any order. 

In \eqref{72} a change of integration order needs a justification,
that can be obtained by inserting $x_n^l$ in front of $\tilde k_\nu$ for
$l=d_++2$ and a convergence factor $\chi(\varepsilon y')$, with 
$\chi\in C^\infty_0$, $\chi(0)=1$, in front of $\tilde t_\nu$: 
evidently $K_\nu T_\nu u=x_n^{-l}\lim_{\varepsilon\to0} x_n^l
K_\nu(\chi(\varepsilon\cdot)T_\nu u)$ in $\cal S(\Rp^n)$. 
For each $x_n$ and $y_n$ the
method used in the proof of \cite[Thm.\ 18.1.8]{H} gives that
 \begin{equation}
 \lim_{\varepsilon\to0}\op{OP}'(\tilde k_\nu(\cdot,x_n,\cdot)\circ
  \tilde t_{\nu,\varepsilon}(\cdot,y_n,\cdot))u(\cdot,y_n)=
 \op{OP}'(\tilde k_\nu(\cdot,x_n,\cdot)\circ
  \tilde t_\nu(\cdot,y_n,\cdot))u(\cdot,y_n),
 \label{72''} 
 \end{equation} 
when $\tilde t_{\nu,\varepsilon}=\chi(\varepsilon y')\tilde t_\nu$.
Since $\ang{y_n}^{2}|\op{OP}'(\tilde k_\nu\circ\tilde
t_{\nu,\varepsilon})u(x',y_n)|<C$ for some constant $C$ independent of
$(x',y_n)$ and $\varepsilon$ for $0<\varepsilon\le1$, we infer that
 \begin{equation}
 \begin{split}
 K_\nu T_\nu u &= \lim_{\varepsilon\to0}\int_0^\infty
  \op{OP}'(\tilde k_\nu(\cdot,x_n,\cdot)\circ
  \tilde t_{\nu,\varepsilon}(\cdot,y_n,\cdot))u(\cdot,y_n)\,dy_n
 \\
 &=\int_0^\infty \op{OP}'(\tilde k_\nu(\cdot,x_n,\cdot)\circ
  \tilde t_\nu(\cdot,y_n,\cdot))u(\cdot,y_n) \,dy_n,
 \end{split}  \label{72'''}
 \end{equation}
from where \eqref{72'} is obtained by application of Fubini's theorem.

From \eqref{71} it is seen that $(\sum_{\nu=0}^l\tilde k_\nu\circ
\tilde t_\nu)_{l\in\N}$ is a Cauchy-sequence in
the Fr\'echet space $S^{d_1+d_2-1}_{1,0}(\cal S(\Rpp))$. Hence the
series converges to a limit $\tilde g$ as claimed, even rapidly since
\eqref{71} gives that $\norm{\sum_{\nu=l}^\infty\tilde k_\nu
\circ\tilde t_\nu}{j}$ is $\cal O(l^{-N})$ for any $N$ and $j$.

$2^\circ$ To define $\tilde k_\nu$ and $\tilde t_\nu$ one can use that
$L_2(\R_+)$ has an orthonormal basis consisting of (certain
unconventional) Laguerre functions $(2\pi)^{-1}\varphi_\nu(y_n,\sigma)$ for
$\nu\in\N_0$, cf.~\cite[(1.27) ff.]{G2}. Using this with the parameter
$\sigma=\ang{\xi'}$ one has
 \begin{equation}
 \tilde g(x',x_n,y_n,\xi')=
 \sum_{\nu=0}^\infty 
 b_\nu(x',x_n,\xi')\varphi_\nu(y_n,\ang{\xi'})
 \label{76}
 \end{equation}
for each $x'$, $x_n$ and $\xi'$ when
 \begin{equation}
 b_\nu(x',x_n,\xi')=(2\pi)^{-1}\smash[t]{\int}
 \tilde g(x',x_n,y_n,\xi')\varphi_\nu(y_n,\ang{\xi'})\,dy_n.
 \label{77}
 \end{equation}
It remains to be verified that one can let $\tilde k_\nu=b_\nu$ and 
$\tilde t_\nu=\varphi_\nu(\cdot,\ang{\cdot})$.

To see that $\tilde k_\nu$ and $\tilde t_\nu$ satisfy \eqref{71} we
use the inequality \eqref{sk4}. Concerning $\tilde t_\nu$ an
application of \cite[(2.2.20)]{G1} for $\alpha'=0$ leads to 
 \begin{equation}
 \norm{x_n^mD^l_{x_n}\varphi_\nu(x_n,\ang{\xi'})}{L_{2,x_n}}
 \le C(1+\nu)^{(1+\varepsilon)l-m}\ang{\xi'}^{l-m},
 \label{78}
 \end{equation}
when $\varepsilon>0$. (Here and in the following the equivalent
seminorms based on the $L_2$ norm on $\cal S(\Rp)$ are used.) 
For $\alpha'\ne0$ the identity
 \begin{equation}
 D_{\xi_j}\varphi_\nu=
 (\nu\varphi_{\nu-1}-(\nu+1)\varphi_{\nu+1})(2\ang{\xi'})^{-1}D_j\ang{\xi'}
 \label{79}
 \end{equation}
can be used successively, and when combined with \eqref{78} it is seen
that the worst resulting term contains $(\nu+1)\dots(\nu+|\alpha'|)
\varphi_{\nu+|\alpha'|}(2\ang{\xi'})^{-|\alpha'|}
\prod (D_j\ang{\xi'})^{|\alpha'_j|}$.
However, this is estimated by $\nu^{|\alpha'|+(1+\varepsilon)l-m}
\ang{\xi'}^{-|\alpha'|+l-m}$, and the other terms, of which there is a
fixed $\alpha'$-dependent number, have similar estimates.
Hence 
$\norm{\tilde t_\nu}{S^{-\frac12}_{1,0}(\cal S(\Rp)),j}\linebreak[3]=\cal O(\nu^{3j})$.

To treat the $\tilde k_\nu$ the proof in \cite[p.\ 169]{G1} is modified. 
Using \cite[(2.2.15)]{G1} it is found that
 \begin{align}
 &\norm{\{b_{\nu}(x',x_n,\xi')\}_{\nu=0}^\infty}{\ell_{2,N}}
 \notag \\
 :=&\,\norm{\{b_{\nu}(x',x_n,\xi')(1+\nu)^N\}_{\nu=0}^\infty}{\ell_2}
 \notag  \\
 =&\,\tfrac{1}{2^N}\norm{(\tfrac{1}{\ang{\xi'}}\partial_{y_n}y_n\partial_{y_n}
              +\ang{\xi'}y_n+1)^N\tilde g(x',x_n,\cdot,\xi')}{L_{2,y_n}}
 \notag \\
 \le&\, C_N \sum_{j+k\le N}\ang{\xi'}^{k-j-\frac12}\sup_{x_n,y_n,x'}
 |(1+\ang{\xi'}y_n)(\partial_{y_n}y_n\partial_{y_n})^{j}y_n^k
  \tilde g|
 \notag \\
 \le&\,C'_N\ang{\xi'}^{d+\frac12}.
 \label{80}
 \end{align}
In a similar manner one finds for the `coefficient sequence' 
$x_n^mD^l_{x_n}D^{\beta'}_{x'}b_\nu$ that its $\ell_{2,N}$-norm is 
$\cal O(\ang{\xi'}^{d+\frac12+l-m})$; $D^{\alpha'}_{\xi'}$ can be
applied to the definition of $b_\nu$ and using Leibniz' formula and
\eqref{79} one reduces to the case where $\alpha'=0$.
This implies that $\norm{\tilde k_\nu}{S^{d-\frac12}_{1,0}(\cal
S(\Rp))}=\cal O(\nu^{-N})$ for any $N$, so altogether \eqref{72} is verified.

After the completion of this construction \eqref{74} is proved by
application of Proposition~\ref{S-prop} and $1^\circ$.
  
With $G_1$ given as in the proposition we shall now prove that
$G\subset r^+G_1^*e^+$. By use of Fubini's theorem 
 \begin{equation}
 \begin{split}
 \dual{e^+Gu}{\overline{v}}&=\Dual{\int_0^\infty 
  \op{OP}'(\tilde g(\cdot,x_n,y_n,\cdot))u(\cdot,y_n)\,dy_n}{\overline{v}}
 \\
 &=\int_{\Rpp} (u(\cdot,y_n),\op{OP}'(\tilde g_1(\cdot,x_n,y_n,\cdot))
                          v(\cdot,x_n))\,dy_ndx_n
 \\
 &=\dual{e^+u}{\overline{G_1v}}=\dual{G^*_1e^+u}{\overline{v}}
 \end{split}
 \label{81}
 \end{equation}
for each $u$ and $v$ in $\cal S(\Rp^n)$. The inclusion
of $G$ into $r^+G_1^*e^+$ follows from this.

Each composite $K_\nu L^*_\nu $ has the adjoint 
$L_\nu K_\nu^*e^+$, where $L_\nu=\op{OPK}(\tilde t_\nu^*)$ and
$K_\nu^*e^+=\op{OPT}(\tilde k_\nu^*)$. 
Moreover, from Lemma \ref{sk-lem} one finds that
 \begin{equation}
 \norm{\tilde t_\nu^*\circ\tilde k_\nu^*}{S^{d-1}_{1,0}(\cal S(\Rpp)),j} 
 \le c\norm{\tilde t_\nu}{S^{-\frac12}_{1,0}(\cal S(\Rp)),j'}
 \norm{\tilde k_\nu}{S^{d-\frac12}_{1,0}(\cal S(\Rp)),j'},
 \label{81'}
 \end{equation} 
so the asymptotic properties of 
$\tilde k_\nu$ and $\tilde t_\nu$ shown above imply that \eqref{71} is
satisfied by $\tilde t_\nu^*$ and $\tilde k_\nu^*$. Then \eqref{S3} leads to
convergence of $G_2=\sum L_\nu K^*_\nu e^+$ on $\cal S(\Rp^n)$, and so
 \begin{equation}
 \lim_{l\to\infty}\dual{\textstyle{\sum_{\nu=0}^l}
 K_\nu L^*_\nu u}{ \overline{v}}
 =\dual{u}{\overline{G_2 r^+v}\,}=\dual{G^*_2 u}{\overline{r^+v}}
 =\dual{r^+G^*_2u}{\overline{v}}
 \label{81''}
 \end{equation}
for each $v\in\cal S_0(\Rp^n)$ and $u\in \cal S'_0(\Rp^n)$.
Hence $\sum K_\nu L^*_\nu$ converges weakly to $r^+G_2^*$ as operators $\cal
S'_0(\Rp^n)\to\cal S'(\Rp^n)$; for $u\in e^+C^\infty_0(\Rn_+)$ it
is seen that $r^+G^*_2u=r^+G_1^*u$ in $\cal S'(\Rp^n)$. 

$3^\circ$ is shown by a modification of the corresponding part of 
Proposition~\ref{trac-prop}. Observe that when a Poisson operator
$K=\op{OPK}(\tilde k)\equiv0$, then $K\psi=0$ for $\psi\in\cal
S(\R^{n-1})$ in particular. Therefore the pseudo-differential operator
$\op{OP}'(\tilde k(\cdot,x_n,\cdot))$ equals $0$ on $\cal
S((\R^{n-1}))$ for each $x_n>0$, that is to say, $\tilde
k(x',x_n,\xi')\equiv 0$ for each $x_n$. Hence $\tilde k=0$.
\end{proof}

Since the action of $G^*_1$ is determined by $\tilde g$ we can use
$2^\circ$ in Proposition \ref{sgo-prop} to extend the~definition of
\sgos of class 0 to the spaces where $e^+$ is defined.

\begin{defn}  \label{sgo-defn}
When $u$ belongs to $B^{s}_{p,q}(\Rp^n)$ or
$F^{s}_{p,q}(\Rp^n)$ for some $(s,p,q)\in \Dm_0$ (with $p<\infty$ in the
$F$-case) then the action of a \sgo $G$ of class 0 on $u$ is defined
as $Gu=r^+G^*_1e^+u$, where $G_1=\op{OPG}(e^{iD_{x'}\cdot
D_{\xi'}}\overline{\tilde g}(x',y_n,x_n,\xi'))$.
\end{defn}

Having made this definition, $2^\circ$ in 
Proposition~\ref{sgo-prop} may be applied to the sequences $(K,0,\dots)$ and
$(T,0,\dots)$ whereby the usual composition rule is obtained. 
More precisely, the identity $Gu:=\op{OPG}(\tilde k\circ\tilde t)u=KTu$, 
valid for $u\in\cal S(\Rp^n)$ by \eqref{74}, extends to the
situation where $e^+u$ makes sense, for $KTu$ may be written
$KL^*e^+u$ then, and here $KL^*=r^+G_1^*$ according to $2^\circ$.

\medskip

Generally the \sgos have the following continuity properties:
 
\begin{thm} \label{sgo-thm} 
A \sgo $G$ of order $d\in\R$ and class $r\in\Z$ is continuous 
 \begin{alignat}{2}
 G&\colon B^{s}_{p,q}(\Rp^n)\to B^{s-d}_{p,q}(\Rp^n),
 &\quad\text{ for}\quad &(s,p,q)\in \Dm_r, 
 \label{82} \\
 G&\colon F^{s}_{p,q}(\Rp^n)\to F^{s-d}_{p,o}(\Rp^n),
 &\quad\text{ for}\quad &(s,p,q)\in \Dm_r,\quad o\in\left]0,\infty\right],
 \label{83}
 \end{alignat}
when in \eqref{83} also $p<\infty$ holds.

Moreover, if $G$ is continuous from either $B^{s}_{p,q}(\Rp^n)$ or
$F^{s}_{p,q}(\Rp^n)$ to $\cal D'(\Rn_+)$ for some
$(s,p,q)\notin\overline{\Dm}_r$, then the class of $G$ is  $\le r-1$.

$G$ is continuous from $\cal S'(\Rp^n)$ to $\cal S'(\Rp^n)$
if and only if $G$ has class $-\infty$.
\end{thm}

\begin{proof} Suppose first that $r\ge0$. For each of the terms
$K_j\gamma_j$ in \eqref{70} it is clear from Lemma~\ref{gamm1-lem} and
Theorem \ref{pois-thm} that it is continuous as in \eqref{82} and
\eqref{83}. When $G_0$ in \eqref{70} is written $G_0=\sum K_\nu T_\nu$ as
in Proposition \ref{sgo-prop} $2^\circ$ it is also clear that each
$K_\nu T_\nu$ has the stated continuity properties.

For a given $(s,p,q)\in \Dm_r$ and $u\in B^{s}_{p,q}(\Rp^n)$ there is
for $j$ large an estimate
 \begin{equation}
 (\sum_{m=0}^\infty \norm{K_\nu T_\nu u}{B^{s-d}_{p,q}}^r)^{\fracci1r}\le
 c(\sum_{m=0}^\infty\norm{\tilde k_\nu}{S^{d-\frac12}_{1,0},j}^r
 \norm{\tilde t_\nu}{S^{-\frac12}_{1,0},j}^r)^{\fracci1r}
 \norm{u}{B^{s}_{p,q}}
 \label{84}
 \end{equation}
when $r=\min(1,p,q)$, according to the Corollaries~\ref{pois-cor} and 
\ref{trac-cor}. From the estimates of $\tilde k_\nu$ and $\tilde t_\nu$
in the proof of $2^\circ$ in Proposition~\ref{sgo-prop} above it follows
that the sum on the right hand side is finite. Hence
 $\sum K_\nu T_\nu$ converges strongly to an operator in $\Bbb
L(B^{s}_{p,q}, B^{s-d}_{p,q})$; this operator is $G_0$, since 
$\sum K_\nu T_\nu u$ converges in $\cal D'(\Rn_+)$ to $G_0u$. For similar
reasons $G_0$ is also in $\Bbb L(F^{s}_{p,q},F^{s-d}_{p,o})$.

When $r=-m<0$ and in any case when $(s,p,q)\in
\Dm_{r_1}\!\setminus\overline{\Dm}_r$ is given one can simply carry over
the proof of Theorem \ref{trac-thm}. E.g., in \eqref{69} one can replace
$\psi$ by $e^+\varphi\in e^+C^\infty_0(\Rn_+)$ and $S_{r-1}$ by
$K_{r-1}$; then the denseness of $e^+C^\infty_0(\Rn_+)\subset \cal
S'_0(\Rp^n)$ shows that $K^*_{r-1}=0$.

In the case $G$ has class $-\infty$, one can let $P=0$ in
Section~\ref{PG-ssect} below.
\end{proof}

The proof above of the properties of $G$ in the $B^{s}_{p,q}$ and
$F^{s}_{p,q}$ scales was inspired by the one in \cite{G3}. There it was shown 
that $B^{s}_{p,p}\cup F^{s}_{p,2}$ is mapped into
$B^{s-d}_{p,p}\cap F^{s-d}_{p,2}$ (when $1<p<\infty$) by a s.g.o.; 
this also follows from \eqref{83} and the fact that $B^{s}_{p,p}=
F^{s}_{p,p}$, but the property does not hold for the full scales
here. For the trace and Poisson operators similar remarks can be made.

Concerning the adjoints of \sgos of class 0 the situation is analogous
to the one for trace operators, cf.\ Remark~\ref{adj-rem}.

\begin{rem} \label{PG-rem}
The extension of class $-\infty$ operators $T$ and $G$ to $\cal
S'(\Rp^n)$ is a natural consequence of the formulae $T=K^*e^+$ and
$G=r^+G_1^*e^+$, cf.\ Definitions~\ref{trac-defn} and \ref{sgo-defn}.
Moreover, it was shown in \cite[Cor.\ 4.3]{GK} that the restrictions of 
operators $T$ and $G$ (of any class) to $r^+\cal S_0(\Rp^n)$ have 
continuous extensions to the spaces $B^{s}_{p,q;0}(\Rp^n)$ 
and $F^{s}_{p,q;0}(\Rp^n)$ for any $s\in\R$ (when $q=p\in\,]1,\infty[$
resp.\ $q=2$ and $1<p<\infty$). Consistent with the present level of
pedantry, the restrictions to $r^+\cal S_0(\Rp^n)$
extend (factor) {\em through\/} $e^+$ to bounded operators, that evidently 
are equal to $K^*$ and $r^+G_1^*$, respectively, for formally $T$ and $G$
cannot act on (subspaces of) $\cal S'_0(\Rp^n)$.
\end{rem}

\subsection{Operators $P_{+}+G$ of negative class} \label{PG-ssect}
Because of the properties \eqref{30} and \eqref{31} a truncated
ps.d.o.\ $P_{+}$ is in general of class 0, but if it is
differential the class is said to be $-\infty$ since it has the
mentioned properties for any $s$.

However, a sum $P_{+}+G$ may be continuous also for 
$(s,p,q)\notin\overline{\Dm}_0$, simply because two (or more) 
contributions cancel each other. A non-trivial example is given in 
\cite[Ex. 3.15]{G3}. Results on these phenomena are included here\,---\,%
the underlying analysis is that of \cite{G3}, where such operators 
$P_{+}+G$, that are said to be of negative class, were studied first.

A criterion for $P_++G$ to be of class $-m$, for $m\in\N$, is that 
$(P_++G)D^j_{x_n}=(PD^j_{x_n})_++G^{(j)}$ holds for some \sgo $G^{(j)}$ 
of class 0 for each $j\in\{\,1,\dots,m\,\}$. (See \cite{G3} for the
general formula for $(P_++G)D^j_{x_n}$.) This is equivalent to the
fulfilment of $K^{(j)}_P+K^{(j)}_G=0$ for $j\in\{\,1,\dots,m\,\}$, when
$K^{(j)}_G$ refer to Proposition \ref{sgo-prop} and 
 \begin{equation}
 K^{(k)}_Pv=r^+iPD^k_{x_n}(v\otimes\delta_0)\quad\text{ for}\quad
 v\in\cal S(\R^{n-1});
 \label{PG1}
 \end{equation}
the reader can refer to \cite[(3.43)ff.]{G2} for this. In addition
$P_++G$ is said to be of class $r\in\N_0$ when $G$ is so,
and to have class $-\infty$ when it is of class $r$ for each $r\in\Z$.

\begin{thm} \label{PG-thm}
Let $P_{+}$ be a truncated pseudo-differential operator and 
$G$ a s.g.o., both of order $d\in\R$, and suppose that
$P_++G$ is of class $r\in\Z$. Then 
 \begin{alignat}{2}
 P_++G&\colon B^{s}_{p,q}(\Rp^n)&\to B^{s-d}_{p,q}(\Rp^n)
 &\quad\text{ for}\quad (s,p,q)\in \Dm_r
 \label{PG4} \\
 P_++G&\colon F^{s}_{p,q}(\Rp^n)&\to F^{s-d}_{p,q}(\Rp^n)
 &\quad\text{ for}\quad (s,p,q)\in \Dm_r
 \label{PG5} 
 \end{alignat}
are continuous operators (when $p<\infty$ in \eqref{PG5}).

Moreover, if $P_++G$ is continuous from either $B^{s}_{p,q}(\Rp^n)$ or
$F^{s}_{p,q}(\Rp^n)$ to $\cal D'(\Rn_+)$ for some 
$(s,p,q)\notin\overline{\Dm}_r$, then $P_++G$ is of class $\le r-1$;
and $P_++G$ is continuous $\cal S'(\Rp^n)\to \cal S'(\Rp^n)$ if and
only if it has class $-\infty$.
\end{thm}
\begin{proof} The continuity follows from Theorems~\ref{psdo-thm} and
\ref{sgo-thm}; for $r<0$ as in \eqref{68} ff. In case
$(s,p,q)\notin\overline{\Dm}_r$ the techniques used for trace and 
\sgos can be adapted, cf.~\cite{G2}. When $P_++G$ has
class $-\infty$, it is used that 
\begin{equation}
  \dual{(P_++G)u}{\overline{e^+w}}=
  \int_{\Rn_+}u\cdot(\overline{r^+P_1e^+w+G_1w})
  \label{PG6}
\end{equation}
when $u\in C^\infty(\Rp^n)$, $w\in\cap_{m>0}\cal S_m(\Rp^n)$ while
$P_1=\op{OP}(e^{iD_x\cdot D_\xi}\overline{p})$ and $G=r^+G_1^*e^+$.
Given that
$P_{1+}+G_1\colon \cal S(\Rp^n)\to\cal S(\Rp^n)$ maps $\cap\cal S_m$
into itself, then $e^+(r^+P_1+G_1r^+)$ is continuous $\cal
S_0(\Rp^n)\to \cal S_0(\Rp^n)$, since $\cal S_0=e^+(\cap\cal S_m)$,
hence $P_++G$ is contained in $(e^+(r^+\!P_1+G_1r^+))^*\colon \cal
S'(\Rp^n)\to \cal S'(\Rp^n)$ by \eqref{PG6}.

Thus it remains to see that $\gamma_k(P_{1+}+G_1)w=0$ for each
$w\in\cap\cal S_m(\Rp^n)$. By \eqref{-23}, $\gamma_0D^k_{x_n}G_1w$
equals $\op{OPT}(D^k_{x_n}\tilde g_1(x',0,y_n,\xi'))w$, that is
$iK^{(k)*}_Ge^+w$, and 
\begin{equation*}
  \dual{v}{\overline{\gamma_kP_{1+}w}}=
  \dual{PD^k_n(v\otimes\delta_0)}{\overline{e^+w}} =
  \dual{v}{\overline{iK^{(k)*}_Pe^+w}}
\end{equation*}
for $v\in\cal S'(\R^{n-1})$, by \eqref{PG1}.
Therefore, $\gamma_k(P_{1+}+G_1)=i(K^{(k)}_P+K^{(k)}_G)^*e^+\equiv 0$
on\linebreak $\cap\cal S_m(\Rp^n)$ when $P_++G$ has class $-\infty$.
\end{proof}

\section{Green operators} \label{grn-sect}
In full generality the results for elliptic operators on manifolds
shall now be presented. The main goal is to obtain Theorem~\ref{grn-thm} 
below, that is a generalisation of \cite[Corollary 5.5]{G3}.
Since there are not any substantial changes from the usual texts on the
calculus, a brief explanation will suffice.

\bigskip

To begin with it should be made clear that only {\em bounded\/}, open
$C^\infty$ smooth sets $\Omega\subset\Rn$ will be considered in this 
section. And for operators of class $k\in\Z$ only spaces $B^{s}_{p,q}$ and 
$F^{s}_{p,q}$ satisfying $(s,p,q)\in \Dm_k$ are treated,
cf.\ Section~\ref{gamm-ssect} and Figure~\ref{D0-fig}.
The difficulties connected to the unbounded manifolds and to the borderline 
cases $s=k+\max(\fracp-1,\fracc np-n)$ mentioned in Remark~\ref{trac-rem} are
thus left open here, with the convenience of doing so illuminated by

\begin{rem} \label{clas-rem}
In this paper, none of the spaces are larger than those 
considered in \cite{G3}: Even for $(s,p,q)\in \Dm_k$ with $p\le1$
or $p=\infty$ there is an embedding $B^{s}_{p,q}\hookrightarrow 
B^{s'}_{p',p'}$ with $(s',p',p')\in \Dm_k$ and $1<p'<\infty$, for a simple 
embedding can be combined with a Sobolev embedding or an embedding 
as in \eqref{1.38}. (These facts do not hold for the borderline cases, and 
\eqref{1.38} cannot be extended to the case of unbounded manifolds.) 
See Figure~\ref{clas-fig}, 
which for $k=0$ also illustrates that one can even embed into spaces 
in $\Dm_k\!\setminus \overline{\Dm}_{k+1}$. 

For convenience the spaces with $1<p,q<\infty$ are referred to as classical
spaces. In other words, only some norms and quasi-norms not included in 
\cite{G3} are introduced, the spaces being subspaces of the classical spaces
in focus there.

In addition it is remarked that, when $\Omega$ is bounded, the uniformly 
estimated operators considered here coincide with the locally 
estimated operators in, e.g., \cite{G3}.
\end{rem}

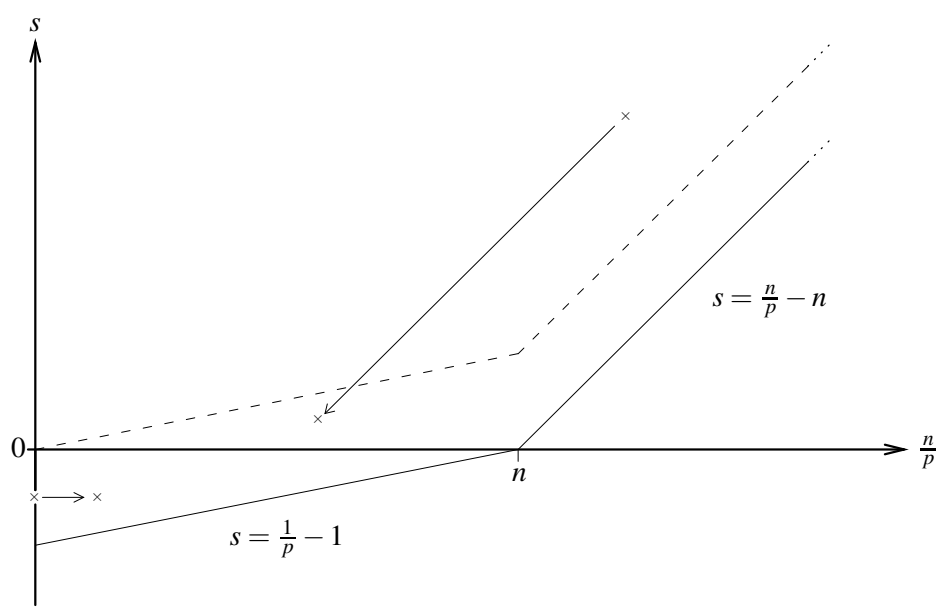
\begin{figure}[htbp]
\hfil
\setlength{\unitlength}{0.0125in}
\begin{picture}(395,261)(0,0)
\path(10,25)(210,65)(330,185)
\dottedline{4}(330,185)(340,195)           
\dashline{4.000}(10,65)(210,105)(330,225)  
\dottedline{4}(330,225)(340,235)           
\path(210,65)(210,60)                      
\path(130,80)(250,200)
\path(131,84)(130,80)(134,81)              
\path(13,45)(30,45)
\path(26,43)(30,45)(26,47)                 
\thicklines
\path(7,65)(370,65)
\path(362.000,63.000)(370.000,65.000)(362.000,67.000)
\path(10,48)(10,235)
\path(10,65)(10,48)
\path(10,42)(10,0)
\path(12.000,227.000)(10.000,235.000)(8.000,227.000)
\put(6.5,43.5){\makebox(0,0)[lb]{\raisebox{0pt}[0pt][0pt]
{\shortstack[l]{{$\scriptscriptstyle\times$}}}}}
\put(32.5,43.5){\makebox(0,0)[lb]{\raisebox{0pt}[0pt][0pt]
{\shortstack[l]{{$\scriptscriptstyle\times$}}}}}
\put(124,76){\makebox(0,0)[lb]{\raisebox{0pt}[0pt][0pt]
{\shortstack[l]{{$\scriptscriptstyle\times$}}}}}
\put(251.5,202.5){\makebox(0,0)[lb]{\raisebox{0pt}[0pt][0pt]
{\shortstack[l]{{$\scriptscriptstyle\times$}}}}}
\put(0,62){\makebox(0,0)[lb]{\raisebox{0pt}[0pt][0pt]{\shortstack[l]{{0}}}}}
\put(90,25){\makebox(0,0)[lb]{\raisebox{0pt}[0pt][0pt]
{\shortstack[l]{{$s=\fracp-1$}}}}}
\put(207,52){\makebox(0,0)[lb]{\raisebox{0pt}[0pt][0pt] 
{\shortstack[l]{{$n$}}}}}
\put(290,125){\makebox(0,0)[lb]{\raisebox{0pt}[0pt][0pt]
{\shortstack[l]{{$s=\fracc np-n$}}}}}
\put(375,62){\makebox(0,0)[lb]{\raisebox{0pt}[0pt][0pt]
{\shortstack[l]{{$\fracc np$}}}}}
\put(7,240){\makebox(0,0)[lb]{\raisebox{0pt}[0pt][0pt]
{\shortstack[l]{{$s$}}}}}
\end{picture}
\hfil
\caption{Embedding into classical spaces in 
 $\protect\Dm_0\!\setminus\overline{\protect\Dm}_1$}
 \label{clas-fig}
\end{figure}

First the operators are generalised to act on sections of vector bundles
$E$ over smooth open bounded subsets $\Omega\subset\Rn$, respectively
on vector bundles $F$ over $\Gamma=\partial\Omega$ (all $C^\infty$ and
hermitian). See for example \cite{GK} where it is shown that one can
do so invariantly . In particular the uniform two-sided transmission
condition and the class concept is invariantly defined on (such)
manifolds. However, to make sense of the transmission condition 
the pseudo-differential operator $P$ should be given on an extending 
bundle $E_1$, that is, a bundle with a boundaryless base manifold 
$\Omega_1\supset \Omega$ for which $E_1\!\bigm|_{\Omega}=E$. 
This will be a tacit assumption on $P$ in the following. For further
explanations of the vector bundle set-up, see \cite{G3} or
\cite[App.~A.5]{G1}. The space of
$C^\infty$ sections of, say, the~bundle $E$ is written $C^\infty(E)$. 
(Since $\Omega$ is an intrinsic part of $E$, 
$C^\infty(E)$ instead of the more tedious $C^\infty(\Omega,E)$
should not course confusion.)

Then, when $P_\Omega$ and $G$ send sections of the same bundle, $E$,
into sections of another bundle $E'$ etc., the Green operator
$\cal A=\bigl(\begin{smallmatrix}P_\Omega+G&K\\ T& S\end{smallmatrix}
\bigr)$ sends $C^\infty(E)\oplus C^\infty(F)$ into 
$C^\infty(E')\oplus C^\infty(F')$. In general the fibre dimensions of
$E$ and $F$ are denoted by $N>0$ and $M\ge0$, and similarly for the 
primed bundles. Then \eqref{i2} is a case with trivial bundles.

The spaces of sections $B^{s}_{p,q}(E)$, \dots, $F^{s}_{p,q}(F')$ 
are defined in the standard way by use of local trivialisations, and it
is verified that Theorems~\ref{pois-thm}, \ref{psdo-thm},
\ref{trac-thm} and \ref{sgo-thm} as well as Theorem~\ref{PG-thm}
remain valid if $\Rp^n$ is replaced by bundles $E$ and $E'$ over
$\Omega$ whereas $\R^{n-1}$ is replaced by bundles $F$ and $F'$ over $\Gamma$. 

Consequently Theorem~\ref{i1} is extended thus: for each $(s,p,q)\in \Dm_r$
there is boundedness of 
 \begin{align}
 \cal A &\colon
  \begin{array}{ccc}
  B^{s}_{p,q}(E)\\ \oplus\\ B^{s-\fracpi}_{p,q}(F) 
  \end{array}
 \to
  \begin{array}{ccc}
  B^{s-d}_{p,q}(E')\\ \oplus \\   B^{s-d-\fracpi}_{p,q}(F')
  \end{array},
 \label{grn2}  
 \\[2.5\jot]
 \cal A &\colon
  \begin{array}{ccc}
  F^{s}_{p,q}(E)\\ \oplus \\ F^{s-\fracpi}_{p,p}(F)
  \end{array} 
 \to
  \begin{array}{ccc}
  F^{s-d}_{p,q}(E') \\ \oplus  \\F^{s-d-\fracpi}_{p,p}(F')
  \end{array}, 
 \quad\text{ if}\quad p<\infty,
 \label{grn3} 
 \end{align}
when each entry in $\cal A$ has order $d\in\R$ and class $r\in\Z$. In
addition, $\cal A$ cannot be continuous from any of the spaces on the
left hand side in \eqref{grn2} and \eqref{grn3} to $\cal D'(E')\times
\cal D'(F')$ when $(s,p,q)\notin
\overline{\Dm}_r$ without the class of each entry being $\le r-1$.

\medskip

Secondly, when $\cal A'=\bigl(\begin{smallmatrix}P'_\Omega+G'&K'\\ T'&
S'\end{smallmatrix}\bigr)$ is a Green operator defined on
$C^\infty(E')\oplus C^\infty(F')$, so that $\cal A'\cal A$ makes sense on 
$C^\infty(E)\oplus C^\infty(F)$, then the composition rules simply express
that this composite $\cal A'\cal A$ is equal to yet another Green
operator $\cal A''$, cf.\ \cite[Cor.\ 5.5]{GK}. The identity 
$\cal A''=\cal A'\cal A$ holds
also when $\cal A'\cal A$ is considered on (the larger) Besov and 
Triebel--Lizorkin spaces of sections. Indeed, by Remark~\ref{clas-rem} 
the composition rules shown for the $B^{s}_{p,p}$ spaces with
$1<p<\infty$ also holds in the spaces treated here.

\bigskip

A main case of interest is the one in which there exists a parametrix,
$\widetilde{\cal A}$, of $\cal A$, that is, another Green operator such
that the operator identities
 \begin{align}
 \widetilde{\cal A}\cal A&=1-\cal R
 \label{grn4}
 \\ 
 \cal A\widetilde{\cal A}&=1-\cal R'
 \label{grn4'}
 \end{align}
hold on the spaces on the left and right hand sides of \eqref{grn2} and
\eqref{grn3}, respectively, for {\em negligible\/} operators 
$\cal R$ and $\cal R'$. This means that they have order $-\infty$,
so necessarily $\cal R$ has range in $C^\infty(E)\oplus C^\infty(F)$ and
$\cal R'$ in $C^\infty(E')\oplus C^\infty(F')$.

When the order of $\cal A$ is an integer $d\in\Z$ and each entry in
$\cal A$ is {\em poly-homogeneous\/} (explained in 
\cite[Sect.~1.2]{G1}, e.g.), there is a well-known ellipticity condition 
assuring the existence of $\widetilde{\cal A}$. If the principal symbols 
are denoted by $p^0(x,\xi)$, $g^0(x,\eta_n,\xi)$ etc., {\em ellipticity\/} 
means that the following two conditions (which are expressed in local
coordinates) are fulfilled:
 \begin{itemize}
  \item[(I)] The principal symbol of $P$ is for each $|\xi|\ge 1$ a bijection
  \begin{equation}
  p^0(x,\xi)\colon \C^N\to \C^{N'}.
  \label{grn5}
  \end{equation}
  \item[(II)] The principal boundary symbol operator
  \begin{equation}
  a^0(x',\xi',D_n)=
  \begin{pmatrix}
  p^0(x',0,\xi',D_n)_{+}+g^0(x',\xi,D_n)& k^0(x',\xi',D_n)\\[1\jot]
  t^0(x',\xi',D_n)                      & s^0(x',\xi')
  \end{pmatrix} 
  \label{grn6}
  \end{equation}
 is a bijection
 \begin{equation}
 \cal S(\Rp)^N\times\C^M @>\quad a^0(D_n)\quad>> \cal S(\Rp)^{N'}\times\C^{M'}
 \label{grn7}
 \end{equation}
 for each $x'\in \Gamma$ and each $|\xi'|\ge 1$.
 \end{itemize}
It was shown in \cite[Thm.~5.4]{G3} that if $\cal A$ is elliptic, then there
exists a parametrix $\widetilde{\cal A}$ of order $-d$ and class $r-d$. 
In this case \eqref{grn4} becomes an operator identity valid on the spaces
$B^{s}_{p,q}(E)\oplus B^{s-\fracpi}_{p,q}(F)$ and 
$F^{s}_{p,q}(E)\oplus F^{s-\fracpi}_{p,p}(F)$, and \eqref{grn4'} holds on
$B^{s-d}_{p,q}(E')\oplus B^{s-d-\fracpi}_{p,q}(F')$ and 
$F^{s-d}_{p,q}(E')\oplus F^{s-d-\fracpi}_{p,p}(F')$\,---\,in both cases for 
each $(s,p,q)\in \Dm_r$. Observe that $\cal R$ is then necessarily of 
class $r$ while the class of $\cal R'$ must be $r-d$.

{\em Injective\/} and {\em surjective\/} ellipticity of $\cal A$ means 
that (I) and (II) above hold only with `bijection' replaced by, respectively,
`injection' and `surjection'. In the affirmative case there exists an 
$\widetilde{\cal A}$ satisfying \eqref{grn4} and \eqref{grn4'}, respectively,
and it is termed a left respectively a right parametrix. 

\subsection{Fredholm properties} \label{frdh-ssect}
Already when $\cal A$ is either injectively or surjectively elliptic
one can deduce various properties for its kernel and range. 
Instead of generalising the Fredholm theory to the category of 
quasi-Banach spaces, one can proceed as in \cite{G3}. Basically
this is possible because, as seen in Remark~\ref{clas-rem}, the spaces
considered here are contained in the Banach spaces treated there.  

This will be explained in the following, where a version of
Theorem~\ref{i2-thm} for vector bundles will be proved. First
it will be convenient to introduce the
vector bundles $V=E\oplus F$ and $V'=E'\oplus F'$ and use them to
borrow the spaces $B^{s+\bold{a}}_{p,q}(V)$ and
$B^{s-\bold{b}}_{p,q}(V')$, respectively $F^{s+\bold{a}}_{p,q}(V)$ and
$F^{s-\bold{b}}_{p,q}(V')$, from \eqref{100} ff.\ below, where they are
introduced systematically. The vectors $\bold{a}$ and $\bold{b}$ 
indicate that there is a space for each column and row in $\cal A$;
in the present case they are equal to zero.

The injectively elliptic case is quite simple: for each $(s,p,q)\in
\Dm_r$ it is seen from the embedding relations and \eqref{grn4} that
 \begin{equation}
 \{\,u\in B^{s+\bold{a}}_{p,q}(V)\mid \cal Au=0\,\}=
 \{\,u\in C^\infty(V)\mid \cal Au=0\,\}.
 \label{grn8}
 \end{equation}
A similar argument works in the Triebel--Lizorkin case, and thus,
since $F^{s}_{2,2}=B^{s}_{2,2}$, the kernel of $\cal A$, written
$\ker\cal A$, is independent of $(s,p,q)\in \Dm_r$ as well as of whether
we consider \eqref{grn2} or \eqref{grn3}. Hence $\ker\cal A$ equals
the space in \eqref{grn8}\,---\,and it has finite dimension by \cite{G3}. 
Moreover, the~image $\cal A(B^{s+\bold{a}}_{p,q}(V))$ is closed
in $B^{s-d-\bold{b}}_{p,q}(V')$ and similarly
for the operator in \eqref{grn3}. The latter fact was proved 
in \cite{G3} for the spaces in consideration there, and if $\cal
Au_m\to v$ in $B^{s-d-\bold{b}}_{p,q}(V')$ we
determine $(s',p',p')$ as in Remark~\ref{clas-rem} and conclude from
\cite{G3} that
$v=\cal A w$ for some $w\in B^{s'+\bold{a}}_{p',p'}(V)$; 
here $w=\widetilde{\cal A}v+\cal Rw$ according to \eqref{grn4}, so
$w\in B^{s+\bold{a}}_{p,q}(V)$. In a similar way
the~analogous statement for the Triebel--Lizorkin spaces carry over
from \cite{G3}. Hereby $1^\circ$ of Theorem~\ref{i2-thm} is proved.

As a preparation for the surjectively elliptic case we shall first
treat the case where $\cal A$ is elliptic; evidently the arguments
above for the injectively elliptic case apply to $\cal A$ then.
Concerning the range of $\cal A$ we use an embedding of $B^{s}_{p,q}$
into a classical space $B^{s'}_{p',p'}$. For the classical spaces it was
shown in \cite{G3} that there exists a finite dimensional subspace
$\cal N\subset C^\infty(V')$ which is a complement of $\im\cal A$, that is,
 \begin{equation}
 B^{s'-d-\bold{b}}_{p',p'}(V')=
 \cal N\oplus\cal A(B^{s'+\bold{a}}_{p',p'}(V))
 \label{grn20}
 \end{equation}
for every $(s',p',p')\in\Dm_r$ with $1<p'<\infty$. This implies that 
 \begin{equation}
 B^{s-d-\bold{b}}_{p,q}(V')=
 \cal N\oplus\cal A(B^{s+\bold{a}}_{p,q}(V))
 \label{grn21}
 \end{equation}
for when \eqref{grn20} is applied to an element of the subspace
$B^{s-d-\bold{b}}_{p,q}(V')$ it follows from \eqref{grn4} that the component 
in the range of $\cal A$ belongs to $\cal A(B^{s+\bold{a}}_{p,q}(V))$. That 
the sum is direct is seen already from \eqref{grn20}.
From \eqref{grn21} we conclude that $\cal N$ is a complement of $\im\cal A$
also in the non-classical cases, and by the construction it is
independent of $(s,p,q)$ and of finite dimension. The $F$ case is
covered by a similar argument.

When $\cal A$ is surjectively elliptic the study of {\em ranges\/} of
$\cal A$ that is found in \cite[(5.21) ff.]{G3} is easily modified,
and we sketch this in the $B$ case when $d=r=0$; the $F$ case is
completely analogous.
The tools in \cite{G3} consist of some remarks on the Banach space cases
with $(s,p,q)$ in $ \Dm_0\!\setminus \overline{\Dm}_1$, and these need not be
changed at all. So recall from \cite{G3} that 
 \begin{equation}
 \cal A(B^{s'+\bold{a}}_{p',q'})=\bigl\{\,f\in B^{s'-\bold{b}}_{p',q'}(V')
  \bigm|\dual{f}{\overline{g}}=0\quad\text{for }g\in\ker\cal A^* \,\bigr\} 
 \label{grn9}
 \end{equation}
when $(s',p',q')$ is a parameter in $\Dm_0\!\setminus\overline{\Dm}_1$ with 
$1<p',q'<\infty$. In addition there is an argument which by 
Remark~\ref{clas-rem} easily gives that 
$\rho_0\le \nu(s,p,q)\le \rho_1$ for 
{\em any\/} $(s,p,q)\in \Dm_0$; here $\rho_0=\op{dim}\ker\cal A^*$,
$\rho_1=\op{codim}\cal A\cal A^*(B^{s+\bold{a}}_{p,q})$ and 
$\nu(s,p,q)$ denotes the codimension of
$\cal A(B^{s+\bold{a}}_{p,q})$ in $B^{s-\bold{b}}_{p,q}(V')$. 
In virtue of the injectively and two-sided elliptic cases treated
above the numbers $\rho_0$ and $\rho_1$ are $(s,p,q)$-independent. 
Consequently the conclusion from \cite{G3} that $\rho_0=\rho_1$
yields the independence of $\nu$ from $(s,p,q)$. 

Now we take $g_1$,\dots,$g_\nu$ in $C^\infty(V')$ as a basis for 
$\ker\cal A^*$, and we may assume that $\dual{g_j}{\overline{g_k}}=\kd{j=k}$.
From \eqref{grn9} it is seen that $\cal A(B^{s+\bold{a}}_{p,q})$ and
$\ker\cal A^*$ are linearly independent, and for this reason
$g_1$,\dots,$g_\nu$ are linearly independent in the quotient
$B^{s-\bold{b}}_{p,q}/\cal A(B^{s+\bold{a}}_{p,q})$. Hence
$(g_1,\dots,g_\nu)$ is a basis for the quotient, so
$\ker\cal A^*+\cal A(B^{s+\bold{a}}_{p,q})$ is equal to
$B^{s-\bold{b}}_{p,q}$. Altogether this shows that 
 \begin{equation}
 B^{s-\bold{b}}_{p,q}(V')=\ker\cal A^*\oplus\cal A(B^{s+\bold{a}}_{p,q}),
 \label{grn9'}
 \end{equation}
that is, $\cal A(B^{s+\bold{a}}_{p,q})$ has the finite dimensional
$(s,p,q)$-independent complement $\ker\cal A^*$. When $\cal Au_m\to v$
in $B^{s-\bold{b}}_{p,q}$ we write $v=\cal
Aw+\lambda_1g_1+\dots+\lambda_\nu g_\nu$ by use of \eqref{grn9'}. Then
\eqref{grn9} gives that $0=\dual{v}{\overline{g_j}}=\lambda_j$, so that $v=\cal
Aw$. Hence the range of $\cal A$ is closed.  

More generally one can reduce to the case where $d=r=0$, see
\cite{G3}. This reduction uses order-reducing operators, written as
$\Lambda^m_{-,E}$ for $m\in\Z$, that are chosen so that they for all
admissible parameters $(s,p,q)$ have 
the~following group and continuity properties:
 \begin{gather}
 \Lambda^k_{-,E}\Lambda^m_{-,E}=\Lambda^{k+m}_{-,E},\qquad 
 \Lambda^0_{-,E}=1,
 \label{grn10} \\
 \Lambda^m_{-,E}\colon B^{s}_{p,q}(E)\overset{\sim}{\to}
 B^{s-m}_{p,q}(E),\qquad
 \Lambda^m_{-,E}\colon F^{s}_{p,q}(E)\overset{\sim}{\to}
 F^{s-m}_{p,q}(E).
 \label{grn12}
 \end{gather}   
Such operators were constructed in \cite[Thm. 5.1 $3^\circ$]{G3} (but
called $\Xi^{m}_{-,E}$ there, see also \cite[Ex.\ 2.10]{G2} for 
a brief review) and in \cite{F2}.
Their continuity properties are a consequence of Section~\ref{cont-sect}, 
since they in general are of the form $P^{(m)}_\Omega+G^{(m)}$, and the
group property, valid by the earlier remark on composition rules, implies
the bijectivity. 

One should observe that in the reduction procedure,
\eqref{grn9'} easily carries over to a  similar statement for the more
general surjectively elliptic Green operators, except that $\ker\cal
A^*$ is replaced by another fixed finite dimensional space of smooth 
sections. Altogether this proves $2^\circ$ of Theorem~\ref{i2-thm}.

To show $3^\circ$ there, suppose that (the vector bundle version of)
\eqref{i11} holds for some $(s_1,p_1,q_1)\in\Dm_r$ for a subspace $\cal
N\subset C^\infty(V')$. By $2^\circ$ there is a parameter independent range
complement $\cal M\subset C^\infty(V')$. 
If $(g_1,\dots,g_k)$ is a linearly independent tupel in
$\cal N$, its image is so in $B^{s_1-d-\bold{b}}_{p_1,q_1}(V')/\cal
A(B^{s_1+\bold{a}}_{p_1,q_1})$ by \eqref{i11}. Hence $\op{dim}\cal N
\le\op{dim}\cal M$. In addition the quotient is isomorphic as a vector 
space to both $\cal N$ and $\cal M$, so there is equality.

For an arbitrary $(s,p,q)\in\Dm_r$ the identity \eqref{i11} now holds if and
only if a basis $(g_1,\dots,g_k)$ for $\cal N$ still gives a basis in 
$B^{s-d-\bold{b}}_{p,q}(V')/\cal A(B^{s+\bold{a}}_{p,q})$. Because
$k=\op{dim}\cal M$ was seen above, it suffices to see that
$(Qg_1,\dots,Qg_k)$ is linearly independent, when $Q$ denotes the quotient
operator. Let $0=\lambda_1Qg_1+\dots\lambda_kQg_k$. Then 
 \begin{equation}
 \lambda_1g_1+\dots+\lambda_kg_k=\cal Au
 \label{grn50}
 \end{equation}
for a unique $\cal Au$ in $B^{s-d-\bold{b}}_{p,q}(V')$. But 
 \begin{equation}
 \lambda_1g_1+\dots+\lambda_kg_k=w+\cal Av
 \label{grn51} 
 \end{equation}
for uniquely determined $w\in\cal M$ and $\cal Av\in
B^{s_1-d-\bold{b}}_{p_1,q_1}(V')$. It suffices now to verify that $w=0$, for
then \eqref{grn51} implies that
 \begin{equation}
 0=\lambda_1Q_1g_1+\dots\lambda_kQ_1g_k\quad\text{in}\quad
 B^{s_1-d-\bold{b}}_{p_1,q_1}(V')/\cal A(B^{s_1+\bold{a}}_{p_1,q_1}),
 \label{grn53}
 \end{equation}
when $Q_1$ denotes the quotient operator for $(s_1,p_1,q_1)$. That $w=0$ is
evident when $B^{s_1-d-\bold{b}}_{p_1,q_1}\hookrightarrow 
B^{s-d-\bold{b}}_{p,q}$, for then \eqref{grn51} is also a decomposition in
$B^{s-d-\bold{b}}_{p,q}$. Similarly $w=0$ holds when the reverse embedding
does so. Hereby \eqref{i11} has been established for $(s_1+1,\infty,1)$, so
by repeating the argument any $(s,p,q)\in\Dm_r$ with $s<s_1$ is covered, and
then in a last application {\em any\/} $(s,p,q)\in\Dm_r$ is so. Thus
$3^\circ$ is proved.

In Corollary~\ref{i2-cor} any $(f,\varphi)$ in
$B^{s_1-d-\bold{b}}_{p_1,q_1}(V')$ or
$F^{s_1-d-\bold{b}}_{p_1,q_1}(V')$ gives a functional
$\dual{f}{\cdot}+\dual{\varphi}{\cdot}$ on $\cal N$: 
one can take a larger classical space $B^{s'-d-\bold{b}}_{p',q'}(V')$
with $(s',p',q')\in \Dm_r\!\setminus\overline{\Dm}_{r+1}$,
cf.~Remark~\ref{clas-rem}; with $s=-s'+d$,
$\fracp+\fracc1{p'}=1$ and $\fracc1q+\fracc1{q'}=1$, this space is
dual to $B^{s}_{p,q;0}(E)\oplus B^{s+1-\fracpi}_{p,q}(F)$ and
$e_\Omega g\in B^{s}_{p,q;0}(E)$, for by construction $(s,p,q)\in
\Dm_{d-r}\!\setminus\overline{\Dm}_{d-r+1}$ so it suffices with 
$\gamma_j g=0$ for $j<d-r$ as assumed. Moreover, any 
$\varphi\in\cal D'(F')$ gives a functional on $C^\infty(F')$. 

Using this it is elementary to verify that validity of \eqref{i16} or
\eqref{i17} at $(s_1,p_1,q_1)$ implies \eqref{i11} and \eqref{i12},
respectively. For every $(s,p,q)$ in $\Dm_r$, $3^\circ$ then gives that $\cal
N$ is a range complement, and the inclusions into $\cal N^\perp$ in
\eqref{i16} and \eqref{i17} are clear since $\cal A(C^\infty(V))\subset \cal
N^\perp$ is seen by consideration of $(s_1,p_1,q_1)$. The other inclusions
follow from the ones proved, for when an element in, say,
$B^{s-d-\bold{b}}_{p,q}(V')\cap \cal N^\perp$ is decomposed as in $2^\circ$,
then the $\cal N$ component is trivial.

\subsection{Multi-order operators}  \label{mo-ssect}
Using order-reducing operators one can also reduce multi-order
Green operators to the case of order and class 0 and carry through the
preceding considerations, cf.\ \cite{G3}. Instead of going into details 
this section is concluded with a precise summary, which contains the 
previous results as well as those in Section~\ref{summ-sect} on single 
order operators as special cases.

\bigskip

In the following, $\Omega\subset\Rn$ is a
smooth open bounded set with $\partial\Omega=\Gamma$,
and $\cal A$ denotes a multiorder Green operator, i.e., 
$\cal A=\bigl(\begin{smallmatrix}P_\Omega+G&K\\ T&
S\end{smallmatrix}\bigr)$, where $P=(P_{ij})$ and $G=(G_{ij})$,
$K=(K_{ij})$, $T=(T_{ij})$ and $S=(S_{ij})$. Here $1\le i\le i_1$ and
$i_1< i\le i_2$, respectively, in the two rows of the block matrix
$\cal A$, and $1\le j\le j_1$ respectively $j_1< j\le j_2$ holds in 
the~two columns of $\cal A$; that is, $\cal A$ is an $i_2\times j_2$
matrix operator. 

Each $P_{ij}$, $G_{ij}$, $K_{ij}$, $T_{ij}$ 
respectively $S_{ij}$ belongs to the poly-homogeneous calculus, 
i.e., they are a pseudo-differential operator satisfying the~uniform
two-sided transmission condition (at $\Gamma$), a singular Green, a Poisson and
a trace operator, resp.\  a pseudo-differential operator 
on $\Gamma$. The orders of the operators are taken to be
$d+b_i+a_j$, where $d\in\Z$, $\bold{a}=(a_j)\in\Z^{j_2}$ 
and $\bold{b}=(b_i)\in\Z^{i_2}$, and the class of $P_{ij,\Omega}+G_{ij}$ 
and of $T_{ij}$ is supposed to be $r+a_j$ for some $r\in\Z$.%
\footnote{For short $\cal A$ is then said to have order $d$ and class $r$.}

The operators are supposed to act on sections of
vector bundles $E_j$ over $\Omega$ and $F_j$ over $\Gamma$, with values in
other bundles $E'_i$ and $F'_i$. Letting $V=(E_1\oplus\dots\oplus
E_{j_1})\cup (F_{j_1+1}\oplus\dots\oplus F_{j_2})$, while 
$V'=(E'_1\oplus\dots\oplus E'_{i_1})\cup
(F'_{i_1+1}\oplus\dots\oplus F'_{i_2})$, the Green operator $\cal A$
sends $C^\infty(V)$ to $C^\infty(V')$. Here one can either regard
$C^\infty(V)$ as an abbreviation for $C^\infty(E_1)\oplus\dots\oplus
C^\infty(F_{j_2})$, or verify that $V$ is a vector
bundle with base manifold $\Omega\cup\Gamma$, 
cf. the~definition in \cite{Lan}. Observe that hereby the
dimension of the base manifold, as well as that of the fibres over its
points $x$, depends on whether $x\in\Omega$ or $x\in\Gamma$. Similar
remarks apply to $V'$. 

To have a convenient notation we shall now introduce spaces that are
adapted to the order and class of each entry in $\cal A$, namely (with
$p<\infty$ in the Triebel--Lizorkin spaces)
 \begin{align}
 B^{s+\bold{a}}_{p,q}(V)&=
 (\bigoplus_{j\le j_1} B^{s+a_j}_{p,q}(E_j) )
 \oplus (\bigoplus_{j_1<j}  B^{s+a_j-\fracpi}_{p,q}(F_j) )
 \label{100} \\
 B^{s-\bold{b}}_{p,q}(V')&=( \bigoplus_{i\le i_1}  B^{s-b_i}_{p,q}(E'_i))
 \oplus ( \bigoplus_{i_1<i}  B^{s-b_i-\fracpi}_{p,q}(F'_i) ),
 \label{101}
 \\
 F^{s+\bold{a}}_{p,q}(V)&=
 (\bigoplus_{j\le j_1} F^{s+a_j}_{p,q}(E_j) )
 \oplus (\bigoplus_{j_1<j}  F^{s+a_j-\fracpi}_{p,p}(F_j) )
 \label{102} 
 \\
 F^{s-\bold{b}}_{p,q}(V')&=( \bigoplus_{i\le i_1}  F^{s-b_i}_{p,q}(E'_i))
 \oplus ( \bigoplus_{i_1<i}  F^{s-b_i-\fracpi}_{p,p}(F'_i) ).
 \label{103}
 \end{align}
(As usual $F^{s}_{p,p}(F_j)=B^{s}_{p,p}(F_j)$ etc.)
It is convenient to take the quasi-norms to be
$ \norm{v}{B^{s+\bold{a}}_{p,q}}=\big( 
 \norm{v_1}{B^{s+a_1}_{p,q}(E_1)}^q+\dots+
 \norm{v_{j_2}}{B^{s+a_{j_2}-\fracpi}_{p,q}(F_{j_2})}^q\big)^{\fracci1q}$,
and
$ \norm{v}{F^{s+\bold{a}}_{p,q}}=\big( 
 \norm{v_1}{F^{s+a_1}_{p,q}(E_1)}^p+\dots+
 \norm{v_{j_2}}{F^{s+a_{j_2}-\fracpi}_{p,q}(E_{j_2})}^p\big)^{\fracpi}$,
with similar conventions for $B^{s-\bold{b}}_{p,q}$ and $F^{s-\bold{b}}_{p,q}$.

{\em The ellipticity concept\/} for multi-order Green operators is like
the one for single-order operators, except that $p^0(x,\xi)$ is a
matrix with $p^0_{ij}$ equal to the principal symbol of $P_{ij}$ {\em
relative\/} to the order $d+b_i+a_j$ of $P_{ij}$; and similarly for
$a^0(x',\xi',D_n)$.

\begin{thm} \label{grn-thm}
Let $\cal A$ denote a multi-order Green operator going from $V$ to $V'$
as described above. Then $\cal A$ is {\em continuous\/}
 \begin{equation}
 \cal A\colon B^{s+\bold{a}}_{p,q}(V)\to B^{s-d-\bold{b}}_{p,q}(V'),
 \quad \cal A\colon F^{s+\bold{a}}_{p,q}(V)\to F^{s-d-\bold{b}}_{p,q}(V'),
 \label{104}
 \end{equation}
for each $(s,p,q)\in\Dm_r$, when $p<\infty$ in the Triebel--Lizorkin spaces.

If $\cal A$ is injectively elliptic, surjectively elliptic, respectively
elliptic, then $\cal A$ has a left, right respectively two-sided 
{\em parametrix\/} $\widetilde{\cal A}$ in the calculus; it can be taken of
order $-d$ and class $r-d$, and then 
$\widetilde{\cal A}$ is continuous in the opposite
direction in \eqref{104} for all the parameters $(s,p,q)$ mentioned above.

When $\cal A$ is continuous $B^{s+\bold{a}}_{p,q}(V)\to\cal
D'(V')$ or $\widetilde{\cal A}$ is so from
$B^{s-d-\bold{b}}_{p,q}(V')$ to $\cal D'(V)$ for some
$(s,p,q)\notin\overline{\Dm}_r$, then the {\em class\/} of
$\cal A$ is $\le r-1$ and $\widetilde{\cal A}$ has class $\le
r-1-d$, respectively. A similar conclusion holds for $F^{s+\bold{a}}_{p,q}(V)$
and $F^{s-d-\bold{b}}_{p,q}(V')$.

Furthermore, when $\cal A$ is injectively elliptic,
the {\em inverse regularity\/} properties in Corollary~\ref{i1-cor} carry
over to the operators in \eqref{104}. Moreover, $1^\circ$ of
Theorem~\ref{i2-thm} is valid mutatis mutandem for $\cal A$, and the ranges
are closed.

When $\cal A$ is surjectively elliptic, analogues of $2^\circ$ and
$3^\circ$ of Theorem~\ref{i2-thm} as well as of Corollary~\ref{i2-cor}
hold for $\cal A$ (when $\gamma_j g_i=0$ for $j<d+b_i-r$).

In the elliptic case, all these properties hold for $\cal A$, and the
parametrices are two-sided.
\end{thm}

On the basis of the single-order case described above,
Theorem~\ref{grn-thm} is obtained by a straightforward extension of
the proof of \cite[Cor.~5.5]{G3}.

It should be observed explicitly, that the $(s,p,q)$-independence of
$\ker \cal A$ implies that it is the~same space regardless of
whether $\cal  A$ is considered on $B^{s+\bold{a}}_{p,q}(V)$ or on
$F^{s+\bold{a}}_{p,q}(V)$; this follows since $B^{s+\bold{a}}_{2,2}=
F^{s+\bold{a}}_{2,2}$. A similar argument shows that one can take a
space $\cal N$ that is a complement of $\cal
A(B^{s+\bold{a}}_{p,q})$ in $B^{s-d-\bold{b}}_{p,q}(V')$ as well as of $\cal
A(F^{s+\bold{a}}_{p,q})$ in $F^{s-d-\bold{b}}_{p,q}(V')$.

Also it should be observed that the inverse regularity properties as in
Corollary~\ref{i1-cor} follow by application of \eqref{grn4} to \eqref{i8}
(or its analogue).

\bigskip

The theorem above is a  generalisation of 
\cite[Cor.\ 5.5]{G3} to the 
scales of $B^{s}_{p,q}$ and $F^{s}_{p,q}$ spaces with $0<p,q\le\infty$
(with $p<\infty$ for the $F$ spaces) with a rather more detailed 
Fredholm theory characterising the ranges of $\cal A$.
Moreover, the H\"older--Zygmund spaces $C^s=B^{s}_{\infty,\infty}$,
$s>0$, are included, and unlike \cite{RS} it is unnecessary to assume
that $s-d>0$ when applying operators of order $d$ to $C^s$: the space
$B^{s-d}_{\infty,\infty}$ can receive in any case.
For further historical remarks see the~beginning and end of \cite{G3}.

\section{Applications} \label{appl-sect}
As a first example it is clear that Theorem~\ref{grn-thm} above
applies to such boundary problems as those in \eqref{i1} above.

Secondly the various orthogonal decompositions into divergence-free
and gradient subspaces in \cite[Ex.~3.14]{G3} are generalised by
Theorem~\ref{grn-thm} to the spaces $B^{s}_{p,q}(\overline{\Omega})^n$ and
$F^{s}_{p,q}(\overline{\Omega})^n$ for $(s,p,q)\in\Dm_0$. 

Thirdly the inverse regularity properties in Corollary~\ref{i1-cor}
carry over to {\em semi-linear\/} 
perturbations of $\cal A$, as long as the non-linear term is `better
behaved' than $\cal A$. This is proved for the stationary
Navier--Stokes equations (considered with boundary conditions of class $1$ and
$2$) in \cite[Thm.~5.5.3]{JJ93}. A paper on this application to more general
semi-linear problems is being worked out; results and methods are
sketched in \cite{JJ95stjm}.

For the stationary Navier--Stokes equations the mentioned inverse
regularity results have led to an extension of the weak $L_2$ solvability
theory for the Dirichl\'et problem, cf.~\cite{Tm}, to existence of
solutions in the
$B^{s}_{p,q}$ and $F^{s}_{p,q}$ spaces, in rough terms when $s\ge
\max(1,\fracc np -\frac n2+1)$. See \cite[Thm.~5.5.5]{JJ93} for details.


\providecommand{\bysame}{\leavevmode\hbox to3em{\hrulefill}\thinspace}
\providecommand{\MR}{\relax\ifhmode\unskip\space\fi MR }
\providecommand{\MRhref}[2]{%
  \href{http://www.ams.org/mathscinet-getitem?mr=#1}{#2}
}
\providecommand{\href}[2]{#2}

\end{document}